\documentclass{amsart}

\usepackage{amsmath}
\usepackage{graphicx}
\usepackage{amsfonts}
\usepackage{amssymb}

\newtheorem{theorem}{Theorem}
\theoremstyle{plain}

\newtheorem{corollary}{Corollary}

\newtheorem{example}{Example}

\newtheorem{proposition}{Proposition}
\newtheorem{remark}{Remark}

\numberwithin{equation}{section}

\begin{document}

\title
{On the bialgebra structure of the free loop homology}

\author{Samson Saneblidze}
\address{Samson Saneblidze, A. Razmadze Mathematical Institute,
I.Javakhishvili Tbilisi State University 2, Merab Aleksidze II Lane,
 Tbilisi 0193, Georgia}
 \email{samson.saneblidze@tsu.ge}

\thanks{This work was supported by Shota Rustaveli
National Science Foundation of Georgia (SRNSFG) [grant number
FR-23-5538]}

\date{}


\begin{abstract}

We introduce  a commutative product of degree $-n$ on the homology
$H_\ast(X)$  of an $n$-dimensional special cubical set $X$
and lift it on the free loop homology $H_\ast(\Lambda M)$
for $M=|X|$ to be the geometric realization.
These products agree  with the  intersection and string topology  products respectively when $M$ is an oriented closed manifold, and  we
 establish the compatibility relation between the string topology
 product and the standard  coproduct on $H_\ast(\Lambda M).$
Motivated by the above relationship we introduce the notion of loop
bialgebra for differential graded coalgebras $C$ by means of  the
coHochschild complex $\Lambda C.$ We calculate  the loop bialgebra structure for some spaces.

\end{abstract}

\maketitle

 \section{Introduction}

 Let $M$ be an  oriented closed triangulated $n$-manifold.
 An initial motivation of the paper was to establish relationship
 between
 the string topology product of degree $-n$  introduced by Chas and Sullivan on   the homology $H_\ast(\Lambda M)$
 of the free loop space $\Lambda M$  \cite{CS}
 \begin{equation}\label{stringM}
 \Cap:H_\ast(\Lambda M)\otimes H_\ast(\Lambda M)\rightarrow
 H_{\ast-n}(\Lambda M)
 \end{equation}
 and the standard coproduct
 \[\Delta:  H_\ast(\Lambda M)\rightarrow H_\ast(\Lambda M)\otimes
 H_\ast(\Lambda M)      \]
 in fact defined for any topological space $Y$ instead of $\Lambda M.$
 In this way we first define the classical intersection product
\begin{equation}\label{stcap}\sqcap:H_p( M)\otimes H_q( M)\rightarrow
H_{p+q-n}( M)\end{equation}
as induced  by a  chain-level pairing of degree $-n$
\[\sqcap: C_p(K^{_\Box})\otimes
 C_q(K^{_\Box})\rightarrow C_{p+q-n}(K^{_\Box}),\]
where $K^{_\Box}$ is a cubical subdivision  of $M$ canonically derived
from a triangulation $K$ of $M.$
Then  without direct using the Poincar\'{e} isomorphism
$H_i(M)\overset{\approx}{\rightarrow} H^{n-i}(M)$
we establish
that $\sqcap$ is a map of $H_\ast(M)$ -- bicomodules.

\begin{theorem}\label{intersection}
The following diagrams

\begin{equation}\label{biintersection1}
\begin{array}{cccccc}
H_\ast( M)\otimes H_\ast( M) \otimes H_\ast( M) &
\xrightarrow{\sqcap\, \otimes 1}                 &
  H_\ast( M)\otimes H_\ast( M)\\
       \hspace{0.15in} \uparrow\,  _{(1\otimes T )\circ( \Delta\otimes
       1 )} &                  &
       \hspace{-0.17in}

          _{\Delta}
 \uparrow \vspace{1mm}\\
 H_\ast( M)\otimes H_\ast( M) &  \xrightarrow{\sqcap} &
   H_\ast( M)
  \end{array}
\end{equation}
and

\begin{equation}\label{biintersection2}
\begin{array}{cccccc}
H_\ast( M)\otimes H_\ast( M) \otimes H_\ast( M) & \xrightarrow{1\otimes\,
\sqcap}                 &
  H_\ast( M)\otimes H_\ast( M)\\
       \hspace{-0.1in}   _{(T\otimes 1 )\circ( 1\otimes \Delta
       )}\uparrow &                  &
       \hspace{-0.17in}

          _{\Delta}
 \uparrow \vspace{1mm}\\
 H_\ast( M)\otimes H_\ast( M) &  \xrightarrow{\sqcap} &
   H_\ast( M)
  \end{array}
\end{equation}
are  commutative.
\end{theorem}

Let $X$   be a cubical set and $Y:=|X|$ its geometric realization.
We construct
a new combinatorial model
 $|\widehat{\mathbf{\Omega}}X| \xrightarrow{\iota}
 |\widehat{\mathbf{\Lambda}} X| \xrightarrow{\zeta} |X|$
of the free loop fibration
 $\Omega Y\rightarrow \Lambda Y\rightarrow Y$ (Theorem \ref{freeloopmodel})
 where   $\widehat{\mathbf{\Omega}}X$  and $\widehat{\mathbf{\Lambda}} X$ are permutahedral sets. Then we introduce a twisted differential in the tensor product of permutahedral chain complexes
  $C^{\diamond}_\ast(\widehat{\mathbf{\Lambda}} X)\otimes
 C^{\diamond}_\ast(\widehat{\mathbf{\Omega}} X) $
 to obtain the chain complex
 $C^{\diamond}_\ast(\widehat{\mathbf{\Lambda}} X)\otimes _\tau
 C^{\diamond}_\ast(\widehat{\mathbf{\Omega}} X) $
such that there are the chain maps
\[\nu_r : C^{\diamond}_\ast(\widehat{\mathbf{\Lambda}} X)   \rightarrow   C^{\diamond}_\ast(\widehat{\mathbf{\Lambda}} X)\otimes _\tau
 C^{\diamond}_\ast(\widehat{\mathbf{\Omega}} X)     \]
and
\[ \mu_r: C^{\diamond}_\ast(\widehat{\mathbf{\Lambda}} X)\otimes _\tau
 C^{\diamond}_\ast(\widehat{\mathbf{\Omega}} X)\rightarrow   C^{\diamond}_\ast(\widehat{\mathbf{\Lambda}} X);   \]
also there is "an extended switch chain map"
\[ \mathcal {T}_r:   C^{\diamond}_\ast(\widehat{\mathbf{\Lambda}} X)   \otimes C^{\diamond}_\ast(\widehat{\mathbf{\Lambda}} X)\otimes _\tau
 C^{\diamond}_\ast(\widehat{\mathbf{\Omega}} X)\rightarrow
  C^{\diamond}_\ast(\widehat{\mathbf{\Lambda}} X)   \otimes C^{\diamond}_\ast(\widehat{\mathbf{\Lambda}} X)\otimes _\tau
 C^{\diamond}_\ast(\widehat{\mathbf{\Omega}} X).  \]
 Note that the above twisting tensor product of chain complexes and what follows is a particular case of a more general algebraic phenomenon involving the coHochschild chain complex of a dg coalgebra (see subsection \ref{twistedsec}).
Denoting $\mathcal{H}_\ast(Y):= H_\ast(  C^{\diamond}_\ast(\widehat{\mathbf{\Lambda}} X)\otimes _\tau
 C^{\diamond}_\ast(\widehat{\mathbf{\Omega}} X) )$ we get the following maps in homology
 \[ H_\ast(\Lambda Y)\xrightarrow{\nu_r}  \mathcal{H}_\ast(Y)  \xrightarrow{\mu_r} H_\ast(\Lambda Y) \]
and
\[     \mathcal{T}_r: H_\ast(\Lambda Y)\otimes \mathcal{H}_\ast(Y)\rightarrow  H_\ast(\Lambda Y)\otimes \mathcal{H}_\ast(Y). \]
Dually, we have the maps
 \[ H_\ast(\Lambda Y)\xrightarrow{\nu_l}  \mathcal{H}_\ast(Y) \xrightarrow{\mu_l} H_\ast(\Lambda Y) \]
and
\[     \mathcal{T}_l: \mathcal{H}_\ast(Y)\otimes H_\ast(\Lambda Y) \rightarrow  \mathcal{H}_\ast(Y)\otimes H_\ast(\Lambda Y). \]
The aforementioned relationship  between the string topology product and the
coproduct on $H_\ast(\Lambda M)$ is established by the following
\begin{theorem}\label{string}
The following  diagrams
\begin{equation}\label{bistring1}
\begin{array}{cccccc}

          H_\ast( \Lambda M)\otimes  H_\ast( \Lambda M)\otimes
           \mathcal{H}_\ast(M)\!\!\!\!
          &  \xrightarrow{\Cap\, \otimes\, \mu_r}  & \!\!
           H_\ast( \Lambda M)\otimes H_\ast( \Lambda M)
   \\
        _{  1\otimes \mathcal{T}_r }    \uparrow
        \vspace{2mm}\\
     \,\,\,\, H_\ast( \Lambda M)\otimes H_\ast(\Lambda  M) \otimes
           \mathcal{H}_\ast(M) & &
       \hspace{-0.15in}           _{\Delta}  \uparrow \vspace{1mm}\\

 \hspace{-0.3in}  _{\Delta\,\otimes \,\nu_r} \uparrow \\
    H_\ast( \Lambda M)\otimes H_\ast( \Lambda M)
    &  \xrightarrow{\Cap} &  H_\ast( \Lambda M)

   \end{array}
\end{equation}
and
\begin{equation}\label{bistring2}
\begin{array}{cccccc}

         \mathcal{H}_\ast(M)\otimes   H_\ast( \Lambda M)\otimes
          H_\ast( \Lambda M)
          \!\!\!\!
          &  \xrightarrow{ \mu_l\,\otimes\, \Cap }  & \!\! H_\ast( \Lambda M)\otimes H_\ast( \Lambda M)
   \\
        _{ \mathcal{T}_l \otimes 1}    \uparrow
        \vspace{2mm}\\
     \,\,\,\, \mathcal{H}_\ast(M) \otimes H_\ast( \Lambda M)\otimes H_\ast(\Lambda  M)
            & &
       \hspace{-0.15in}           _{\Delta}  \uparrow \vspace{1mm}\\

 \hspace{-0.3in}  _{\nu_l \,\otimes \,\Delta} \uparrow \\
    H_\ast( \Lambda M)\otimes H_\ast( \Lambda M)
    &  \xrightarrow{\Cap} &  H_\ast( \Lambda M)

   \end{array}
\end{equation}

are commutative.

\end{theorem}

Recall that (\cite{Whitehead}) the $\sqcap$ -- product given by
(\ref{stcap}) is induced by means of the pairing of cellular chain
complexes
\begin{equation}\label{chainclassical}
\widetilde{\sqcap}:C_p(K)\otimes C_q(K^\divideontimes)\rightarrow
C_{p+q-n}(K^{\prime}),
\end{equation}
where $K$ is a simplicial subdivision of $M,$ $K^{\prime}$ is the
barycentric subdivision of $K$ and
$K^\divideontimes= \underset{\sigma\in K}\bigcup D(\sigma)$ is a block dissection of $K^\prime$ by the
barycentric stars
  \[D(\sigma)=\bigcup_ {\sigma_i\in K}\sigma_k\supset \cdots   \supset\sigma_1\supset \sigma\] of simplices $\sigma\in K.$
   More precisely, (\ref{chainclassical})
is defined as the composition
\begin{multline*} C_p(K)\otimes C_q(K^\divideontimes)
\xrightarrow{1\otimes \phi}  C_p(K)\otimes C^{n-q}(K)
\xrightarrow{Sd\,\otimes\, \theta^\ast} \\ C_p(K^{\prime})\otimes
C^{n-q}(K^{\prime})
 \xrightarrow{\frown} C_{p+q-n}(K')
\end{multline*}
where  $\phi: C_q(K^{\divideontimes})\rightarrow C^{n-q}(K)$ is  the
Poincar\'{e} chain isomorphism, $Sd$ is the subdivision operator (see
Figure 1), $\theta: K^\prime \rightarrow K$ is a simplicial
displacement and
the last map is the chain  cap -- product.
The problem of constructing  the chain-level intersection pairing  gave
rise
a number of works. A good reference to the subject is the recent book
\cite{Friedman}  (see also \cite{R-T}).
\begin{remark}
1.  Note that if we try to shorten (\ref{chainclassical}) by immediately  applying
  the $\frown$ -- product instead of $ Sd\otimes \theta^\ast$
  to have the value in $C_{p+q-n}(K),$ then the obtained product would
  not satisfy the Leibnitz rule.

  2. The idea to evoke the cubical set $K^{_\Box}$ here arose  by the fact that in (\ref{chainclassical})
  the union of supports of elementary chains in $\sigma \,\widetilde{\sqcap}\, D(\tau)$
  for any two simplices  $\sigma$ and $\tau$ forms a cube in $K^{_\Box}.$
\end{remark}

The  definition of $K^{_\Box}$ is as follows.
Let $K$ denote the triangulated complex of $M.$
For each pair $\sigma\supset \tau$ of simplices from $K$ we assign the
cubical cell $I(\sigma\supset \tau)$ of dimension $|\sigma|-|\tau|:$
Namely, if
$ (\sigma_{k}\supset  \cdots \supset \sigma_{1})\in K^{\prime},$
 $ (\sigma_{k}\supset  \cdots \supset \sigma_{1})\subset \sigma_k,  $
 denotes a barycentric subdivision simplex with the vertices $\sigma_i$
 being subsimplices  of $\sigma_k,$
then
\[  I(\sigma\supset \tau)= \bigcup_{\sigma\supset \sigma_{i}\supset
\tau}
(\sigma\supset \sigma_{r}\supset  \cdots \supset \sigma_{1}\supset\tau)
\subset \sigma \]
with two extreme vertices $ I(\tau\supset \tau)=\min  I(\sigma\supset
\tau)$ and $ I(\sigma\supset \sigma)=\max I(\sigma\supset \tau).$
In particular, we obtain a cubical subdivision of
 the geometric realization of every simplex  $\sigma\in K$   (see Figure 1),
 \[\sigma=I(\sigma \supset v_0)\cup\cdots \cup I(\sigma\supset v_m),\ \ \ \sigma=(v_0,...,v_m)\in K . \]
The
  cubical cellular structure of $M$
formed by the cubes $I(\sigma\supset \tau)$ for all pair of simplices
$\sigma\supset \tau$ is just denoted   by  $ K^{_\Box}.$
Introduce
the cubical face operators $d^\epsilon_i$ on $ K^{_\Box}$ as follows.
Given a pair $\tau \subset \sigma$  of simplices with  $\sigma =(v_0,...,v_m),$
let $\sigma=\tau\cup (v_{q_1},...,v_{q_r}),\, 0\leq q_i\leq m.$
Then
\[d^0_{i}I(\sigma\supset \tau)= I((\sigma\setminus v_{q_i})\supset
\tau)\ \
\text{and} \ \
d^1_{i}I(\sigma\supset \tau)= I(\sigma\supset (\tau  \cup v_{q_i})).\]
The degenerate  operators $\eta_i$   are added formally.  The obtained cubical
set   is denoted by
 \[
 X =\{X_m, d^\epsilon_i, \eta_i \}_{0\leq m\leq n} : =\{K^{\Box}_m, d^\epsilon_i, \eta_i \}_{0\leq m\leq n}.\]
Furthermore, 
let   $P(m)$  denote the set of partitions of the set
$\underline{m}:=\{1,2,...,m\},$ and let 
\[
\overline{P}(m):= \overline{P}'(m)\cup  \overline{P}''(m)\ \ \text{for}\ \      \overline{P}'(m)= P(m) \cup \varnothing\,|\,\underline{m}\ \, \text{and}\ \,
\overline{P}''(m)= P(m)\cup \underline{m}\,|\,\varnothing .
\]
 For $A|B= (i_1,...,i_p)\mid  (j_1,...,j_q)    \in  \overline{P}(m),$
  the corresponding "Cartesian decomposition" of a cube $u\in K^{\Box}_m$ is
\[
u_B\times u_A:=  d^0_B(u)\times d^1_A(u)  =d^0_{j_1}\circ\cdots \circ
d^0_{j_{q}}(u)\times   d^1_{i_1}\circ \cdots\circ d^1_{i_{p}}(u)\]
with
 \[
d^0_{\underline{m}}(u)\times d^1_{\varnothing}(u)=\min u\times u\ \  \
\text{and}\ \ \  d^0_{\varnothing}(u)\times d^1_{\underline{m}}(u)=
u\times \max u.
  \]
Given $u\in K^{\Box}_m,$ denote \[
 \partial^\epsilon u:= \underset{1\leq i\leq m}{\bigcup}d^\epsilon _i u
 \ \   \text{for} \ \  \epsilon=0,1,\ \
\text{and}  \ \
\partial u:=\partial^0 u\cup \partial^1 u, \  \text{and} \ \bar\partial
u:=u \cup \partial u.   \]
 When $w$ is  a face of $v\in X,$ write $w\subset v.$ For each   $w\in X$   with  $w=d^0_{B_e}(e),\,e\in X_n,$ denote
 \[
    D_w:= \bigcup_{e\in X_n }  d^1_{A_e}(e). \]
For each $w\in X $   fix one element  $ d^1_{A_e}(e)  \in D_w,$  denoted by $e_w,$
 such that
   if $w\subset w'\subset e,$ then $e_{w'}\subset e_w;$
  In particular, for a vertex $w\in K^{\Box}_0$ we have $e_w=e,$ and then
   all $d^1_{A_e}(e)$ faces of $e$ as $e_{w'}$'s  are chosen.
Note also that  $|e_w|=n-|w|.$

Now define
 the cellular maps
\[    \theta^{_\Box}: K^{_\Box}\rightarrow  K
\ \ \text{and}\ \
\Theta^{\divideontimes }: K^{_\Box} \rightarrow K^\divideontimes  \]
 as follows.
Recall  the
simplicial displacement  $\theta: K^{\prime} \rightarrow K$  given
 for a vertex
 $\sigma\in K^{\prime}$   by $\theta(\sigma)=\max \sigma;$ then set
 $\theta^{_\Box}(I(\sigma\supset \min \sigma))=\sigma$ with
  $\theta^{_\Box}(I(\sigma\supset \sigma_{(i)}))= _{(i)}\!\!\sigma
  $ for
   $ _{(i)}\sigma:= \partial_{0}\cdots \partial_{i-1}(\sigma)$
   and
  $\sigma_{(i)}:=\partial_{i+1} \cdots \partial_{n}(\sigma).$

  To define $\Theta^\divideontimes$ first observe that
   the union $K^\boxplus:= \underset{w\in K^{\Box}}\bigcup D_w$ is a block dissection of $K^\prime$  different from $K^\divideontimes,$ but
  the set-theoretical identity map
  $id: K'\rightarrow K'$  can be viewed as a cellular map
   $\iota^\divideontimes: K^{\boxplus}\rightarrow K^\divideontimes$
   with the cellular homeomorphisms $\iota^\divideontimes|_{D_w}: D_w\rightarrow D(\sigma)$ for $w$ to be a subdivision cube of a simplex $\sigma$ (see Figure 1).

  Secondly, define a "cubical
   displacement"   $\Theta^{_\Box}:K^{_\Box}\rightarrow K^{\boxplus}$ being a cellular surjection as follows.
   For each $\omega\in K^{_\Box}$ consider the cell $e_w\in D_w.$
Define
\[\Theta^{_\Box}(e_w)=D_w,\ \
\text{with}\ \
\Theta^{_\Box}(\partial e_w)=\partial  D_w
\]
 such that $\Theta^{_\Box}(d^0_ie_w)$ is degenerate for all $ i$ unless
 $i=1 $ in which case
\[ \Theta^{_\Box}(d^0_1e_w)= \bigcup_{w\subset w'\nsubseteq e} D_{w'},
                               \ \ \ \text{while}\ \ \
       \Theta^{_\Box}(\partial^1e_w)=
       \bigcup_{ w\subset w'\subset e}
                 D_{w'}.
                               \]
Then $\Theta^{\divideontimes }=  \iota^{\divideontimes }\circ
  \Theta^{_\Box}. $
Let
 \[\theta^{_\Box}: C_\ast(K^{_\Box})\rightarrow C_\ast( K)
\ \ \text{and}\ \
\Theta^{\divideontimes}: C_{\ast}(K^{_\Box}) \rightarrow
C_\ast(K^\divideontimes)
  \]
be the induced chain maps respectively. Let the cellular map
\[Sd'_{_\Box}:   K^{_\Box}\rightarrow  K'  \]
 be resolved from the composition
$Sd=Sd'_{_\Box}\circ Sd_{_\Box}$  (see Figure 1),
and \[\theta^{\vartriangle}: K'\rightarrow K^{_\Box} \]  be a cellular map with    $   \theta^{\vartriangle}(\sigma')\subseteq  I(\sigma\supset \tau)$ for all $\sigma'= (\sigma\supset \sigma_{r}\supset  \cdots
\supset \sigma_{1}\supset\tau)$
    such that on the chain level
\[\theta^{\vartriangle} \circ Sd'_{_\Box}=id. \]

Define  the $\sqcap$ -- product
\[   \sqcap:  C_p(K^{_\Box})  \otimes  C_q(K^{_\Box})\rightarrow C_{p+q-n}(K^{_\Box})   \]
 as the composition
\[
  C_p(K^{_\Box})\otimes C_{q}(K^{_\Box}) \xrightarrow
 {\theta^{\Box}\,\otimes\, \Theta^{\divideontimes}}
   C_p(K)\otimes C_{q}(K^\divideontimes)
   \xrightarrow {\widetilde{\sqcap}}
   C_{p+q-n}(K^\prime)\xrightarrow{\theta^{\vartriangle}}
    C_{p+q-n}(K^{_\Box}).
 \]
In particular,
 the following diagram
\begin{equation}\label{comparison}
\begin{array}{cccccc}
  C_p(K)\otimes C_{q}(K^\divideontimes)  &
  \xrightarrow{\widetilde{\sqcap}}                 &
      C_{p+q-n}(K') \\
       \hspace{-0.5in}  _{\theta^{_\Box}\otimes\,
       \Theta^{\divideontimes}} \uparrow            &
       &
       \hspace{-0.3in}
         _{Sd'_{_\Box}} \uparrow \vspace{1mm}\\
 \!\!\! C_p(K^{_\Box})\otimes C_{q}(K^{_\Box}) &  \xrightarrow{\sqcap}
 &  C_{p+q-n}(K^{_\Box})
     \end{array}
\end{equation}
commutes. By examining the cellular map $\Theta^{\divideontimes}$
we easily deduce that
in terms of
  elementary chain  cubes
$u,v\in C_\ast(X) $
  \[
  u\sqcap v=\left\{
      \begin{array}{lllll}
      d^1_{A_u}(u) ,   &  v=e_w, & w=d^0_{B_u} (u),\ \ \text{some}\ \ B_u,
         \vspace{1mm}\\
           0      ,&    v=e_w    ,&   w\neq d^0_{B_u}( u), \ \ \text{neither}\ \ B_u, \vspace{1mm}     \\
-d^1(u\sqcap e_w)+ u\sqcap d^1(e_w),  &
                  v=d^0_{1}(e_w).

 \end{array}
 \right.
      \]
Also note that
      we  may have non-zero product $u\sqcap d^0_i(v)$  for $u$  and $v$
      being some cells in $D_w,$ but in each such case
  $u\sqcap v=  du\sqcap v=u\sqcap dv =0.$

We have the following
\begin{proposition}\label{algebra} The product $\sqcap : C_p(X)\otimes
C_{q}(X) \rightarrow  C_{p+q-n}(X)$  for $u \otimes v\in C_p(X)\otimes
C_{q}(X)$
satisfies the equality
\[  d(u \sqcap  v)= (-1)^{n+q}\, d(u)\sqcap  v +   u\sqcap d(v).     \]
  \end{proposition}
Thus,  we obtain the induced $\sqcap$ -- product on
    the homology $H_\ast(X)$, and
\begin{proposition}\label{commut}
The  product
     \[\sqcap  :H_p( X)\otimes H_q( X)\rightarrow H_{p+q-n}( X)\]
    is commutative and associative.
  \end{proposition}
Consequently,
\begin{proposition}\label{classical} Let $M$ be an oriented closed
 $n$-manifold. The  $\sqcap$ -- product  on $H_\ast(M)$
 agrees with the classical intersection product being commutative and
 associative.
 \end{proposition}

Furthermore, on the chain level we have
\begin{proposition}\label{serrediagonal} (i) The   product $\sqcap:
C_*(X) \otimes C_*(X) \to C_{*-n}(X)$ satisfies the equation
\begin{equation}\label{left}
  (\sqcap \otimes 1) \circ (1 \otimes T)\circ(\Delta_{_\Box}
 \otimes 1)
  = \Delta_{_\Box} \circ \sqcap
   \end{equation}
  with respect to  the cubical diagonal
$\Delta_{_\Box}: C_*(X) \to C_*(X) \otimes C_*(X);$

(ii)
There is  a chain homotopy
 \[(  1\otimes \sqcap) \circ (T \otimes 1)\circ(1\otimes
   \Delta_{_\Box})\simeq
   \Delta_{_\Box} \circ \sqcap.
   \]
  \end{proposition}

Theorem \ref{intersection} follows from Proposition
\ref{serrediagonal}. The definition of the $\Cap$ -- product on the loop homology $H_\ast(\Lambda Y)$ uses
the  combinatorial model  $|\widehat{\mathbf{\Omega}}X| \xrightarrow{\iota}
 |\widehat{\mathbf{\Lambda}} X| \xrightarrow{\zeta} |X|$
of the free loop fibration
 $\Omega Y\rightarrow \Lambda Y\rightarrow Y.$ This model
  is built   by means of   \emph{  the cubical  necklical set}
 $\widehat{\mathbf{\Omega}} X $ and
 \emph{the cubical   closed necklical set}
 $\widehat{\mathbf{\Lambda}}X$
 both having canonical  permutahedral set structures. The definition of
these  necklical sets mimics the one of simplicial necklical sets
 \cite{RS2}. Using an explicit diagonal of permutahedra \cite{SU} (subsection \ref{DP} below)  we  introduce
 the coproducts on the permutahedral chain complexes   $C^\diamond_\ast({\mathbf{\Omega}} X)$ and  $C^\diamond_\ast({\mathbf{\Lambda}} X)$
  that induce the standard coproducts on the  loop  homologies
 $H_\ast(\Omega Y)$   and
  $H_\ast(\Lambda Y)$ respectively (compare \cite{saneFREE}).
  Then we detect the relation between  the  $\Cap$ -- product and the  $\Delta$ -- coproduct  on   $H_\ast(\Lambda Y)$  (Theorem \ref{string}).

   As  $C^\diamond_\ast({\mathbf{\Lambda}} X)$ is identified with the coHochschild complex $ \hat \Lambda C_\ast(X)$  of the cubical chains $C_\ast(X)$ (Theorem \ref{hat-coHoch})
  we immediately obtain the (twisted) coproduct on
$ \hat \Lambda C_\ast(X)$
 in terms of the coproducts on $C_\ast(X)$ and $C_\ast({\mathbf{\Omega}}
X).$
 In the pure algebraic setting the above relationship motivates to introduce the notion of a \emph{loop bialgebra}
for a dg coalgebra $C$ endowed with higher order cooperations (including Steenrod's chain $\smile_1$ -- cooperation,  e.g., $C$ is a coGerstenhaber coalgebra)   in the last section.

\textbf{Acknowledgments.} I am grateful to M. Rivera for valuable discussions during the process of writing  the paper.

 \section{The proof of Propositions  \ref{algebra}, \ref{commut}    and
 \ref{serrediagonal}}

Throughout the paper  the coefficients is a field unless otherwise is stated.
The chain complex $(C_\ast(X),d)$  of a pointed  cubical set $(X,x_0)$ is defined as
$C_\ast(X)=C'(X)/C'_{>0}(D(x_0))$ where $C'(X)$ is the chains of $X,$
and $D(x_0)$ is the degeneracies   in $X$ arising from the  vertex $x_0;$
 the differential $d$ is defined for $u\in X_m$   by
\[d(u)=d^0(u)-d^1(u)=\sum_{1\leq i\leq m} (-1)^{i}d^0_i(u)-\sum_{1\leq i\leq m}(-1)^{i}d^1_i(u),\]
and the cubical chain-level (Serre) diagonal $\Delta_{_\Box}: C_*(X) \to C_*(X)
\otimes C_*(X)$
is defined  by
\begin{equation}\label{dcube}
\Delta_{_\Box}(u)=\sum_{ A|B\in \overline{P}(m)}
sgn(A;B)\cdot u_B\otimes u_A.
\end{equation}

  \begin{remark}
  We have
$
  \left( \dim u_B\,,\dim u_A)\right)=(\#A,\#B),$ so that
     this equality    answers to the sign of the differential of the
     cobar construction on $C_*(X)$
     (cf. \!(\ref{Psign}),(\ref{Pset}),(\ref{cobarsign})
      and Theorem \ref{hat-coHoch}).
  \end{remark}

\noindent\emph{Proof of Proposition \ref{algebra}}.
If $u\sqcap v$  is zero or zero-dimensional,   then $d u\sqcap v=
u\sqcap d v=0.$ Let $u\sqcap v$ be  positive dimensional.
 Consider
$d^{\epsilon}(u\sqcap v),\, \epsilon=0,1.$

1.   Let $v=e_w.$  We have two subcases:

1a.  $\epsilon=0.$
Using the  cubical relation $d^0d^1(u)=-d^1d^0(u)$ we immediately obtain the equality
 \[d^0(u\sqcap e_w)=(-1)^{n+q}d^0(u)\sqcap e_w.\]

1b. $\epsilon=1.$  By the second item of the definition of the $\sqcap$ -- product
 \[ d^1(u\sqcap e_w)=u\sqcap \left(- d^0_1(e_w)+ d^1(e_w)\right).\]
 By definition 
 we have $d^1(u) \sqcap  e_w=0$ (since $w\nsubseteq \partial^1(u)$)
 and  $u\sqcap d^0_i( e_w  )=0$ for  $i>1.$
Thus,
\[ d(u\sqcap e_w)= (-1)^{n+q} \, d(u)\sqcap e_w +  u\sqcap d(e_w).\]

2. Let $v=d^0_1(e_w).$

2a.  $\epsilon=0.$ Then
 \begin{multline*}
d^0(u\sqcap d^0_1(e_w))= d^0(-d^1(u\sqcap e_w )+ u\sqcap d^1(e_w))=
d^1d^0(u\sqcap e_w )+ d^0(u\sqcap d^1(e_w))=
\\
(-1)^{n+q}d^1(d^0(u)\sqcap e_w) -(-1)^{n+q}d^0(u)\sqcap d^1(e_w)=\\
(-1)^{n+q}\,d^0(u)\sqcap d^1(e_w)-(-1)^{n+q}\, d^0(u)\sqcap d^0_1(e_w)
 -
 (-1)^{n+q}\,d^0(u)\sqcap d^1(e_w)=\\
 - (-1)^{n+q}\,d^0(u)\sqcap d^0_1(e_w).
\end{multline*}

2b.   $\epsilon=1.$ Then
\begin{multline*}
d^1(u\sqcap d^0_1(e_w))=d^1(-d^1(u\sqcap e_w )+ u\sqcap d^1(e_w))=d^1(u\sqcap d^1(e_w))=
\\
u\sqcap(d^0+d^1)(d^1(e_w))=u\sqcap d^0_1d^1(e_w)=-
u\sqcap d^1d^0_1(e_w).
\end{multline*}
Thus,
\[ d(u \sqcap d^0_1(e_w)) =  d^0(u)\sqcap d^0_1(e_w)+u\sqcap d^1d^0_1(e_w).  \]
Since $     d^1(u)\sqcap d^0_1(e_w)=0 $  (since $w\nsubseteq \partial^1(u)$) and    $  u\sqcap d^0d^0(e_w)   =0,    $ obtain
\[
\hspace{1.05in} d(u \sqcap d^0_1(e_w)) = (-1)^{n+q}\, d(u)\sqcap d^0_1(e_w)+u\sqcap d(d^0_1(e_w)).   \hspace{0.3in}\Box
\]
\bigskip

\noindent \emph{Proof of Proposition \ref{commut}}.
Let $(X^{op},\widetilde{d}^\epsilon_i) $ be the cubical set obtained from $X$
by interchanging  the face operators $d^0_i(u)$ and $d^1_i(u)$ for all
$i$ and $u,$ i.e.,
$\widetilde{d}^0_i(u)=d^1_i(u)$ and $\widetilde {d}^1_i(u)=d^0_i(u).$
 Then
define  the $\sqcap^{op}$-- product
\[   \sqcap^{op}:  C_p(X^{op})  \otimes  C_q(X^{op})\rightarrow C_{p+q-n}(X^{op})   \]
 as the composition
\[
  C_p(K^{_\Box})\otimes C_{q}(K^{_\Box}) \xrightarrow
 {\Theta^{\divideontimes}\otimes\,\theta^{\Box}}
 C_{p}(K^\divideontimes)  \otimes   C_q(K)
   \xrightarrow {\widetilde{\sqcap}^{\,op}}
   C_{p+q-n}(K^\prime)\xrightarrow{\theta^{\vartriangle}}
    C_{p+q-n}(K^{_\Box}).
 \]
Let $e^{op}_w \in X^{op}$ be defined   as   $e_w\in X,$ but by replacing $d^1$ operator by $d^{0},$
and then
in terms of
  elementary chain  cubes
$u,v\in X^{op} $
  \[
  u\sqcap v=\left\{
      \begin{array}{lllll}
      d^0_{B_v}(v) ,   &  u=e^{op}_w,  \ \ \ \ \ w=d^1_{A_v}( v),
         \vspace{1mm}\\
           0      ,&    u=e^{op}_w,    \ \ \  \  \   w\neq d^1_{A_v}( v)  ,  \vspace{1mm}     \\
-d^0( e^{op}_w  \sqcap^{op}  v) +  d^0(e^{op}_w)\sqcap^{op}  v,  &
                  u=d^1_{1}(e^{op}_w).

 \end{array}
 \right.
      \]
A given $u\in X$ considered as an element in $X^{op}$ is denoted by  $u^{op},$    and let 

\noindent $\iota: X \rightarrow X^{op},  \,\,  u\rightarrow u^{op}. $
Then
\[    (u\sqcap v)^{op}=(-1)^{ pq }\, v^{op}\sqcap^{op} u^{op}. \]
Since
 $\iota $ induces an isomorphism in homology and $\sqcap$ and $\sqcap^{op}$ induce the same 
 product in homology, and the commutativity follows. To check the associativity is straightforward.
\hspace{4.3in}$\Box$

\vspace{0.1in}

\noindent \emph{Proof of Proposition \ref{serrediagonal}}.
(i)
 Let
 $u\otimes v\in C_p(X)\otimes C_q(X),$  and
  $u= u_D\times u_C$  for some $C|D\in \overline{P}(p).$

(i1) $ v=e_w.$ Let  $u_{B}\otimes u_{A}$  be a component in $ \Delta_{_\Box}(u\sqcap e_w)$ for some $A|B\in \overline{P}(p+q-n).$ Then $u=w\times  (u\sqcap e_w ) =w\times u_{B}\times u_{A}=
u_D\times u_A$ for $u_D:=w\times u_{B}.$  Hence,  $u_{B}=u_D\sqcap e_w.$
For any component  $u_{D'}\otimes u_{A'} \in \Delta_{_\Box}(u)$   such that $w\nsubseteq u_{D'}$ we have by definition that
    $u_{D'}\sqcap e_w=0.$  Thus, equality (\ref{left})   is verified
    for $(u,v)=(u,e_w).$

(i2)   $v=d^0_1(e_w).$ If $ d^1(u\sqcap e_w)= u\sqcap d^1(e_w)$   nothing is to prove. Otherwise $d^1(u\sqcap e_w)$
contains
 $ u\sqcap d^1(e_w)$
as a summand component.
 Consider $d^1_i(u\sqcap e_w)     $ for some $i$ not contained in
 $ u\sqcap d^1(e_w),$        and let   $u_{C}\otimes u_A$  be a component in $\Delta_{_\Box}d^1_i(u\sqcap e_w).$
 Then there is a component $u_{C'}\otimes u_A\in \Delta_{_\Box}(u\sqcap e_w)$ such that  $d^1_k( u_{C'})= u_C,  $ where $k$ is defined by $i=i_k$  in $ A=\{ i_1<\cdots <i_k<  \cdots <i_{p}\},\,   d^1_{A}( u\sqcap e_w )= u_A. $
 By the item (i1) there is a component $u_{D}\otimes u_{A}\in  \Delta_{_\Box}(u)$   such that $u_D \sqcap e_w=u_{C'}. $
Hence, the item (i2), and, consequently, the item (i)  is proved.

(ii) Let $X^{op}$  be the cubical set as in the proof of Proposition \ref{commut}.   Similarly to  the proof of the item (i)  we establish  the equality in $X^{op}$
 \[(  1\otimes  \sqcap^{op}) \circ (T \otimes 1)\circ(1\otimes
   \Delta_{_\Box})
  = \Delta_{_\Box} \circ \sqcap^{op}.
   \]
   Consequently, we get the chain homotopy as desired.
  \hspace{1.6in}$\Box$

\vspace{0.2in}

\hspace{-0.15in}
\unitlength 1mm 
\linethickness{0.8pt}
\ifx\plotpoint\undefined\newsavebox{\plotpoint}\fi 
\begin{picture}(125.164,88)(0,0)
\put(49.798,2.676){\line(1,0){30.473}}
\put(18.284,48.312){\line(1,0){30.473}}
\put(83.096,48.312){\line(1,0){30.473}}
\multiput(80.271,2.676)(.067162651,.105668675){166}{\line(0,1){.105668675}}
\multiput(48.757,48.312)(.067162651,.105662651){166}{\line(0,1){.105662651}}
\multiput(113.569,48.312)(.067162651,.105662651){166}{\line(0,1){.105662651}}
\multiput(91.42,20.217)(-.067434146,.076863415){205}{\line(0,1){.076863415}}
\multiput(59.906,65.852)(-.067434146,.076863415){205}{\line(0,1){.076863415}}
\multiput(124.718,65.852)(-.067434146,.076863415){205}{\line(0,1){.076863415}}
\multiput(77.596,35.974)(-9.66233,.04933){3}{\line(-1,0){9.66233}}
\multiput(46.082,81.609)(-9.66233,.04967){3}{\line(-1,0){9.66233}}
\multiput(110.894,81.609)(-9.66233,.04967){3}{\line(-1,0){9.66233}}
\multiput(48.609,35.676)(-.067444444,-.072949074){216}{\line(0,-1){.072949074}}
\multiput(17.095,81.312)(-.067444444,-.072949074){216}{\line(0,-1){.072949074}}
\multiput(81.907,81.312)(-.067444444,-.072949074){216}{\line(0,-1){.072949074}}
\multiput(34.041,19.919)(.067275862,-.073685345){232}{\line(0,-1){.073685345}}
\multiput(2.527,65.555)(.067275862,-.073685345){232}{\line(0,-1){.073685345}}
\multiput(67.339,65.555)(.067275862,-.073685345){232}{\line(0,-1){.073685345}}
\multiput(2.527,65.406)(8.218286,.063714){7}{\line(1,0){8.218286}}
\multiput(17.392,81.609)(.0674516129,-.0716064516){465}{\line(0,-1){.0716064516}}
\multiput(45.784,81.609)(-.0674019608,-.0816102941){408}{\line(0,-1){.0816102941}}
\multiput(34.041,19.771)(8.197,.063714){7}{\line(1,0){8.197}}
\multiput(48.609,36.122)(.0673961864,-.0708601695){472}{\line(0,-1){.0708601695}}
\multiput(77.596,35.825)(-.0673421687,-.0798771084){415}{\line(0,-1){.0798771084}}
\put(89.934,73.136){\line(0,1){0}}
\multiput(82.055,81.758)(.172246032,-.067246032){126}{\line(1,0){.172246032}}
\put(96.474,80.569){\line(0,1){0}}
\multiput(96.772,81.609)(.0660556,-1.8581111){18}{\line(0,-1){1.8581111}}
\multiput(110.894,81.461)(.066875,-.825025){40}{\line(0,-1){.825025}}
\multiput(91.123,57.825)(.158273381,-.067374101){139}{\line(1,0){.158273381}}
\multiput(83.542,48.609)(.154072993,.067270073){137}{\line(1,0){.154072993}}
\put(104.65,60.947){\line(0,1){0}}
\multiput(104.799,57.825)(.167386555,.067453782){119}{\line(1,0){.167386555}}
\multiput(97.366,65.704)(.213160377,-.067320755){106}{\line(1,0){.213160377}}
\put(67.785,65.852){\line(1,0){56.636}}
\multiput(103.758,72.987)(.194924528,-.067311321){106}{\line(1,0){.194924528}}
\multiput(97.515,65.555)(.200537736,.067311321){106}{\line(1,0){.200537736}}
\multiput(97.217,66.001)(-.206234234,.066954955){111}{\line(-1,0){.206234234}}
\multiput(91.271,57.676)(-.199865546,.067453782){119}{\line(-1,0){.199865546}}
\multiput(97.366,65.703)(-.169230769,-.067461538){130}{\line(-1,0){.169230769}}
\multiput(83.541,65.852)(-.04933,-5.74767){3}{\line(0,-1){5.74767}}
\multiput(83.541,65.852)(-.0646087,.698){23}{\line(0,1){.698}}
\put(82.352,81.758){\circle*{.595}}
\put(67.636,65.703){\circle*{.595}}
\put(83.095,48.609){\circle*{.595}}
\put(113.569,48.46){\circle*{.595}}
\put(124.718,66.001){\circle*{.595}}
\put(110.893,81.163){\circle*{.595}}
\put(97.366,65.852){\circle*{.665}}
\put(2.676,65.555){\circle*{.595}}
\put(17.541,81.46){\circle*{.595}}
\put(45.784,81.46){\circle*{.595}}
\put(59.757,65.852){\circle*{.595}}
\put(32.257,65.406){\circle*{.595}}
\put(18.433,48.014){\circle*{.841}}
\put(48.757,48.46){\circle*{.595}}
\put(48.757,36.271){\circle*{.297}}
\put(34.338,19.622){\circle*{.595}}
\put(77.595,35.973){\circle*{.595}}
\put(91.271,20.216){\circle*{1.189}}
\put(64.217,19.919){\circle*{.595}}
\put(49.649,2.676){\circle*{.665}}
\put(80.42,2.527){\circle*{.595}}
\multiput(63.325,30.622)(.19325,-.0669){40}{\line(1,0){.19325}}
\multiput(49.501,24.379)(.04933,-3.36967){3}{\line(0,-1){3.36967}}
\multiput(49.649,14.27)(.19622667,-.06738667){75}{\line(1,0){.19622667}}
\multiput(64.366,9.216)(.151912088,.066978022){91}{\line(1,0){.151912088}}
\put(78.19,15.311){\line(0,1){9.96}}
\multiput(78.19,25.271)(-.16186667,.06606667){45}{\line(-1,0){.16186667}}
\put(63.325,30.473){\line(0,1){5.797}}
\multiput(78.339,25.122)(.14864286,.0672381){42}{\line(1,0){.14864286}}
\multiput(78.19,15.162)(.13803571,-.06635714){56}{\line(1,0){.13803571}}
\put(64.366,9.068){\line(0,-1){6.392}}
\multiput(49.501,14.27)(-.11534483,-.06662069){58}{\line(-1,0){.11534483}}
\multiput(49.501,24.379)(-.13131667,.06688333){60}{\line(-1,0){.13131667}}
\put(49.501,24.527){\circle*{1.5}}
\put(41.771,27.946){\circle*{.595}}
\put(63.176,35.825){\circle*{.595}}
\put(63.176,30.622){\circle*{1.5}}
\put(56.338,27.946){\circle*{1.5}}
\put(71.055,27.946){\circle*{1.5}}
\put(71.2,12.2){\circle*{1.5}}
\put(78.19,25.271){\circle*{1.5}}
\put(84.433,27.798){\circle*{.595}}
\put(78.19,19.77){\circle*{1.5}}
\put(78.041,15.46){\circle*{1.5}}
\put(85.771,11.446){\circle*{.595}}
\put(64.217,9.216){\circle*{1.5}}
\put(63.92,2.676){\circle*{.595}}
\put(57.379,11.595){\circle*{1.5}}
\put(49.798,14.419){\circle*{1.5}}
\put(42.514,10.406){\circle*{.595}}
\put(49.352,19.919){\circle*{1.5}}
\put(74.474,73.285){\circle*{.595}}
\put(83.393,70.906){\circle*{.595}}
\put(103.907,73.136){\circle*{.595}}
\put(111.636,70.312){\circle*{.595}}
\put(118.92,72.69){\circle*{.595}}
\put(112.231,65.852){\circle*{.595}}
\put(112.528,61.244){\circle*{.595}}
\put(120.109,58.42){\circle*{.595}}
\put(104.65,57.528){\circle*{.595}}
\put(97.961,55.149){\circle*{.595}}
\put(91.123,57.528){\circle*{.595}}
\put(83.244,60.203){\circle*{.595}}
\put(75.366,56.933){\circle*{.595}}
\put(96.92,75.96){\circle*{.595}}
\put(45.933,83.987){\makebox(0,0)[cc]{$v_1$}}
\put(111.042,83.839){\makebox(0,0)[cc]{$v_1$}}
\put(77.893,37.906){\makebox(0,0)[cc]{$v_1$}}
\put(32.257,69.42){\makebox(0,0)[cc]{$v_3$}}
\put(98.406,68.676){\makebox(0,0)[cc]{$v_3$}}
\put(64.217,22.2){\makebox(0,0)[cc]{$_{w}$}}
\put(16.946,83.69){\makebox(0,0)[cc]{$v_2$}}
\put(82.798,83.839){\makebox(0,0)[cc]{$v_2$}}
\put(48.609,38.054){\makebox(0,0)[cc]{$v_2$}}
\put(2.081,68.379){\makebox(0,0)[cc]{$v_4$}}
\put(67.636,68.825){\makebox(0,0)[cc]{$v_4$}}
\put(17.392,45.636){\makebox(0,0)[cc]{$v_5$}}
\put(83.541,45.784){\makebox(0,0)[cc]{$v_5$}}
\put(49.798,0){\makebox(0,0)[cc]{$v_5$}}
\put(48.906,45.784){\makebox(0,0)[cc]{$v_6$}}
\put(113.569,45.933){\makebox(0,0)[cc]{$v_6$}}
\put(80.271,.595){\makebox(0,0)[cc]{$v_6$}}
\put(59.757,69.122){\makebox(0,0)[cc]{$v_7$}}
\put(125.164,68.825){\makebox(0,0)[cc]{$v_7$}}
\put(63.325,26.716){\makebox(0,0)[cc]{$e_w$}}
\put(57,23.8){\makebox(0,0)[cc]{$_{d^0_1}$}}
\put(70.014,23.8){\makebox(0,0)[cc]{$_{d^0_2}$}}
\put(58.271,31.5){\makebox(0,0)[cc]{$_{d^1_2}$}}
\put(68.23,31.5){\makebox(0,0)[cc]{$_{d^1_1}$}}

\put(53.5,18){\makebox(0,0)[cc]{$_{w'_1}$}}
\put(62,14.419){\makebox(0,0)[cc]{$_{w'_2}$}}
\put(67,12.933){\makebox(0,0)[cc]{$_{w'_3}$}}
\put(73.5,18){\makebox(0,0)[cc]{$_{w'_4}$}}
\put(30.325,19.473){\makebox(0,0)[cc]{$v_4$}}
\put(94.69,20.365){\makebox(0,0)[cc]{$v_7$}}
\put(30.771,88){\makebox(0,0)[cc]{$K$}}
\put(95.731,88){\makebox(0,0)[cc]{$K'$}}
\put(66.2,-4){\makebox(0,0)[cc]{$K^{\Box}$}}
\put(59.163,49.203){\vector(1,0){12.338}}
\put(64.366,52.325){\makebox(0,0)[cc]{$Sd$}}

\multiput(83.375,48.875)(.067450495,.0810643564){404}{\line(0,1){.0810643564}}
\multiput(82.25,81.75)(.0673491379,-.0713900862){464}{\line(0,-1){.0713900862}}
\multiput(96.474,75.961)(.09933,.04933){3}{\line(1,0){.09933}}
\multiput(90.974,73.136)(-.1188,.0596){5}{\line(-1,0){.1188}}
\put(96.623,76.109){\line(5,-2){14.865}}
\multiput(111.488,70.163)(.0637143,-.6795){14}{\line(0,-1){.6795}}
\multiput(112.38,60.65)(-.16201124,-.06680899){89}{\line(-1,0){.16201124}}
\multiput(83.096,70.609)(.063714,-1.465286){7}{\line(0,-1){1.465286}}
\multiput(96.92,75.812)(.16458333,.06725){84}{\line(1,0){.16458333}}
\multiput(67.785,65.704)(.21773239,.067){71}{\line(1,0){.21773239}}
\multiput(83.393,70.906)(.19261972,.067){71}{\line(1,0){.19261972}}
\multiput(83.393,71.204)(.20904687,.06735937){64}{\line(1,0){.20904687}}
\put(98.11,48.163){\circle*{.892}}
\put(83.096,70.609){\circle*{.892}}
\put(83.542,65.704){\circle*{.892}}
\put(96.623,81.907){\circle*{.892}}
\put(89.934,73.434){\circle*{.892}}
\multiput(49.946,24.825)(.15032048,.0668091){89}{\line(1,0){.15032048}}

\put(71.25,77.125){\vector(-1,0){12}}

\put(40.75,34.75){\vector(-1,1){9.25}}

\put(27.75,42.75){\vector(1,-1){9.25}}
\put(28.75,37.00){\makebox(0,0)[cc]{$Sd_{_\Box}$}}

\put(95.75,43.75){\vector(-1,-1){9.25}}
\put(90.75,33.75){\vector(1,1){9.25}}
\put(100.00,37.00){\makebox(0,0)[cc]{$Sd'_{_\Box}$}}

\put(64.875,79.875){\makebox(0,0)[cc]{$\theta$}}
\put(89.75,41.00){\makebox(0,0)[cc]{$\theta^{\Delta}$}}

\put(39.75,40.875){\makebox(0,0)[cc]{$\theta^{_\Box}$}}

\end{picture}

\vspace{0.4in}

\begin{center}

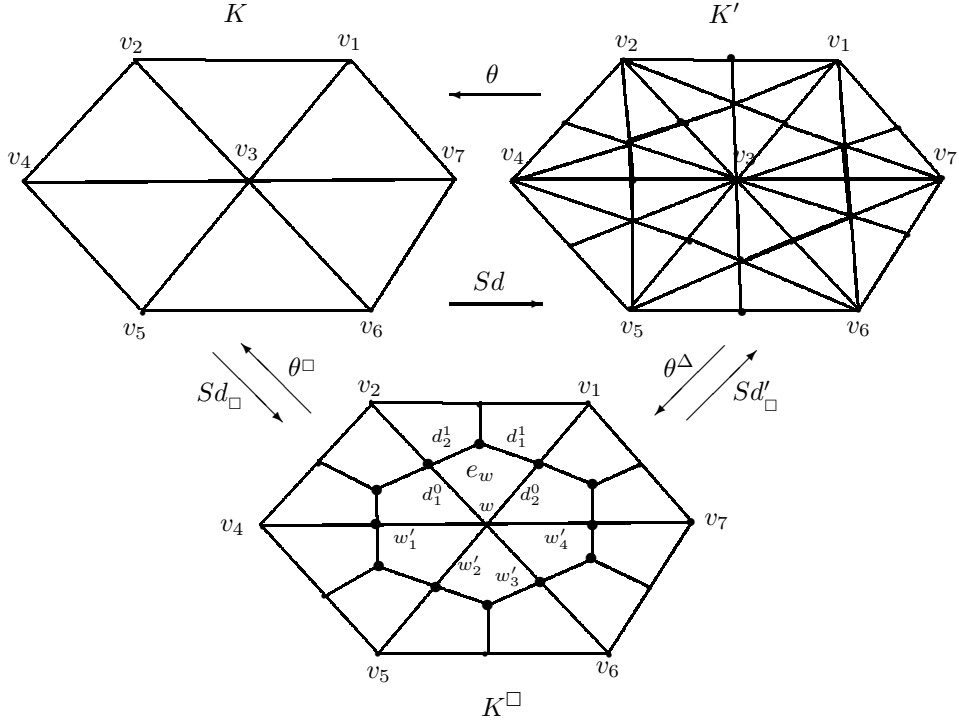
Figure 1.      The first barycentric  simplicial  and  cubical
subdivisions of $K$ with $Sd=Sd'_{_\Box}\circ Sd_{_\Box}.$
\end{center}

\vspace{0.2in}

To prove Theorem \ref{string} we need some preliminaries.

 \section{Cubical  (closed) necklaces and  necklical sets}

\subsection{Cubical necklaces}
 Denote by $Set_{_{\Box}}$ the category of cubical sets and by
 $I^{m}$ the standard $m$-cube. A \textit{cubical
 necklace} is a wedge of standard cubes
\[T=I^{n_1} \vee ... \vee I^{n_k} \in Set_{_\Box},\ \ \ \ n_i,k\geq
1,\]
 where
the last vertex of $I^{n_i}$ is identified with the first vertex of
$I^{n_{i+1}}$ whenever $k\geq2$ and $1\leq i<k.$
Each $I^{n_i}$ is a subcubical set of $T$, which we call a
\textit{bead} of $T$. Denote by $b(T)$ the number of beads in $T$. The
set  $T_0$ of the \textit{vertices} of $T$, inherits a partial ordering
from the ordering of the beads in $T$ and the partial ordering of the
vertices of each $I^{n_i}$. A morphism of cubical necklaces $f: T \to
T'$  is a morphism of cubical sets which preserves first and last
vertices. If $T= I^{n_1} \vee ... \vee I^{n_k}$ is a cubical necklace,
then the dimension of $T$ is defined to be $\text{dim}(T)= n_1 + \cdots
+ n_k -k$. Denote by $Nec_{_\Box}$ the category of cubical necklaces. A
\textit{cubical necklical set} is a functor $ Nec^{op}_{_\Box}
\rightarrow Set$ and a morphism of cubical necklical sets is given by a
natural transformation of functors. Denote the category of cubical
necklical sets by $Set^{_{\Box}}_{Nec}$.

\begin{proposition} \label{maps1}
Any non-identity morphism in $Nec_{_\Box}$ is a composition of
morphisms of the following type
\begin{itemize}
\item [(i)]  $f: T \to T'$ is an injective morphism of cubical
    necklaces and $ \dim(T') - \dim(T) =1,$ $ b(T)-b(T^\prime)=1,$
    and  $T$ and $T'$ have the same number of vertices;

\item [(ii)]
 $f_{p,j}: I^{n_1} \vee ... \vee I^{n_p+1}  \vee... \vee  I^{n_k} \to I^{n_1} \vee ...\vee
    I^{n_{p}}  \vee... \vee  I^{n_k}$
 is a morphism of cubical necklaces of the form
    \[f_{p,j}=id \vee ...\vee id \vee \eta^j\vee id\vee ...   \vee id,\,
    1\leq p\leq k, \] such that
     $\eta^j: I^{n_p+1} \to I^{n_p}$ is a cubical co-degeneracy morphism
      for  $n_p\geq 1$   and $1\leq j\leq n_p.$

\item [(ii')] $f_p: I^{n_1} \vee ...\vee I^{n_p}  \vee... \vee  I^{n_k} \to I^{n_1} \vee ...\vee I^{n_{p-1}} \vee I^{n_{p+1}} \vee... \vee  I^{n_k},\,k,p\geq 2,$ is a morphism of cubical necklaces such that $f_p$ collapses the $p$-th bead $I^{n_p}$  in the domain to the last vertex of the $(p-1)$-th bead in the target and the restriction of $f$ to all the other beads is identity.
\end{itemize}
\end{proposition}
\begin{remark}\label{cartesiandecomp}
\normalfont
1. Unlike  simplicial necklaces here is no morphism of the form
 $ id\vee \delta^{\epsilon}_i\vee id :T \to T^{\prime},\epsilon=0,1,$
 because neither proper face of the standard cube $I^k$ contains the both minimal
 and maximal vertices of $I^k$ simultaneously.

2. Morphisms of type $(i)$
 are of the form
 \[\delta^{A_i|B_i}:  I^{n_1} \vee ... \vee I^{n_i}\vee I^{n_i+1}  \vee... \vee  I^{n_k}
  \to
 I^{m_1} \vee ...\vee    I^{m_i} \vee... \vee  I^{m_{k-1}}\]
for
$\delta^{A_i|B_i}:= id\vee ...\vee id \vee S^{A_i|B_i} \vee id\vee...\vee id $  where
$ S^{A_i|B_i}: I^{n_i} \vee I^{n_{i+1}} \to  I^{m_i},\,$  $m_i=n_i +
n_{i+1},\,$
$A_i|B_i\in P(m_i),\,1\leq i< k,$
 is the injective map of cubical sets whose image in
$I^{m_i}$ is the wedge of the two subcubical sets corresponding to the
$A_i|B_i$-th term in the Serre diagonal map applied to the unique
non-degenerate top dimensional cube in $I^{m_i}$ (cf. (\ref{dcube})).
For $n\geq 2,$ denote
\[\kappa(n)= \left( \begin{array}{c}
   n\\
   1 \\
  \end{array}
 \right) +\cdots + \left( \begin{array}{c}
   n\\
   n-1\\
  \end{array}
 \right).\]
Then
for each cubical  necklace $T^\prime=I^{m_1} \vee ... \vee
I^{m_{k-1}},$ there are exactly
$\underset{1\leq i<k}{\sum}\kappa(m_i)$
morphisms
$\delta^{A_i|B_i}: T \to T^\prime.$

\end{remark}

\subsection{Cubical closed necklaces}
We now define the category $Nec^{_\Box}_c$ of \textit{cubical closed
necklaces}. The objects of $Nec^{_\Box}_c$ are cubical sets of the form
$R=I^{n_0}\vee T$, where $n_0 \geq 0$, $T= I^{n_1} \vee ... \vee
I^{n_k}$ is a cubical necklace in $Nec^{_\Box}$, and the first vertex
of $I^{n_0}$ is identified with the last vertex of $T$.  We will call
$I^{n_0}$ and $I^{n_k}$ the \textit{first}  and \textit{last beads} of
$R$, respectively. Thus, $b(R)=b(T)+1.$  The vertices of $R$ also
inherit a  natural partial ordering from the ordering of the set of
beads of $R$ and partial ordering of the vertices on each bead
(see Figure 2).

Morphisms between cubical closed necklaces are defined to be maps of
cubical sets which preserve first beads.
 If $R=I^{n_0} \vee T= I^{n_0}
\vee I^{n_1} \vee ... \vee I^{n_k}$ is a cubical closed necklace, then
the dimension of $R$ is defined to be $\dim(R)=n_0+\dim(T)= n_0 + n_1 +
\cdots + n_k -k$.

A \textit{cubical closed necklical set} is a functor $K:
{Nec^{_\Box}_c}^{op} \rightarrow Set$ and a morphism of cubical closed
necklical sets is given by a natural transformation of functors. Denote
 by $Set^{_\Box}_{Nec_c}$ the category of cubical closed necklical sets.
A cubical set $X$   gives rise to an  example of a cubical closed
necklical set $K_X: {Nec^{_\Box}_c}^{op} \to Set$
 via  the assignment $K_X(R)= Hom(R,X)$, the set of all \emph{cubical set
 maps} from $R$ to $X$.

 Now we describe a useful set of generators for the morphisms in
 $Nec^{_\Box}_c$ similar to those described for $Nec^{_\Box}$ in
 Proposition \ref{maps1}.

\begin{proposition}\label{maps2}

Any non-identity morphism in $Nec^{_\Box}_c$  is a composition of
morphisms of the following type:

\begin{itemize}
\item[(i)] injective morphisms $f: R \to R'$ of cubical closed
    necklaces such that $\dim(R') - \dim(R)=1,$  $R$ and $R'$ have
    the same number of vertices, $b(R)-b(R^\prime)=1,$
 and $f$ preserves the last beads of the first sections;
\item[(i')] injective morphisms $f:R \to R'$ of cubical closed
    necklaces such that $\dim(R') - \dim(R)=1,$
 $R$ and $R'$ have the same number of vertices,
 $b(R)-b(R^\prime)=1,$
 and $f$ maps the last bead of  $R$ into the first bead of $R';$
\item[(i'')]  injective morphisms $f:R \to R'$ of cubical closed
    necklaces such that $\dim(R') - \dim(R)=1,$
 $R$ and $R'$ have the same number of vertices,
 $b(R)-b(R^\prime)=1,$ and $f|_{T_2}: T_2\rightarrow T'_2$ is map of
 type  (i)
 in Proposition \ref{maps1};

\item[(ii)] morphisms $f_{p,j}:  I^{n_0} \vee I^{n_1} \vee ...\vee
I^{n_p+1}\vee ... \vee I^{n_k}
    \to I^{n_0}  \vee I^{n_1} \vee ...\vee I^{n_p}\vee... \vee I^{n_k}$
    of cubical closed necklaces where
    \[f_{p,j}=  id \vee... \vee id \vee \eta^j\vee id\vee ... \vee id,\,\,
     0\leq p\leq k\] and
    $\eta^j: I^{n_p+1} \to I^{n_p}$ is a cubical
    co-degeneracy morphism for $n_0\geq 0 $ with $1\leq j\leq n_0+1 $  and  $1\leq j\leq n_p$  for  $p,\, n_p\geq 1;$

\item [(ii')]  morphisms
 \begin{multline*}
\hspace{0.35in} f_p:I^{n_0}\vee I^{n_1} \vee ...\vee I^{n_p}  \vee... \vee  I^{n_k} \to \\
    I^{n_0}\vee I^{n_1} \vee ...\vee I^{n_{p-1}} \vee I^{n_{p+1}} \vee... \vee  I^{n_k},
    \end{multline*}
    $  k\geq2,,p\geq 1,$
     of cubical closed necklaces such that $f_p$ collapses the $p$-th bead $I^{n_p}$  in the domain to the last vertex of the $(p-1)$-th bead in the target and the restriction of $f$ to all the other beads is identity.

\end{itemize}
\end{proposition}

\begin{remark}\label{freetypes}
\normalfont
1.  Morphisms of type $(i)$
are of the form (cf. Remark \ref{cartesiandecomp})
\begin{multline*}
\delta^{A_i|B_i}: I^{n_0}\vee I^{n_1} \vee ... \vee I^{n_i}\vee I^{n_i+1}  \vee... \vee  I^{n_{k}}
  \rightarrow\\
 I^{m_0}\vee I^{m_1} \vee ...\vee    I^{m_i} \vee... \vee  I^{m_{k-1}}
 \end{multline*}
for
$\delta^{A_i|B_i}:= id\vee ...\vee id \vee S^{A_i|B_i} \vee id\vee...\vee id$ with $ 0\leq i<k$
and
\[
A_i|B_i\in         \left\{
              \begin{array}{lll}
       \overline{P}'(m_0), & m_0=n_0+n_1 ,   &i=0 \vspace{1mm}\\
                  P(m_i),  &  m_i=n_i+n_{i+1}, &   1 \leq i<k,
              \end{array}
            \right.
\]
Thus
 there are exactly  $\underset{0\leq i< k_1}{\sum}\kappa(m_i)+1$
morphisms
$\delta^{A_i|B_i}: R \to R^\prime.$

\medskip
 \noindent
 2.  Morphisms of type $(i')$
 are of the  form
 \begin{multline*}
\delta^{C|D}_{op}: I^{n_0}\vee I^{n_1}   \vee... \vee  I^{n_{k}}
\xrightarrow{\approx}
 I^{n_{k}}\vee I^{n_0} \vee I^{n_1} \vee ...  \vee  I^{n_{k-1}}
\xrightarrow{ S^{C|D}\vee \,id }\\
  I^{m_0}\vee I^{n_1} \vee ... \vee  I^{n_{k-1}}
 \end{multline*}
for $C|D\in  \overline{P}''(m_{0}), \,  m_{0}=n_0+n_{k} .$
Thus,
 there are exactly  $\kappa(m_0)+1$
morphisms
$\delta^{C|D}_{op}: R \to R^\prime.$
\end{remark}

\vspace{0.2in}

\unitlength 0.7mm 
\linethickness{0.4pt}
\ifx\plotpoint\undefined\newsavebox{\plotpoint}\fi
\begin{picture}(64.39,56.434)(0,0)
\put(21.125,20.625){\circle*{1.031}}
\put(19.875,42){\circle*{1}}
\put(63.875,32.75){\circle*{1.031}}
\put(41.875,55.875){\circle*{1.118}}
\put(60.875,21.25){\circle*{1}}
\put(40.875,9.75){\circle*{1}}
\multiput(19.875,42.375)(.054115854,-.033536585){164}{\line(1,0){.054115854}}
\put(28.75,36.875){\line(0,-1){10.5}}
\put(28.75,26.375){\line(-4,-3){7.5}}
\multiput(28.625,36.875)(-.047619048,-.033730159){189}{\line(-1,0){.047619048}}
\multiput(19.625,30.5)(.03289474,-.25328947){38}{\line(0,-1){.25328947}}
\multiput(19.5,30.375)(-.047155689,.033682635){167}{\line(-1,0){.047155689}}
\multiput(11.625,36)(.044642857,.033653846){182}{\line(1,0){.044642857}}
\multiput(11.625,35.75)(.03289474,-.24671053){38}{\line(0,-1){.24671053}}
\multiput(12.875,26.375)(.046428571,-.033571429){175}{\line(1,0){.046428571}}
\multiput(61.625,42.5)(.03373016,-.15079365){63}{\line(0,-1){.15079365}}
\multiput(63.75,33)(-.03370787,-.13061798){89}{\line(0,-1){.13061798}}
\multiput(19.875,41.875)(.0328947,-.4868421){19}{\line(0,-1){.4868421}}
\multiput(28.625,26.5)(-.042467949,.033653846){156}{\line(-1,0){.042467949}}
\multiput(21.5,32.25)(-.0460526,.0328947){19}{\line(-1,0){.0460526}}
\multiput(12.875,26.25)(.036347518,.033687943){141}{\line(1,0){.036347518}}
\multiput(18.5,31.5)(.03571429,.03316327){49}{\line(1,0){.03571429}}
\multiput(19.75,42.375)(.033738938,.038163717){226}{\line(0,1){.038163717}}
\multiput(27.375,51)(.097315436,.033557047){149}{\line(1,0){.097315436}}
\multiput(41.875,56.25)(.0734375,-.03359375){160}{\line(1,0){.0734375}}
\multiput(53.625,50.875)(.033713693,-.034751037){241}{\line(0,-1){.034751037}}
\multiput(21.125,20.75)(.033653846,-.033653846){234}{\line(1,0){.033653846}}
\multiput(29,12.875)(.126344086,-.033602151){93}{\line(1,0){.126344086}}
\multiput(40.875,9.875)(.117788462,.033653846){104}{\line(1,0){.117788462}}
\multiput(53.125,13.375)(.033695652,.035326087){230}{\line(0,1){.035326087}}
\multiput(19.75,41.875)(.10218254,.033730159){126}{\line(1,0){.10218254}}
\multiput(32.625,46.125)(.0336363636,.0359090909){275}{\line(0,1){.0359090909}}
\multiput(41.875,56)(.033608491,-.046580189){212}{\line(0,-1){.046580189}}
\multiput(49,46.125)(.118055556,-.033564815){108}{\line(1,0){.118055556}}
\multiput(21.25,20.625)(.16197183,-.0334507){71}{\line(1,0){.16197183}}
\put(32.75,18.25){\line(0,1){.125}}
\multiput(32.75,18.375)(.0336734694,-.0352040816){245}{\line(0,-1){.0352040816}}
\multiput(41,9.75)(.033675799,.039383562){219}{\line(0,1){.039383562}}
\multiput(48.375,18.375)(.14389535,.03343023){86}{\line(1,0){.14389535}}
\put(32.5,46.25){\circle*{1}}
\put(49,46){\circle*{1}}
\put(53.625,50.75){\circle*{1}}
\put(48.5,18.5){\circle*{1}}
\put(53.125,13.375){\circle*{1}}
\put(33,18.25){\circle*{1}}
\put(29.125,12.875){\circle*{1}}
\put(28.625,36.875){\circle*{1}}
\put(28.875,26.625){\circle*{1}}
\put(20.375,33){\circle*{1}}
\put(11.75,36){\circle*{1}}
\put(13,26.375){\circle*{1}}
\put(19.75,30.375){\circle*{1}}
\put(61.5,42.625){\circle*{1}}
\put(27.375,50.875){\circle*{1}}
\end{picture}

\vspace{0.1in}

\begin{center}
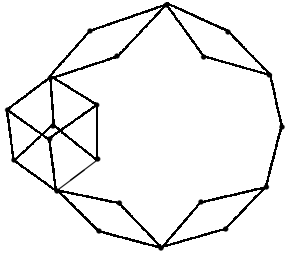
Figure 2. A cubical closed necklace $\mathbf{I}^3\vee I^2\vee I^2\vee
I^1\vee I^1\vee I^2\vee I^2$ of dimension $7.$
\end{center}

\vspace{0.1in}

\section{(Closed)  cubical necklaces  and permutahedra}

Here we show that morphisms of (closed)  cubical necklaces  are closely related with the cell structure of permutahedra. We begin with recalling the definition of  permutahedron and its some properties.

\subsection{The permutahedra $P_n$}\label{permu}

The permutahedron $P_{n}$ is the convex hull of
  $n! $ vertices
  $( \sigma(1),...,\sigma(n))\! \in \!\mathbb{R}^{n}$ for
   $\sigma\in S_{n}.$ Let $P(n)$ denote (ordered) partitions of the set
   $\underline{n}=\{1,2,...,n\}.$
As a cellular
complex, $P_{n}$ is an $\left( n-1\right) $-dimensional convex polytope
whose $\left(
n-k\right) $-faces are indexed by  partitions $A_{1}|\cdots|A_{k}\in
P(n).$
One can define the permutahedra inductively as subdivisions of the
standard $n$-cube $I^{n}.$
Assign the label $\underline{1}$ to the single point $P_{1}.$ If
$P_{n-1}$ has been
constructed and $a=A_{1}|\cdots|A_{k}$ is one of its faces, form the
sequence
$\alpha_{\ast}=\left\{ \alpha_{0}=0,
\alpha_{1},\ldots,\alpha_{k-1},\alpha_{k}=\infty\right\} $ where
$\alpha_{j}=\#\left( A_{k-j+1}\cup\cdots\cup A_{k}\right) ,$ $1\leq
j\leq k-1$ and $\#$
denotes cardinality. Define the \emph{subdivision of} $I$
\emph{relative to} $a$ to be
\begin{equation*}
I/\alpha_{\ast}=I_{1}\cup I_{2}\cup\cdots\cup I_{k},
\end{equation*}
where $I_{j}=\left[
1-\frac{1}{2^{\alpha_{j-1}}},1-\frac{1}{2^{\alpha_{j}}}\right] $ and
$\frac{1}{2^{\infty}}=0.$ Then
\begin{equation*}
P_{n}=\bigcup\limits_{a\in P_{n-1}}a\times \, I/\alpha_{\ast}
\end{equation*}
with faces labeled as follows (see Figures 4 and 5):
\begin{equation*}
\begin{tabular}{c|cc}
\textbf{Face of }$\underset{\ }{a\times I/\alpha_{\ast}}$ &
\textbf{Partition of }
$\underline{n}$ &  \\ \hline
&  &  \\
$a\times0$ & $A_{1}|\cdots|A_{k}|n$ &  \\
&  &  \\
$a\times I_{j}$ & $A_{1}|\cdots|A_{k-j+1}\cup n|\cdots|A_{k}.$ &
\\
&  &  \\
$a\times(I_{j}\cap I_{j+1})$ & $A_{1}|\cdots|A_{k-j}|n|A_{k-j+1}|\cdots
|A_{k},$ & $1\leq j\leq k-1$  \\
&  &  \\
$a\times1$ & $n|A_{1}|\cdots|A_{k}$ &
\end{tabular}
\
\end{equation*}

\vspace{0.3in}

\begin{center}
\setlength{\unitlength}{0.0004in}
\begin{picture}
(2975,2685)(3126,-2038) \thicklines \put(3601,239){\line( 1,0){1800}}
\put(5401,239){\line( 0,-1){1800}} \put(5401,-1561){\line(-1, 0){1800}}
\put(3601,-1561){\line( 0,1){1800}}
\put(3601,239){\makebox(0,0){$\bullet$}}
\put(3601,-661){\makebox(0,0){$\bullet$}}
\put(3601,-1561){\makebox(0,0){$\bullet$}}
\put(5401,239){\makebox(0,0){$\bullet$}}
\put(5401,-661){\makebox(0,0){$\bullet$}}

\put(5401,-1561){\makebox(0,0){$\bullet$}}
\put(4500,-680){\makebox(0,0){$123$}}

\put(3300,-1800){\makebox(0,0){$_{1|2|3}$}}

\put(3150,-699){\makebox(0,0){$_{1|3|2}$}}

\put(3300,464){\makebox(0,0){$_{3|1|2}$}}

\put(5800,-1800){\makebox(0,0){$_{2|1|3}$}}

\put(5850,-699){\makebox(0,0){$_{2|3|1}$}}

 \put(5800,464){\makebox(0,0){$_{3|2|1}$}}

\put(3040,-1260){\makebox(0,0){$1|23$}}
 \put(4550,530){\makebox(0,0){$3|12$}}
\put(3040,-111){\makebox(0,0){$13|2$}}
\put(5960,-111){\makebox(0,0){$23|1$}}
\put(5960,-1260){\makebox(0,0){$2|13$}}
 \put(4550,-1890){\makebox(0,0){$12|3$}}
\end{picture}

\vspace{0.1in}

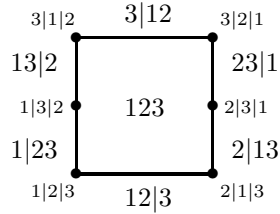
Figure 3. $P_{3}$ as a subdivision of $P_{2}\times I$.
\end{center}

\vspace{0.2in}

\begin{center}
\setlength{\unitlength}{0.00023in}
\begin{picture}
(7500,7500) \thicklines
\put(3000,4800){\line( 0,-1){4800}}
\put(3000,4800){\makebox(0,0){$\bullet$}}
\put(3000,2400){\makebox(0,0){$\bullet$}} \put(3000,0){\makebox
(0,0){$\bullet$}}
\put(3000,3600){\makebox(0,0){$\bullet$}}
\put(3000,0){\line( 1, 0){4800}}
\put(7800,0){\line( 0, 1){4800}}
\put(7800,4800){\line(-1, 0){4800}}
\put(7800,4800){\makebox(0,0){$\bullet$}}
\put(7800,3600){\makebox(0,0){$\bullet$}} \put(7800,2400){\makebox
(0,0){$\bullet$}}
\put(7800,0){\makebox(0,0){$\bullet$}}
\put(3000,2400){\line( 1, 0){4800}}
\put(0,6800){\line( 0,-1){4800}} \put(0,6800){\makebox(0,0){$\bullet$}}
\put(0,5600){\makebox(0,0){$\bullet$}}
\put(0,4400){\makebox(0,0){$\bullet$}}
\put(0,2000){\makebox(0,0){$\bullet$}}
\put(0,6800){\line( 1, 0){4800}}
\put(0,5600){\line( 1, 0){1200}} \put(2000,5600){\line( 1, 0){2800}}
\put(0,2000){\line( 1, 0){1200}} \put(1800,2000){\line( 1, 0){900}}
\put(3300,2000){\line( 1, 0){1500}}
\put(4800,5000){\line( 0,1){1700}} \put(4800,2600){\line( 0,1){2050}}
\put(4800,2000){\line( 0,1){200}}
\put(4800,6800){\makebox(0,0){$\bullet$}}
\put(4800,5600){\makebox(0,0){$\bullet$}} \put(4800,4400){\makebox
(0,0){$\bullet$}}
\put(4800,2000){\makebox(0,0){$\bullet$}}
\put(0,2000){\line( 3,-2){3000}}
\put(3000,4800){\line(-3,2){3000}}
\put(1500,5800){\line( 0,-1){4800}}
\put(1500,5800){\makebox(0,0){$\bullet$}}
\put(1500,4600){\makebox(0,0){$\bullet$}} \put(1500,3400){\makebox
(0,0){$\bullet$}}
\put(1500,1000){\makebox(0,0){$\bullet$}}
\put(3000,3600){\line(-3, 2){1500}}
\put(1500,3400){\line(-3, 2){1500}}
\put(6300,3400){\line(-3, 2){1500}}
\put(6300,5800){\line( 0,-1){800}} \put(6300,4600){\line( 0,-1){2000}}
\put(6300,2180){\line( 0,-1){1200}}
\put(6300,5800){\makebox(0,0){$\bullet$ }}
\put(6300,3400){\makebox(0,0){$\bullet$}} \put(6300,1000){\makebox
(0,0){$\bullet$}}
\put(6300,4600){\makebox(0,0){$\bullet$}}
\put(7800,3600){\line(-3, 2){1500}}
\put(4800,2000){\line(3, -2){3000}}
\put(4800,6800){\line( 3,-2){3000}}
 \put(5100,7300){\makebox(0,0){$_{4|3|2|1}$}}
 \put (2700,-500){\makebox(0,0){$_{1|2|3|4}$}}
 \put (8200,-500){\makebox(0,0){$_{2|1|3|4}$}}

\put (5600,-300){\makebox(0,0){$_{12|3|4}$}}

\put (7100, 1700.00){\makebox(0,0){$_{2|134}$}}

\put (7190, 4500.00){\makebox(0,0){$_{24|13}$}}

\put (5600, 3000.00){\makebox(0,0){$_{23|14}$}}

\put (5600, 5300.00){\makebox(0,0){$_{234|1}$}}
\end{picture}
\vspace{0.3in}

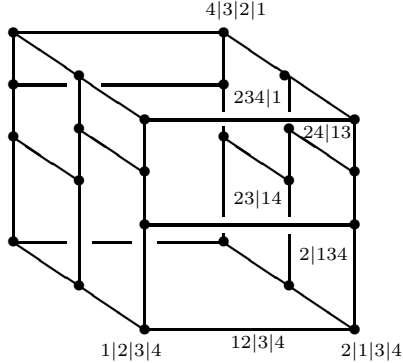
Figure 4. $P_{4}$ as a subdivision of $P_{3}\times I.$
\end{center}

\vspace{0.1in}

A \emph{cubical }vertex of $P_{n}$ is a vertex common to both $P_{n}$
and $
I^{n-1}.$ Note that $a$ is a cubical\ vertex of $P_{n-1}$ if and only
if $
a|n $ and $n|a$ are cubical  vertices of $P_{n}.$ Precisely, $a$ is of
the form
 $a=a_1|...|a_{i-1}|1|a_{i+1}|...|a_{n}  $ with $a_1>\cdots >a_{i-1}$
 and $a_{i+1}<\cdots <a_{n}.$

\vspace{0.2in}

\subsection{The cellular projection $\varsigma_n$}\label{varsigman}

To define the model of the  free loop fibration (cf. Theorem \ref{freeloopmodel})
we need to fix a cellular projection
\begin{equation}\label{varsigma}
\varsigma_n: P_{n+1}\rightarrow I^n
\end{equation}
as follows. Given a vertex
 $a=a_1|...|a_i|1|a_{i+1}|...|a_{n}\in P_{n+1},$ let
  \[\varsigma_n(a)=   b_1|...|b_i|1|b_{i+1}|...|b_{n}\]
   be a
cubical vertex with
 $b_1>\cdots >b_i$ obtained from the set $a_1,...,a_i$ by ordered it decreasingly   and $b_{i+1}<\cdots <b_{n}$ from the set $a_{i+1},...,a_n$   ordered it increasingly.   In particular,  the cells $  1|(\underline{n+1}\smallsetminus 1) $  and $ (\underline{n+1}\smallsetminus 1) | 1 $ are degenerate  as well as all codim 1 cells unless
the codim 1  cells   of the form
 $  a_i|(\underline{n+1}\smallsetminus a_i) $  and $ (\underline{n+1}\smallsetminus a_i) | a_i $  for $a_i\neq 1.$ Precisely,
 for a  $k$ -- subcube $ u\subset I^n,\,   u=d^{\epsilon}_C(I^n),\, C=\{c_1,...,c_{n-k}\}, $ there is a unique cell $u(\epsilon)\subset P_{n+1}$
 with $\varsigma_n(u(\epsilon))=u,$ where
 \begin{equation}\label{cubcell}
    u(\epsilon)=
 \begin{cases}
 \{1\}\cup  (\underline{n}\smallsetminus C) \mid c_1+1|\cdots |c_{n-k}+1 , & \epsilon=0, \\
   c_1+1|\cdots|c_{n-k}+1\mid \{1\}\cup  (\underline{n}\smallsetminus C), & \epsilon=1.
 \end{cases}
 \end{equation}
Furthermore, let  $ (I^n)^{I}$  be  the space of
all continues  maps from  the interval $I=[0,1]$ to the $n$-cube $I^n,$ and let
$P_{0,n}(I^n) \subset (I^n)^{I}$ be the subspace of maps
 with  $f(0)=\min I^n$ and  $f(1)=\max
I^n.$ Fix a cellular homeomorphism $\chi :P_{n}\times I\rightarrow P_{n+1}$
such that $\chi(P_n\times 0)=1|\underline{n}$   and $\chi(P_n\times 1)=\underline{n}|1 .  $
Then  by the exponential law  the composition
  $ \varsigma_n\circ \chi : P_n\times I\rightarrow I^n$
   induces  a map
\begin{equation}\label{omegan}
\omega_{n}: P_{n}\rightarrow  P_{0,n}(I^n) \subset (I^n)^{I}
\end{equation}
in which  $\omega_1(P_1)$ is the identity $I^1 \rightarrow I^1.$

\subsection{The diagonal of permutahedra}\label{DP}

Here we describe a combinatorial diagonal of permutahedra (cf. \cite{SU})
\[  \Delta_P: P_n \rightarrow P_n\times P_n    .\]
Given  a cell $e\subset P_n,$ denote the
set of vertices of $e$ by $\mathcal{V}_{e}.$ Hence
$\mathcal{V}_{e}\subset S_{n}.$ As a vertex $v$ of the standard  $n$ -- cube
defines a unique component $a\times b$  of the cubical diagonal   by $\max a=v=\min b,\, |a|+|b|=n,$
a vertex $\sigma \in \mathcal{V}_{e}$ determines a unique subset $A_\sigma\times B_\sigma$ of components of the diagonal
$\Delta_{P}(e)$ called  \emph{Complementary Pairs} (CP's). Namely,
let $\sigma=x_{1}|\cdots|x_{n}.$ Think of $\sigma$ as an ordered
sequence
of positive integers, construct two elements  $\overleftarrow{\sigma_1}
|\cdots|\overleftarrow{\sigma_p}$ and $\overrightarrow{\sigma_q}
|\cdots|\overrightarrow{\sigma_1}$ of $P(n)$ where
 $\overleftarrow{\sigma_j}$  and $\overrightarrow{\sigma_i}$
denote  $j^{th}$ decreasing and $i^{th}$ increasing subsequence of
maximal length of $\sigma$ respectively.
First,  form
the \emph{Strong Complementary Pair }(SCP)
\[
a_{\sigma}\times b_{\sigma}:=\overleftarrow{\sigma_1}|\cdots
|\overleftarrow{\sigma_p}\times\overrightarrow{\sigma_q}|\cdots
|\overrightarrow{\sigma_1}\in A_\sigma\times B_\sigma.
\]
Then proceed  as follows. Let $a=A_{1}|\cdots|A_{p}\in P(n)$. For $1\leq j<p,$ choose a subset
$M_{j}\subseteq(A_{j}\smallsetminus\{\min A_{j}\})$ such that $\min
M_{j}>\max
A_{j+1}$ when $M_j\neq \varnothing,$ and  define the \emph{right-shift
}$M_{j}$
\emph{action}
\[
R_{M_{j}}(a):=
A_{1}|\cdots|A_{j}\smallsetminus M_{j}|A_{j+1}\cup M_{j}|\cdots|A_{p}
\ \ \text{with} \ \  R_{\varnothing}=Id. \]
Let $\textbf{M}:=(M_{1},M_{2} ,\ldots,M_{p-1}),$  and denote by $R_{\mathbf{M}}\left(a\right)$  the composition  
\[  R_{\mathbf{M}}\left(a\right):={R}_{M_{p-1}}\cdots R_{M_{2}}R_{M_{1}}(a).
\] 

Dually, let $b=B_{q}|\cdots|B_{1}\in P(n).$ For $1\leq i<q,$ choose a
subset
$N_{i}\subseteq(B_{i}\smallsetminus\linebreak\{\min B_{i}\})$ such that
$\min N_{i}>\max
B_{i+1}$ when $N_{i}\neq\varnothing,$   and define the \emph{left-shift
}$N_{i}
$\emph{ action}
\[
L_{N_{i}}(b):=
B_{q}|\cdots|B_{i+1}\cup N_{i}|B_{i}\smallsetminus N_{i}|\cdots|B_{1}\ \
\text{with} \ \  L_{\varnothing}=Id.
\]
Let \textbf{N}$:=(N_{1},N_{2}
,\ldots,N_{q-1}),$ and denote by    $L_{\mathbf{N}}\left(  b\right) $   the composition 
\[
L_{\mathbf{N}}\left(  b\right):= L_{N_{q-1}}\cdots L_{N_{2}}
L_{N_{1}}(b).
\]
Define
\[
A_{\sigma}\times B_{\sigma}=\bigcup\limits_{\mathbf{M,N}}\left\{
{R}_{\mathbf{M}}(a_{\sigma})\times L_{\mathbf{N}}\left(
b_{\sigma}\right)
\right\},
\]
and then
\begin{equation}\label{DeltaSet}
\Delta_{P}(e)=\bigcup_{\sigma\in\mathcal{V}_{e}}A_{\sigma}\times
B_{\sigma}.
\end{equation}
For example, on the top dimensional cell $e^{2}$ of $P_{3}$,
$\Delta_{P}(e^{2})$ is the
union of
\[
\hspace{-0.1in}
\begin{array}
[c]{ll}%
A_{1|2|3}\times B_{1|2|3}=\left\{  1|2|3\times123\right\}  , &
A_{1|3|2}\times
B_{1|3|2}=\left\{  1|23\times13|2\right\}  ,\\
A_{2|1|3}\times B_{2|1|3}=\left\{  12|3\times2|13,\text{ }12|3\times
23|1\right\}  , & A_{2|3|1}\times B_{2|3|1}=\{2|13\times23|1\},\\
A_{3|1|2}\times B_{3|1|2}=\{13|2\times3|12,\text{ }1|23\times3|12\}, &
A_{3|2|1}\times B_{3|2|1}=\{123\times3|2|1\}.
\end{array}
\]

To lift the diagonal on the chain level, let
$(C_*(P_n), d)$ be
 the cellular chain complex of $P_n$ where $d$ is defined for the top
 cell   $e^{n-1}$ by
\begin{equation}\label{Psign}
 d(e^{n-1})=
\sum_{A|B\in P(n) }(-1)^{\# A} sgn(A,B)\,\, A|B,
\end{equation}
and for  proper cells  $A_1|...|A_k\subset P_n$ is extended as a
derivation
\[d(A_1|...|A_k)=\sum_{1\leq r \leq k} -(-1)^{\#(A_1\cup...\cup A_{r-1})+r}
A_1|...|A_{r-1}|d(A_r)|A_{r+1}|...|A_k . \]
Then (\ref{DeltaSet}) induces the coproduct
$\Delta_P:C_*(P_n)\rightarrow C_*(P_n)\otimes C_*(P_n) $ by
\begin{equation}
\label{DeltaP}
\Delta_P(e)=\sum_{\substack{(a,b)\in ( A_\sigma,B_\sigma ) \\ \sigma\in
\mathcal{V}_\sigma  }} sgn( a,b)\,\,   a\otimes b.
\end{equation}
Note that if $e=P_{n_1}\times \cdots \times P_{n_k}$ is a proper cell of $P_n,$
then
$\Delta_P(e)$ is automatically the comultiplicative extension on
the monomial $C_\ast(P_{n_1})\otimes \cdots \otimes C_\ast(P_{n_k}).$
Thus,
$(C_\ast(P_n), d,\Delta_P)$ is a dg (non-coassociative) coalgebra.

\subsection{Comparison of the diagonals of permutahedra and cube via the projection $\varsigma_n$}\label{cubicalc}
 The cellular projection $\varsigma_n:P_{n+1}\rightarrow I^n$
given by (\ref{varsigma}) induces the map of dg coalgebras
\[   C( \varsigma_n ): C_\ast(P_{n+1})\rightarrow C_\ast(I^n).   \]
More precisely,  for any component
$u\otimes v\in \Delta_{_{\Box}}(I^n)$ the pair $u(0)\otimes v(1)$ given by
(\ref{cubcell})
is automatically  SCP, and, hence, is a component of
$ \Delta_P(P_{n+1}).$

\subsection{The two kinds of correspondence between morphisms of
cubical necklaces and cells of permutahedra}

The above combinatorial description of $P_n$ immediately implies the
following propositions. Let $P(A)$ denote the set of
partitions of a finite set $A.$
\begin{proposition}\label{basedloops}
 For  $k,n\geq2,$ there is a canonical bijection $g_\Omega$ between
 the injective morphisms
of cubical  necklaces $f: T=I^{n_1}\vee...\vee I^{n_{k}}\to I^n=
T^\prime$ and the $(n-k)$-dimensional cells  of $P_{n}.$
\end{proposition}
\begin{proof}
We have  $f=\delta^{A_{k-1}|B_{k-1}}\circ \cdots \circ
\delta^{A_1|B_1}$ where     $A_1|B_1\in P(n)$  and    $A_i|B_i\in
P(B_{i-1})$  for $i\geq 2.$
The map $g_\Omega$  is defined by
\[  g_\Omega(f) = A_1| \cdots |A_{k-1}| B_{k-1} \subset P_{n}. \]
\end{proof}
In particular, for $k=2$   and $A|B\in P(n),$  there is the bijection
 (see Figure 5)
\[
g_\Omega: \{\delta^{A|B} \}  \longleftrightarrow \{\text{codimension
1 cells}\
A|B\ \text{of}\ \ P_{n} \}.
\]
Given  a subset  $A=(a_1,...,a_m)\subset \underline{n}$  for $1\leq m\leq n$ and an integer $k,$   denote $A+k:=(a_1+k,..,a_m+k).$    We also have 
\begin{proposition} \label{freeloops}
For $n>n_0\geq 0$ and necklaces $T_1$ with $b(T_1)=k$  and $T_2$ with $b(T_2)=\ell,$
let  $f:I^{n_0}\vee T_1\vee I^1\vee T_2 \rightarrow I^n\vee I^1$  be a morphism of closed necklaces  (where $k=0$ when $T_1=\varnothing $ and $\ell=0,$ when $T_2=\varnothing $).
 Then
there is a canonical bijection $g_\Lambda$
between morphisms $f$
 and  $(n-k-\ell)$-dimensional cells  of
$P_{n+1}.$
\end{proposition}

\begin{proof}
A  map $f$ factors as the composition
\[
f=  \delta^{C_{\ell}|D_{\ell}}_{op}\circ \cdots \circ
\delta^{C_{1}|D_{1}}_{op}\circ \delta^{A_{k}|B_{k}}\circ \cdots
\circ \delta^{A_1|B_1}\ \ \text{with}\ \  C_1|D_1\in \overline{P}''(A_k) \]
and then
  \begin{equation}\label{fbi}
   g_\Lambda(f)=
  (C_{1}+1)|\cdots  | (C_{\ell}+1)|(D_{\ell}+1)\cup 1\mid
  ( B_{k}+1)| \cdots |(B_{1}+1)\in P_{n+1}.
  \end{equation}
  \end{proof}
\noindent In particular,  for $k+\ell=1$   and $A|B\in \overline P(n),$ there is the
bijection (see Figure 6)
\begin{multline*}
  g_\Lambda: \{ \delta^{A|B}\cup \delta_{op}^{C|D} \}
\longleftrightarrow\\
 \{\text{codimension 1 cells}\ \
   (A+1)|( B+1)\cup 1\,\bigcup \,(A+1)\cup 1\mid  (B+1) \  \  \text{of}\ \ P_{n+1} \}
\end{multline*}
given by
\[
\begin{array}{llllllll}
g_\Lambda\left(\delta^{A|B}\right) = \left\{\begin{array}{llll}
                  (A+1)\cup 1\mid (B+1) &\subset P_{n+1}, & A|B \in P(n),
                  \\
                 1 \mid (\underline{n}+1) & \subset
                 P_{n+1}, & A|B=\varnothing|\,\underline{n},
               \end{array}
               \right.
              \vspace{2mm} \\
            \!\!   g_\Lambda\left({\delta}^{C|D}_{op}\right)=\left\{
\begin{array}{llll}
    ( C+1)\mid (D+1)\cup 1      & \subset P_{n+1} , &C| D\in P(n), \\
   \, (\underline{n}+1)\mid 1   & \subset P_{n+1}, &
   C|D =\underline{n}\,|\varnothing.
\end{array}
\right.
\end{array}
\]

\newpage
 \hspace{-1.0in}
\unitlength 1mm 
\linethickness{0.4pt}
\ifx\plotpoint\undefined\newsavebox{\plotpoint}\fi
\begin{picture}(129.625,85.75)(0,0)
\put(35.125,45){\line(0,1){22.75}}
\put(94,6.25){\line(0,1){22.75}}
\put(35.125,67.75){\line(1,0){27.75}}
\put(94,29){\line(1,0){27.75}}
\put(35.125,83.375){\line(1,0){27.75}}
\put(94,56.625){\line(1,0){27.75}}
\put(62.875,67.75){\line(0,-1){22.375}}
\put(121.75,29){\line(0,-1){22.375}}
\put(62.875,45.375){\line(-1,0){27.625}}
\put(121.75,6.625){\line(-1,0){27.625}}
\put(35.25,45.375){\circle*{1.118}}
\put(94.125,6.625){\circle*{1.118}}
\put(35.25,67.625){\circle*{1}}
\put(94.125,28.875){\circle*{1}}
\put(35.25,83.25){\circle*{1}}
\put(94.125,56.5){\circle*{1}}
\put(94.125,82.5){\circle*{1}}
\put(62.75,67.625){\circle*{1}}
\put(121.625,28.875){\circle*{1}}
\put(62.75,83.25){\circle*{1}}
\put(121.625,56.5){\circle*{1}}
\put(62.875,45.5){\circle*{1}}
\put(121.75,6.75){\circle*{1}}
\put(121.75,18.25){\circle*{1}}
\put(94,17.5){\circle*{1.031}}

\put(108.2,3.375){\makebox(0,0)[cc]{$[xx0][11x]$}}
\put(130.8 ,12.125){\makebox(0,0)[cc]{$[0x0][x1x]$}}
\put(130.8 ,23.5){\makebox(0,0)[cc]{$[0xx][x11]$}}
\put(108.3 ,32.5){\makebox(0,0)[cc]{$[00x][xx1]$}}
\put(84.7  ,23.5){\makebox(0,0)[cc]{$[x0x][1x1]$}}
\put(84.7  ,11.625){\makebox(0,0)[cc]{$[x00][1xx]$}}

\put(131.5 ,5.5){\makebox(0,0)[cc]{$_{[0x0][x10][11x]}$}}
\put(131.5,18.125){\makebox(0,0)[cc]{$_{[0x0][0x1][x11]}$}}
\put(131.5,29.5){\makebox(0,0)[cc]{$_{[00x][0x1][x11]}$}}

\put(84.0, 5.5){\makebox(0,0)[cc]{$_{[x00][1x0][11x]}$}}
\put(84.0, 17.5){\makebox(0,0)[cc]{$_{[x00][10x][1x1]}$}}
\put(84.0, 29.5){\makebox(0,0)[cc]{$_{[00x][x01][1x1]}$}}

\put(70.875,56.5){\vector(1,0){10.125}}
\put(70.875,82.875){\vector(1,0){10.125}}
\put(64.0 ,18){\vector(1,0){6.5}}

\put(97.7,  23){\makebox(0,0)[cc]{$13|2$}}
\put(107.3,26){\makebox(0,0)[cc]{$3|12$}}
\put(117.7,23){\makebox(0,0)[cc]{$23|1$}}
\put(117.5,12.5){\makebox(0,0)[cc]{$2|13$}}
\put(107.3,9){\makebox(0,0)[cc]{$12|3$}}
\put(97.7,12.5){\makebox(0,0)[cc]{$1|23$}}
\put(98,17.625){\makebox(0,0)[cc]{$_{1|3|2}$}}
\put(98,27.25){\makebox(0,0)[cc]{$_{3|1|2}$}}
\put(118.3,27.25){\makebox(0,0)[cc]{$_{3|2|1}$}}
\put(118.3,17.5){\makebox(0,0)[cc]{$_{2|3|1}$}}
\put(118.3,8){\makebox(0,0)[cc]{$_{2|1|3}$}}
\put(98,8){\makebox(0,0)[cc]{$_{1|2|3}$}}

\put(94.125,79.625){\makebox(0,0)[cc]{$1$}}
\put(47.875,86.2){\makebox(0,0)[cc]{$I^1$}}
\put(94.375,85.75){\makebox(0,0)[cc]{$P_1$}}

\put(35,70){\makebox(0,0)[cc]{$01$}}
\put(35,80.875){\makebox(0,0)[cc]{$0$}}
\put(62.625,70){\makebox(0,0)[cc]{$11$}}
\put(62.625,80.875){\makebox(0,0)[cc]{$1$}}
\put(108, 59.5){\makebox(0,0)[cc]{$P_2$}}
\put(48.75,57.125){\makebox(0,0)[cc]{$I^2$}}
\put(108, 18.0){\makebox(0,0)[cc]{$P_3$}}
\put(48,47.75){\makebox(0,0)[cc]{$x0$}}
\put(48.125,65.25){\makebox(0,0)[cc]{$x1$}}
\put(38.5 ,56.625){\makebox(0,0)[cc]{$0x$}}
\put(59.5 ,56.625){\makebox(0,0)[cc]{$1x$}}

\put(62.75,43){\makebox(0,0)[cc]{$10$}}

\put(34.75,43){\makebox(0,0)[cc]{$00$}}

\put(94.125,59.5){\makebox(0,0)[cc]{$1|2$}}
\put(94.125,53.5){\makebox(0,0)[cc]{$[x0][1x]$}}
\put(121.375,59.5){\makebox(0,0)[cc]{$2|1$}}
\put(121.375,53.5){\makebox(0,0)[cc]{$[0x][x1]$}}

\put(0,0){}
\put(47.068,17.396){\line(0,-1){13.75}}
\put(47.068,3.646){\line(6,5){12}}
\put(59.068,13.646){\line(0,1){14.25}}
\put(59.068,27.896){\line(-1,0){19.875}}
\put(26.693,3.896){\line(1,0){20.375}}
\put(47.068,17.771){\line(-1,0){20.375}}

\multiput(39.193,27.771)(-.0416666667,-.0336700337){297}
{\line(-1,0){.0416666667}}

\put(26.818,17.771){\line(0,-1){13.75}}
\put(58.943,13.771){\line(-1,0){11.375}}
\put(46.443,13.771){\line(-1,0){7}}
\put(39.443,27.896){\line(0,-1){9.625}}
\put(39.443,17.396){\line(0,-1){3.5}}

\multiput(39.443,13.896)(-.0425084175,-.0336700337){297}
{\line(-1,0){.0425084175}}

\put(46.943,17.646){\circle*{1}}
\put(47.068,3.771){\circle*{1}}
\put(59.068,14.021){\circle*{1}}
\put(59.193,27.771){\circle*{1}}
\put(39.443,13.771){\circle*{1}}
\put(39.318,27.896){\circle*{1}}
\put(26.943,17.771){\circle*{1}}
\put(26.818,4.021){\circle*{1}}
\put(26.318,1.771){\makebox(0,0)[cc]{$000$}}
\put(58.943, 30){\makebox(0,0)[cc]{$111$}}
\put(43.818,21.896){\makebox(0,0)[cc]{$I^3$}}

\multiput(47,17.5)(.039184953,.0336990596){319}
{\line(1,0){.039184953}}

\end{picture}

\vspace{0.3in}

 \begin{center}
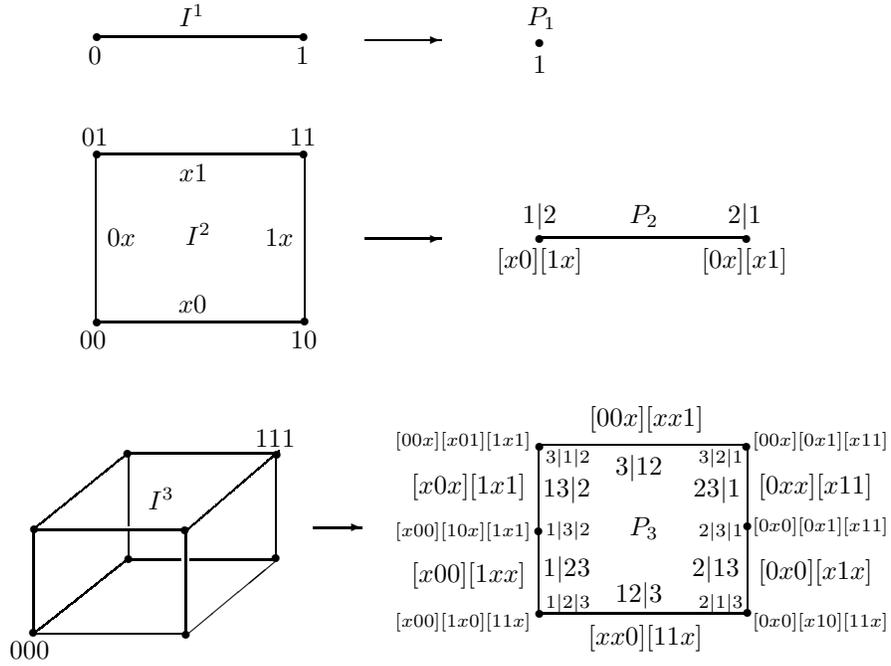
 Figure 5.  The correspondence between the diagonal components of the
 cube
 $I^n:=(x\cdots x)$ (without the primitive terms)
 and   the cells of the permutahedron $[x\cdots x]:=P_n:=\underline{n}$
 for $n=1,2,3.$
\end{center}

\vspace{0.5in}
\hspace{-1.2in}
 \unitlength 1mm 
\linethickness{0.4pt}
\ifx\plotpoint\undefined\newsavebox{\plotpoint}\fi
\begin{picture}(129.625,61)(0,0)
\put(35.125,6.25){\line(0,1){22.75}}
\put(94,6.25){\line(0,1){22.75}}
\put(35.125,29){\line(1,0){27.75}}
\put(94,29){\line(1,0){27.75}}
\put(35.125,45.125){\line(1,0){27.75}}
\put(94,45.125){\line(1,0){27.75}}
\put(62.875,29){\line(0,-1){22.375}}
\put(121.75,29){\line(0,-1){22.375}}
\put(62.875,6.625){\line(-1,0){27.625}}
\put(121.75,6.625){\line(-1,0){27.625}}
\put(35.25,6.625){\circle*{1.118}}
\put(94.125,6.625){\circle*{1.118}}
\put(35.25,28.875){\circle*{1}}
\put(94.125,28.875){\circle*{1}}
\put(35.25,45){\circle*{1}}
\put(94.125,45){\circle*{1}}
\put(35.25,60.5){\circle*{1}}
\put(94.125,60.5){\circle*{1}}
\put(62.75,28.875){\circle*{1}}
\put(121.625,28.875){\circle*{1}}
\put(62.75,45){\circle*{1}}
\put(121.625,45){\circle*{1}}
\put(62.875,6.75){\circle*{1}}
\put(121.75,6.75){\circle*{1}}
\put(121.75,18.25){\circle*{1}}
\put(94,17.5){\circle*{1.031}}

\put(108.125,3.375){\makebox(0,0)[cc]{$x0][1x]$}}

\put(128.625,12.125){\makebox(0,0)[cc]{$1x][x0]$}}

\put(128.625,23.5){\makebox(0,0)[cc]{$11][xx]$}}

\put(108.125,32.5){\makebox(0,0)[cc]{$x1][0x]$}}

\put(86.625,23.5){\makebox(0,0)[cc]{$0x][x1]$}}

\put(86.625,11.625){\makebox(0,0)[cc]{$00][xx]$}}

\put(129.625,5.5){\makebox(0,0)[cc]{$_{10][1x][x0]}$}}

\put(129.625,18.125){\makebox(0,0)[cc]{$_{11][x0][1x]}$}}

\put(129.625,29.5){\makebox(0,0)[cc]{$_{11][0x][x1]}$}}

\put(86.625,17.5){\makebox(0,0)[cc]{$_{00][ox][x1]}$}}

\put(86.625,29.5){\makebox(0,0)[cc]{$_{01][x1][0x]}$}}

\put(86.625,5.5){\makebox(0,0)[cc]{$_{00][x0][1x]}$}}

\put(70.875,45){\vector(1,0){10.125}}
\put(70.875,60.875){\vector(1,0){10.125}}
\put(98.3 ,23.0){\makebox(0,0)[cc]{$13|2$}}
\put(107.3,26){\makebox(0,0)[cc]{$3|12$}}
\put(117.5,23.0){\makebox(0,0)[cc]{$23|1$}}
\put(117.5,12.5){\makebox(0,0)[cc]{$2|13$}}
\put(107.3,8.5){\makebox(0,0)[cc]{$12|3$}}
\put(98.3 ,12.5){\makebox(0,0)[cc]{$1|23$}}

\put(98.0,17.625){\makebox(0,0)[cc]{$_{1|3|2}$}}
\put(98.0,27.25){\makebox(0,0)[cc]{$_{3|1|2}$}}
\put(118.3,27.25){\makebox(0,0)[cc]{$_{3|2|1}$}}
\put(118.3,17.5){\makebox(0,0)[cc]{$_{2|3|1}$}}
\put(118.3,8.0){\makebox(0,0)[cc]{$_{2|1|3}$}}
\put(98.0,8.0){\makebox(0,0)[cc]{$_{1|2|3}$}}

\put(34.875,64){\makebox(0,0)[cc]{$I^0$}}
\put(34.875,57.75){\makebox(0,0)[cc]{$0$}}

\put(94.125,64){\makebox(0,0)[cc]{$P_1$}}
\put(94.125,57.625){\makebox(0,0)[cc]{$1$}}

\put(47.875,48.0){\makebox(0,0)[cc]{$I^1$}}
\put(35,42.625){\makebox(0,0)[cc]{$0$}}
\put(62.625,42.625){\makebox(0,0)[cc]{$1$}}
\put(108,48.0){\makebox(0,0)[cc]{$P_2$}}

\put(48.75,18.375){\makebox(0,0)[cc]{$I^2$}}
\put(108.0,18.0){\makebox(0,0)[cc]{$P_3$}}

\put(48,9.0){\makebox(0,0)[cc]{$x0$}}
\put(48.125,26.5){\makebox(0,0)[cc]{$x1$}}
\put(38.0 ,17.875){\makebox(0,0)[cc]{$0x$}}
\put(60  ,17.875){\makebox(0,0)[cc]{$1x$}}
\put(35,4.25){\makebox(0,0)[cc]{$00$}}
\put(62.75,4.25){\makebox(0,0)[cc]{$10$}}
\put(62.75,31.5){\makebox(0,0)[cc]{$11$}}
\put(34.75,31.5){\makebox(0,0)[cc]{$01$}}

\put(94.125,  42){\makebox(0,0)[cc]{$0][x]$}}
\put(121.375, 42){\makebox(0,0)[cc]{$1][x]$}}

\put(68.0 ,18){\vector(1,0){6.5}}
\put(94.125,48){\makebox(0,0)[cc]{$1|2$}}
\put(121.375,48){\makebox(0,0)[cc]{$2|1$}}
\end{picture}

\vspace{0.2in}
 \begin{center}
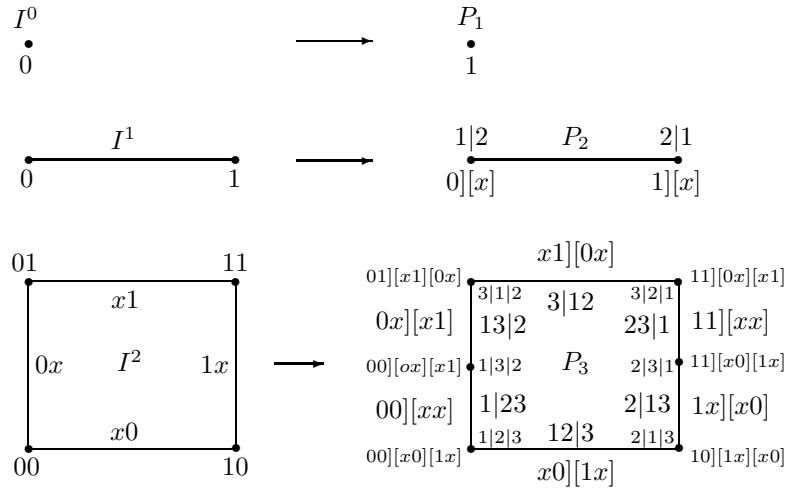
 Figure 6.  The correspondence between the diagonal components of the
 cube $I^n:= (x\cdots x)$
 and   the faces of the permutahedron   $x\cdots
 x]:=P_{n+1}:=\underline{n+1}$ for $n=0,1,2.$
\end{center}

\newpage

 \section{Closed necklical model for the free loop space}

For any cubical set $X$ consider the graded set \[ \Big(\bigsqcup_{R
\in Nec^{_\Box}_c} Hom(R,X) \Big)/\sim\]  where $\sim$ is the
equivalence relation generated by the following rules:
For any $f \in  Hom(R,X)$ and  $f_{p,j}$  and $f_p$
of types (ii)   and  (ii')
in
Proposition \ref{maps2},

\begin{multline}\label{rule1}
f \circ f_{0,n_0+1} \sim f\circ  f_{1,1},\ \
f \circ f_{p,n_p} \sim f\circ  f_{p+1,1},  \,\,\, 1\leq p< k,\ \ \text{and}\\
\ \  f \circ f_{k,n_k} \sim f\circ  f_{0,1},
\end{multline}
and
\begin{equation}\label{rule2}
f \circ f_p \sim f.
\end{equation}
Denote  the equivalence class of $f: R\to X$ by $[f: R \to X].$

\subsection{The closed necklical set $\mathbf{\Lambda} X$} Here we
abuse slightly the language  by calling the object  $\mathbf{\Lambda}
X$ necklical set.
For any cubical set $X$ define a cubical closed necklical set
$\mathbf{\Lambda}X: {Nec^{_\Box}_c}^{op} \to Set$ by declaring
$\mathbf{\Lambda}X(R)$ to be the subset of
 \[\left(\bigsqcup_{R' \in Nec^{_\Box}_c} Hom(R',X) \right)/\sim\]
 consisting of all $\sim$ -- equivalence classes represented by morphisms
 $R \to X \in Nec^{_\Box}_c \downarrow X$. This clearly defines a
 functor: given a morphism $u: R \to R'$ in $Nec^{_\Box}_c$ and an
 element $[f: R' \to X] \in \mathbf{\Lambda} X (R')$ we obtain a well
 defined element  $ [f \circ u: R \to R' \to X ] \in
 \mathbf{\Lambda}X(R)$.  In particular,
  \[\mathbf{\Lambda} X=\{\mathbf{\Lambda}_{n_0,r,k} X\}
_{n_0,r\geq 0,k\geq 1}\] is a trigraded set with
$\mathbf{\Lambda}_{n_0,r,k} X:= \{I^{n_0}\vee T \to X \in Nec^{_\Box}_c
\downarrow X)\mid \dim T=r,\, b(T)=k  \} / \sim .$  But we usually consider
$\mathbf{\Lambda} X$ as bigraded
\[ \mathbf{\Lambda} X=\{\mathbf{\Lambda}_{n,k} X\} \ \ \text{with}\ \  
\mathbf{\Lambda}_{n,k} X=\bigcup_{n=n_0+r}\mathbf{\Lambda}_{n,r,k} X.
 \]
Note that $\mathbf{\Lambda}X$ is precisely the following colimit in the
category of cubical closed necklical sets
\begin{eqnarray*}
\mathbf{\Lambda} X=  \underset{f: R \to X \in
 (Nec^{_\Box}_c \downarrow X) } {\text{colim}} Y(R),
\end{eqnarray*}
where $Y\!\! :\! Nec^{_\Box}_c\! \to Set_{Nec^{_\Box}_c}$ denotes the Yoneda
embedding $Y(R)\!=\! Hom_{Nec^{_\Box}_c}( -\,, R)$.
Analogously we define the cubical
 necklical set $\mathbf{\Omega}(X;x_1,x_2)$
\begin{eqnarray*}
\mathbf{\Omega}(X;x_1,x_2)= \underset{ f: T \to X \in
(Nec^{_\Box}\downarrow X)_{x_1,x_2}} {\text{colim}} Y(T)
\end{eqnarray*}
where $x_1,x_2 \in X_0$  are some vertices, $(Nec^{_\Box} \downarrow
X)_{x_1,x_2}$ denotes the category of maps $f:T \to X$  such that $f$ sends  the first and last vertices of
$T$ to $x_1$  and $x_2$, respectively, and $Y(T)= Hom_{{Nec^{_\Box}}}(-\, , T)$.
When $x_1=x_2=x_0$ is  a base point of $X,$ we simply denote $\mathbf{\Omega} X:=
\mathbf{\Omega}( X; x_0,x_0).$

\subsection{Inverting $1$-cubes formally}

Given a cubical set $(X,d^\epsilon_i, \eta_j),$ form a set

\noindent $X_1^{op}:=\{ x^{op} \mid x \in X_1\ \ \text{is non-degenerate} \}.$
Let $Z(X)$ be the minimal cubical set containing the set
$X\cup X^{op}_1$ such that $d^0_1( x^{op} )= d^1_1( x)$ and
 $d^1_1( x^{op} ) = d^0_1 (x).$
 Denote by $\mathbf{\Lambda}'(Z(X))$ the subset  of $\mathbf{\Lambda}(Z(X))$    such that $f(I^{n_0})\subset X$
   for any $f:I^{n_0}\vee T\rightarrow  Z(X).$
Then define
  \[\widehat{\mathbf{\Lambda}}X:= \mathbf{\Lambda}'(Z(X)) / \sim\]
where the equivalence relation $\sim$ is generated by
\[ f\sim  f'\circ g:I^{n_0}\vee T\rightarrow Z (X) \]
for
 $T= I^{n_{1}}\vee ...\vee I^{n_p}\vee
I^{n_{p+1}}\vee ...\vee I^{n_{k}} $ with
$n_p=n_{p+1}=1,$   $1 \leq p \leq k$, and 
$f(I^{n_p})=(f(I^{n_{p+1}}))^{op},$ so $f$ induces a map
$f':  I^{n_0} \vee I^{n_{1}}\vee ...\vee I^{n_{p-1}}\vee
I^{n_{p+2}}\vee ...\vee I^{n_{k}}\rightarrow X,$
and
\begin{multline*}
g:  I^{n_0} \vee I^{n_{1}}\vee ...\vee I^{n_{p-1}}\vee I^{n_p}\vee
I^{n_{p+1}}\vee I^{n_{p+2}}\vee ...\vee I^{n_{k}}
\rightarrow
\\
 I^{n_0} \vee I^{n_{1}}\vee ...\vee I^{n_{p-1}}\vee I^{n_{p+2}}
\vee ...\vee I^{n_{k}}
\end{multline*}
is the collapse map. 
Similarly, is defined the set $\widehat{\mathbf{\Omega}}X.$

Below we give explicit descriptions of the above cubical (closed) necklical sets
with face and degeneracy operators involved.

\subsection{An explicit construction of $\widehat{\mathbf{\Omega}}X$}

Let $(X,x_0)$ be a pointed cubical set with face and degeneracy maps
denoted by $d^{\epsilon}_i$ and $\eta_j$, respectively. For a cube $\sigma
\in X$ denote by $\min \sigma$ and $\max \sigma$ the first and last
vertices of $\sigma,$ respectively.
 We give  an explicit description of the underlying graded set $\{
 \widehat{\mathbf{\Omega}}_n X \}_{n\geq 0}$ of the cubical necklical
 set $\widehat{\mathbf{\Omega}} X$.

 Let $\overline{ X}=s^{-1} Z(X)_{>0}$ be  the  desuspension of the
 graded set $Z(X)_{>0},$ and $MX$ be the free graded monoid generated
 by $\overline{ X}.$ For $\sigma \in Z(X)_{>0},$  denote $\bar \sigma:=s^{-1}\sigma\in
 \overline X$ with $|\bar \sigma|=| \sigma|-1.$  Let for $n\geq 0,\,$  $k\geq 1$
 and $n_i+1=| \sigma_i|,$   define
\begin{multline*}
{\mathbf{\Omega}}'_{r,k}( X; x_1,x_2)=\{ \bar \sigma_1\cdots \bar \sigma_k\in MX  \mid x_1,x_2\in X_0,\,
\max \sigma_i=\min \sigma_{i+1}\
\text{for all}\ i,\\  \min\sigma_1=x_1,\,   \max\sigma_k= x_2 ,\,\,  
 n_1+\cdots +n_k=r \}.
  \end{multline*}
 Then
\[ \widehat{\mathbf{\Omega}}(X;x_1,x_2)= \{\widehat{\mathbf{\Omega}}_{r,k}(X;x_1,x_2)\}_{r\geq
0,k\geq 1},\ \ \ 
\widehat{\mathbf{\Omega}}_{r,k}(X;x_1,x_2)= \mathbf{\Omega}_{r,k}'(X;x_1,x_2)/\sim \]
where $\sim$ is defined by
\[
\begin{array}{llll}
\bar \sigma_1\cdots  \overline{\eta_{n_i+1}( \sigma_i)}\cdot \bar
\sigma_{i+1}\cdots \bar \sigma_k\,\, \sim\,\,
\bar \sigma_1\cdots  \bar \sigma_{i}\cdot \overline{\eta_1( \sigma_{i+1})}
\cdots \bar \sigma_k, \,\,\,\,
    1\leq i <k,  \hspace{1mm}\\
    \text{and for}\ \ \sigma_{i+1}=\sigma^{op}_i \ \ \text{with}\ \
 \sigma_{i}\in Z(X)_1 \hspace{1mm}\\
    \bar \sigma_1\cdots  \bar \sigma_{i-1}\cdot \bar \sigma_i\cdot \bar
    \sigma_{i+1}\cdot\bar \sigma_{i+2}   \cdots \bar \sigma_k\,\sim\,
\bar \sigma_1\cdots \bar \sigma_{i-1}\cdot \bar \sigma_{i+2}\cdots
\bar \sigma_k\ \  \text{and} \vspace{1mm}\\ 
\bar \sigma_i\cdot \bar \sigma_{i+1}\sim\, \overline{\eta_1( \min \sigma_i)}
.
\end{array}
\]

 The face operators
 \[d_{A_i|B_i}: \widehat{\mathbf{\Omega}}_{r,k}(X;x_1,x_2) X \rightarrow
 \widehat{\mathbf{\Omega}}_{r,k+1}(X;x_1,x_2),\ \ \   1\leq i\leq k, \]
  are  defined
  for
 $  y:=   \bar \sigma_1\cdots \bar \sigma_k$   with   $\sigma_i\in Z(X)_{n_i+1}$  by
\[ d_{A_i|B_i}( y)=\bar \sigma_1\cdots \bar \sigma_{i-1}\cdot \overline{d^0_B(\sigma_i)}\cdot \overline{d^1_A(\sigma_i)}\cdot \bar \sigma_{i+1}\cdots \bar \sigma_{k}, \ \ \
\ A_i|B_i \in P(n_i),  \]
and the degeneracy maps
\[  \varrho_j:  \widehat{\mathbf{\Omega}}_{r,k}(X;x_1,x_2) \rightarrow
 \widehat{\mathbf{\Omega}}_{r,k}(X;x_1,x_2), \ \ \  1\leq j\leq r+k+1,  \]
 are defined by
   \begin{multline*}
\varrho_j(\bar \sigma_1 \cdots  \bar \sigma_k)=
 \bar \sigma_1\cdots \bar \sigma_{p-1}\cdot  \overline{\eta_{i}(
 \sigma_p)}\cdot \bar \sigma_{p+1}\cdots \bar \sigma_k \ \
\text{for}\ \
       j=n_1+\cdots + n_{p-1}+p-1+i.
\end{multline*}
 Then
 $\widehat {\mathbf{\Omega}}_{r,k}X$ is obtained from $\widehat{\mathbf{\Omega}}_{r,k}(X;x_1,x_2)$
 by setting $x_1=x_2=x_0=  \min \sigma_1=\max \sigma_k  $ and   
 $ \varrho_{1}=\varrho_{r+k+1}.$
 Thus  $\widehat {\mathbf{\Omega}}X$
  is  a monoidal permutahedral set with unit $1\in \widehat{\mathbf{\Omega}}_0 X.$
In particular, $ \widehat{\mathbf{\Omega}}_0 X$ is a group.

\subsection{An explicit construction of $\widehat{\mathbf{\Lambda}}X$}

Let
 \[
{\mathbf{\Lambda}}'_{n,k} X \! =\!
\{ (u\,,\, y ) \in \bigcup_{\substack{n_0+r=n\\
x_1,x_2\in X_0}}\!
X_{n_0}\!\times \widehat{\mathbf{\Omega}}_{r,k}(X;x_1,x_2) \mid
\min u=x_2,\, \max u=x_1\}.
\]
 Then define the bigraded set
 \[
 \widehat{\mathbf{\Lambda}} X=\{ \mathbf{\Lambda}_{n,k} X \}_{n,k\geq 0} ,\ \
 \ \
\widehat{\mathbf{\Lambda}} X=\mathbf{\Lambda}' X/\sim
 \]
 where $\sim$ is defined for $(u,y)\in \mathbf{\Lambda}'_{n,k} X$ with $u\in X_{n_0}$ and $y\in \widehat{\mathbf{\Omega}}_{r,k}(X;x_1,x_2)$
via the relations
\[ (\eta_{n_0+1}(u),y)\sim (u,\varrho_1(y))\ \ \text{and}\ \    (u,\varrho_{r+k+1}(y))\sim (\eta_1(u), y). \]
Define   the three types of  face operators
 \[d_{A|B},\, d^{op}_{C|D},\, d_{A_i|B_i}: \widehat{\mathbf{\Lambda}}_{n,k} X
 \rightarrow
 \widehat{\mathbf{\Lambda}}_{n-1.k+1} X \]
  for
  $a:= (u,y)= (u\,,  \bar \sigma_1\cdots \bar \sigma_k)\in \widehat{\mathbf{\Lambda}} X $ with  $u\in X_{n_0}$ and $1\leq i\leq k$
      by
\[
 \begin{array}{llllll}
  d_{A|B}(a)=
   \left( d^0_{B}(u)\, ,\, \overline{d^1_{A}(u)}   \cdot  \bar
     \sigma_1\cdots \bar \sigma_k \right), & A|B\in \overline{P}'(n_0),
 \vspace{1mm}  \\
 d^{op}_{C|D}(a)=
     \left( d^1_{C}(u)\,,\, \bar \sigma_1\cdots \bar \sigma_k  \cdot
     \overline{d^0_{D}(u)}\,  \right),& C|D\in \overline{P}''(n_0),
     \vspace{1mm}\\

 d_{A_i|B_i}(a)= \left(u\, , \, \bar \sigma_1\cdots \bar \sigma_{i-1} \cdot
  \overline{d^0_{B_i}(\sigma_i)}\cdot \overline{d^1_{A_i}(\sigma_i)}\cdot \bar
  \sigma_{i+1}\cdots \bar \sigma_k\right)   , &
   A_i|B_i\in P(n_i),
\end{array}
\]
and
the degeneracy maps
\[  \varrho_j :  \widehat{\mathbf{\Lambda}}_{n-1,k}X \rightarrow
 \widehat{\mathbf{\Lambda}}_{n,k}X \ \ \text{for}\ \  1\leq j\leq n+k+1,
 \]
by
\begin{equation*}
\varrho_j(u,y)=\left\{ \begin{array}{llll} (\eta_{j}(u), y), &
1\leq j\leq n_0+1,\vspace{1mm} \\
                       (u, \varrho_{j-n_0}(y)), & n_0+1< j\leq n+k+1.
\end{array}
\right.
\end{equation*}

\subsection{The geometric realizations of 
$\widehat{\mathbf{\Lambda}} X$  }

Using modelling polytopes as permutahedra $P_n$
the \textit{geometric realization} of the  set
$\widehat{\mathbf{\Lambda}} X $  is
\begin{eqnarray*}
|\widehat{\mathbf{\Lambda}} X|:= \bigsqcup_{n_0\geq0,k\geq1}\widehat{\mathbf{\Lambda}}_{n_0,k} X  \times (P_{n_0+1}\times
P_{n_1}\times\cdots  \times P_{n_k})   / \sim \, ,
\end{eqnarray*}
where  $\widehat{\mathbf{\Lambda}}_{n_0,k} X $ is considered as a topological
space with the discrete topology, and $\sim$ is the equivalence
relation  generated
for  $ a:= (u\,,  \bar \sigma_1\cdots \bar \sigma_k)  \in  \mathbf{\Lambda}_{n_0,k} X $ with    $u\in X_{n_0},\, \sigma_i\in Z(X)_{n_i+1}  $ and
$t\in P_{n_0+1}\times  P_{n_1+1}\times\cdots  \times P_{n_k+1}$ by
\[\left(a, \delta^{A|B} (t)\right )\sim \left(d_{A|B}(a), t\right) \ \
\text{and}\ \
    (a, \varrho^j(t))\sim (\varrho_j(a), t),\]
where 
\begin{multline*}\delta^{A|B}: P_{n_0+1}\times\cdots \times
P_{n_{i}+1}   \times P_{n_{i+1}+1} \times \cdots  \times P_{n_k+1}\rightarrow \\
P_{n_0+1}\times\cdots \times P_{n_{i}+n_{i+1}+2}\times \cdots
\times P_{n_k+1} ,    
\end{multline*}
$\delta^{A|B}=id\times \cdots   \times \iota_{A|B}  \times \cdots   \times id,\,
\iota_{A|B}: P_{n_i+1}\times P_{n_{i+1}+1}\hookrightarrow P_{n_i+n_{i+1}+2},\, 0\leq i<k, $
and for  $j=n_0+n_1+\cdots + n_{p-1}+p +i $
\[\varrho^j
: P_{n_0+1}\times\cdots \times
P_{n_{p}+1}\times \cdots  \times P_{n_k+1}\rightarrow
P_{n_0+1}\times\cdots \times P_{n_p}\times \cdots\times P_{n_k+1},
  \]
 $\varrho^j=id\times \cdots   \times \varrho^i  \times \cdots   \times id,\,$
$\varrho^i: P_{n_p+1}\rightarrow P_{n_p}$ is cellular projection compatible with
 the $i$ -- $th$ standard projection $I^{n_p}\rightarrow I^{n_p-1}$ under the map $\varsigma_{n_p-1}$ given by (\ref{varsigma}).

  Similarly, 
  \begin{eqnarray*}
|\widehat{\mathbf{\Omega}} X|:= \bigsqcup_{r\geq 0;k\geq
1}\widehat{\mathbf{\Omega}}_{r,k} X  \times( P_{n_1+1}\times\cdots  \times
P_{n_k+1}) / \sim, \ \ \ \   r=n_1+\cdots +n_k.
\end{eqnarray*}

\subsection{The quasi-fibration $\zeta$}
We have  the  short sequence
\[
 \widehat{\mathbf{\Omega}} X   \overset{i }{\longrightarrow}
  \widehat{\mathbf{\Lambda}}X   \overset{pr}{\longrightarrow} X
\]
of maps of sets where $i$ is defined  for  $y
\in \widehat{\mathbf{\Omega}} X $ by  $i(y)=(x_0, y),$ 
while $pr(u,y)=u$ for
$(u,y)\in   \widehat{\mathbf{\Lambda}} X.$
The map $i: \widehat{\mathbf{\Omega}} X \to
\widehat{\mathbf{\Lambda}}X$ induces a continuous map
 $ \iota: |\widehat{\mathbf{\Omega}} X|  \to
 |\widehat{\mathbf{\Lambda}}X |$. The projection
$pr : \widehat{\mathbf{\Lambda}}X  {\longrightarrow} X$ together with
the cellular projections $\varsigma_n: P_{n+1}\rightarrow I^{n}$ given
by (\ref{varsigma})
induces a continuous and cellular map $\zeta:
|\widehat{\mathbf{\Lambda}}X|  \to |X|.$

The proof of the following statements are entirely analogous to that in
the simplicial setting \cite{RS2}.

\begin{proposition}\label{quasifree}

For a pointed connected cubical set  $(X,x_0)$  the short sequence
\begin{equation}\label{quasifib}
 |\widehat{\mathbf{\Omega}} X|   \overset{\iota}{\longrightarrow}
  |\widehat{\mathbf{\Lambda}}X|   \overset{\zeta}{\longrightarrow} |X|
\end{equation}
is a quasi-fibration.
\end{proposition}

\begin{theorem}\label{freeloopmodel}
Let $Y=|X|$ be the geometric realization of a path connected cubical
set $X.$
Let $ \Omega Y \overset{i}{\rightarrow} \Lambda Y\overset{\varrho}
\longrightarrow Y$ be  the free loop fibration on $Y.$ There is a
commutative diagram
\begin{equation}\label{freediagram}
\begin{array}{cccccc}
   |\widehat{\mathbf{\Omega}} X| & \overset{\omega}{\longrightarrow}
   &                 \Omega Y \\
        \iota  \downarrow            &                  &
        \hspace{-0.1in} \iota   \downarrow \\
      |\widehat{\mathbf{\Lambda}}   X|  &
      \overset{\Upsilon}{\longrightarrow} &                     \Lambda
      Y                    \\
     \zeta \downarrow    &                                    &
     \hspace{-0.1in}  \varrho     \downarrow           \\
 \hspace{0.1in}|X| & \overset{Id}{\longrightarrow}                    &
 Y
\end{array}
\end{equation}
in which
  $\Upsilon$ and $\omega$ are homotopy equivalences.
\end{theorem}
\begin{proof}

The maps $\Upsilon$ and $\omega$ above are in fact  canonically defined
by means of the cellular projection $\varsigma_n:P_{n+1}\rightarrow I^n$
and the map $\omega_{n}: P_{n}\rightarrow  P_{0,n}(I^n) $
given by (\ref{varsigma})   and (\ref{omegan}), respectively.
Namely, let $y=({t}_1,...,{t}_k)\in |(\bar \sigma_1\cdots \bar \sigma_k )|\subset |\widehat{\mathbf{\Omega}}_{r,k}(X;x_1,x_2)|$  with $|\bar \sigma_i|=n_i;$ then
the maps  $\omega_{n_i}: P_{n_i}\rightarrow  P_{0,n_i}(I^{n_i})  $
   for $1\leq i\leq k$    induce  the map
  \[\omega: |\widehat{\mathbf{\Omega}}_{r,k}(X;x_1,x_2)| \rightarrow  \Omega Y\]
  by
$\omega(y)=\omega_{n_k}({t}_1)* \cdots *
 \omega_{n_1}({t}_k): I\rightarrow Y
$
where $*$ denotes path concatenation in $Y.$ Denote $\lambda_y:=\omega(y),$
a path from $x_1$ to $x_2$ in $Y.$

We now construct $\Upsilon:  |\widehat{\mathbf{\Lambda}}   X|
{\rightarrow}                  \Lambda Y $. Let
$(x,y)\in |(u,\bar \sigma_1\cdots \bar \sigma_k )|\subset |\widehat{\mathbf{\Lambda}} X|,$ and let $|u|=n_0.$
For  $n_0=0$ define  $\Upsilon (x,y)=\lambda_y,$ a loop in $Y$ based at the point $x=x_1=x_2.$
  Let
$n_0 >0,$ and $x=(\mathbf{t},s)\in P_{n_0}\times I\approx
P_{n_0+1}.$  Define
\[\Upsilon(x,y)=\beta\ast \lambda_y \ast \alpha ,\]
where $\alpha$ is  a path from $ \varsigma_{n_0}(\mathbf{t},s)$
to $ \varsigma_{n_0}(\mathbf{t},1)$ in $I^{n_0}$ defined by the
restriction of $ \omega_{n_0}(\mathbf{t})$ to the interval
$[s,1]\subset [0,1]$ and
 $\beta$ is  a path from $ \varsigma_{n_0}(\mathbf{t},0)$ to
 $ \varsigma_{n_0}(\mathbf{t},s)$ in $I^{n_0}$ defined by the
 restriction of $ \omega_{n_0}(\mathbf{t})$ to the interval
 $[0,s]\subset [0,1] .$

 Similarly to  \cite{RS2}  we have that $\omega$ is a homotopy
 equivalence. By Proposition \ref{quasifree} the sequence given by   (\ref{quasifib}) is a quasi-fibraton,
 and, hence
  gives rise to a long exact sequence in
 homotopy groups. To show that $\Upsilon$ is a weak equivalence  it remains to check that $\Upsilon$ induces a bijection of path linear components
 $\Upsilon_0: \pi_0( |\widehat{\mathbf{\Lambda}}X  |)\rightarrow \pi_0(\Lambda Y).$ Recall that
  $\pi_0(\Lambda Y)=\pi_0(\Omega Y)/\sim$
  where the equivalence relation  is generated by $\lambda \sim \mu\lambda \mu^{-1}$ for any
   $\lambda,\mu \in \pi_0(\Omega Y).$  For $x\in \widehat{\mathbf{\Lambda}}_1X $
   and  $\bar y \in \widehat{\mathbf{\Lambda}}_0X $
   the equalities
   \[   d_{\varnothing|\underline{1}}(x,\bar y \bar x^{-1})=  (x_0,\bar x\bar y \bar x^{-1})\  \   \text{and}   \ \
    d_{\underline{1}|\varnothing}(x,\bar y \bar x^{-1})= (x_0 ,\bar y \bar x^{-1}\bar x)= (x_0 , \bar y) \]
    in
    $\widehat{\mathbf{\Omega}}_0 X\subset \widehat{\mathbf{\Lambda}}_0 X  $
    shows that
    \[ \pi_0( \widehat{\mathbf{\Lambda}}X )=\pi_0(\widehat{\mathbf{\Omega}}X ) / \bar y\sim \bar x\bar y \bar x^{-1}. \]
    Since $\omega_0:\pi_0( |\widehat{\mathbf{\Omega}}X  |)\rightarrow \pi_0(\Omega Y)$  is an isomorphism, $\Upsilon_0$  is a bijection.
\end{proof}

\vspace{0.2in}

\hspace{-1.5in}
\unitlength .7mm 
\linethickness{0.4pt}
\ifx\plotpoint\undefined\newsavebox{\plotpoint}\fi 
\begin{picture}(187.302,122.048)(0,0)
\put(94.3,59.827){\line(1,0){25.375}}
\put(143.633,59.827){\line(1,0){25.375}}
\put(119.675,59.827){\line(0,1){21}}
\put(169.008,59.827){\line(0,1){21}}
\put(119.675,80.827){\line(-1,0){25.375}}
\put(169.008,80.827){\line(-1,0){25.375}}
\put(94.3,80.827){\line(0,-1){20.875}}
\put(143.633,80.827){\line(0,-1){20.875}}
\put(119.675,80.702){\circle*{1}}
\put(94.425,80.577){\circle*{1}}
\put(143.758,80.577){\circle*{1}}
\put(116.858,118.853){\circle*{1}}
\put(144.479,119.03){\circle*{1}}
\put(94.3,59.827){\circle*{1}}
\put(143.633,59.827){\circle*{1}}
\put(119.675,59.827){\circle*{1}}
\put(169.008,59.827){\circle*{1}}
\put(94.281,38.692){\line(1,0){22.875}}
\put(159.03,38.411){\line(1,0){22.875}}
\put(94.281,38.442){\circle*{1}}
\put(159.03,38.161){\circle*{1}}
\put(78.531,29.567){\circle*{1}}
\put(143.28,29.286){\circle*{1}}
\put(78.406,12.067){\circle*{1}}
\put(143.155,11.786){\circle*{1}}
\put(102.906,11.942){\circle*{1}}
\put(167.655,11.661){\circle*{1}}
\put(102.906,29.192){\circle*{1}}
\put(167.655,28.911){\circle*{1}}
\put(94.156,20.942){\circle*{1}}
\put(158.905,20.661){\circle*{1}}
\put(117.156,20.942){\circle*{1}}
\put(181.905,20.661){\circle*{1}}
\put(117.281,38.567){\circle*{1}}
\put(182.03,38.286){\circle*{1}}
\put(126.405,69.909){\vector(1,0){9.75}}
\put(126.51,98.527){\vector(1,0){9.75}}
\put(126.097,118.916){\vector(1,0){9.75}}
\multiput(94.281,38.442)(-.084224599,-.048128342){187}{\line(-1,0){.084224599}}
\multiput(159.03,38.161)(-.084224599,-.048128342){187}{\line(-1,0){.084224599}}
\put(78.531,29.442){\line(1,0){24.5}}
\put(143.28,29.161){\line(1,0){24.5}}
\multiput(103.031,29.442)(.076351351,.047972973){185}{\line(1,0){.076351351}}
\multiput(167.78,29.161)(.076351351,.047972973){185}{\line(1,0){.076351351}}
\put(78.531,29.192){\line(0,-1){17.25}}
\put(143.28,28.911){\line(0,-1){17.25}}
\put(102.906,11.817){\line(0,1){17.5}}
\put(167.655,11.536){\line(0,1){17.5}}
\put(94.281,38.567){\line(0,-1){8.625}}
\put(159.03,38.286){\line(0,-1){8.625}}
\put(94.281,28.817){\line(0,-1){8}}
\put(159.03,28.536){\line(0,-1){8}}
\put(94.281,20.817){\line(1,0){8.25}}
\put(159.03,20.536){\line(1,0){8.25}}
\put(117.156,38.317){\line(0,-1){17.375}}
\put(181.905,38.036){\line(0,-1){17.375}}
\multiput(117.156,20.942)(-.075534759,-.048128342){187}{\line(-1,0){.075534759}}
\multiput(181.905,20.661)(-.075534759,-.048128342){187}{\line(-1,0){.075534759}}
\multiput(78.531,11.692)(.078282828,.047979798){99}{\line(1,0){.078282828}}
\multiput(87.156,16.692)(.08584337,.04819277){83}{\line(1,0){.08584337}}
\put(168.155,20.411){\line(1,0){13.625}}
\put(126.115,24.264){\vector(1,0){10}}
\put(130.347,121.166){$_{\Upsilon}$}
\put(130.76,100.777){$_{\Upsilon}$}
\put(167.095,91.786){$_{\lambda}$}
\put(162.05,110.846){$_{\lambda}$}
\put(172.735,57.367){$_{\lambda}$}
\put(180.485,12.564){$_{\lambda}$}
\put(130.655,72.159){$_{\Upsilon}$}
\put(128.865,26.514){$_{\Upsilon}$}
\put(95.572,98.922){\line(1,0){21.443}}
\put(144.344,98.501){\line(1,0){19.971}}
\multiput(144.453,98.191)(.097943,-.047773){8}{\line(1,0){.097943}}
\multiput(146.02,97.426)(.097943,-.047773){8}{\line(1,0){.097943}}
\multiput(147.587,96.662)(.097943,-.047773){8}{\line(1,0){.097943}}
\multiput(149.155,95.897)(.097943,-.047773){8}{\line(1,0){.097943}}
\multiput(150.722,95.133)(.097943,-.047773){8}{\line(1,0){.097943}}
\multiput(152.289,94.369)(.097943,-.047773){8}{\line(1,0){.097943}}
\put(153.072,93.987){\line(1,0){.9461}}
\put(154.964,93.818){\line(1,0){.9461}}
\put(156.857,93.65){\line(1,0){.9461}}
\put(158.749,93.482){\line(1,0){.9461}}
\put(160.641,93.314){\line(1,0){.9461}}
\multiput(162.533,93.146)(.113444,.043381){7}{\line(1,0){.113444}}
\multiput(164.121,93.753)(.113444,.043381){7}{\line(1,0){.113444}}
\multiput(165.71,94.36)(.113444,.043381){7}{\line(1,0){.113444}}
\multiput(167.298,94.968)(.113444,.043381){7}{\line(1,0){.113444}}
\multiput(168.886,95.575)(.113444,.043381){7}{\line(1,0){.113444}}
\multiput(169.68,95.879)(-.115837,.047184){7}{\line(-1,0){.115837}}
\multiput(168.059,96.539)(-.115837,.047184){7}{\line(-1,0){.115837}}
\multiput(166.437,97.2)(-.115837,.047184){7}{\line(-1,0){.115837}}
\multiput(164.815,97.86)(-.115837,.047184){7}{\line(-1,0){.115837}}
\put(95.629,98.767){\circle*{1.061}}
\put(116.842,98.767){\circle*{1.061}}
\put(144.42,98.414){\circle*{1.061}}
\put(164.218,98.414){\circle*{1.061}}
\put(78.545,16.935){\line(1,0){24.386}}
\put(88.356,17.496){\line(0,1){11.773}}
\put(88.356,29.829){\line(0,1){5.045}}
\multiput(88.356,28.427)(-.088387387,-.047972973){111}{\line(-1,0){.088387387}}
\multiput(158.711,20.579)(-.083941176,-.047967914){187}{\line(-1,0){.083941176}}
\put(78.545,11.89){\line(1,0){24.386}}
\put(143.575,11.329){\line(1,0){24.106}}
\multiput(143.194,12.07)(.068583,.045723){12}{\line(1,0){.068583}}
\multiput(144.84,13.167)(.068583,.045723){12}{\line(1,0){.068583}}
\multiput(146.486,14.264)(.068583,.045723){12}{\line(1,0){.068583}}
\multiput(148.132,15.362)(.068583,.045723){12}{\line(1,0){.068583}}
\multiput(149.778,16.459)(.068583,.045723){12}{\line(1,0){.068583}}
\multiput(151.424,17.556)(.068583,.045723){12}{\line(1,0){.068583}}
\multiput(153.07,18.654)(.068583,.045723){12}{\line(1,0){.068583}}
\multiput(154.716,19.751)(.068583,.045723){12}{\line(1,0){.068583}}
\multiput(156.362,20.848)(.068583,.045723){12}{\line(1,0){.068583}}
\multiput(158.008,21.946)(.068583,.045723){12}{\line(1,0){.068583}}
\multiput(159.654,23.043)(.068583,.045723){12}{\line(1,0){.068583}}
\multiput(161.3,24.14)(.068583,.045723){12}{\line(1,0){.068583}}
\multiput(162.946,25.238)(.068583,.045723){12}{\line(1,0){.068583}}
\multiput(164.592,26.335)(.068583,.045723){12}{\line(1,0){.068583}}
\multiput(166.238,27.433)(.068583,.045723){12}{\line(1,0){.068583}}
\multiput(167.884,28.53)(.068583,.045723){12}{\line(1,0){.068583}}
\multiput(169.53,29.627)(.068583,.045723){12}{\line(1,0){.068583}}
\multiput(171.176,30.725)(.068583,.045723){12}{\line(1,0){.068583}}
\multiput(172.822,31.822)(.068583,.045723){12}{\line(1,0){.068583}}
\multiput(174.468,32.919)(.068583,.045723){12}{\line(1,0){.068583}}
\multiput(176.114,34.017)(.068583,.045723){12}{\line(1,0){.068583}}
\multiput(177.76,35.114)(.068583,.045723){12}{\line(1,0){.068583}}
\multiput(179.406,36.211)(.068583,.045723){12}{\line(1,0){.068583}}
\multiput(181.052,37.309)(.068583,.045723){12}{\line(1,0){.068583}}
\multiput(143.754,10.949)(.31144,-.04153){3}{\line(1,0){.31144}}
\multiput(145.623,10.699)(.31144,-.04153){3}{\line(1,0){.31144}}
\multiput(147.492,10.45)(.31144,-.04153){3}{\line(1,0){.31144}}
\multiput(149.36,10.201)(.31144,-.04153){3}{\line(1,0){.31144}}
\multiput(151.229,9.952)(.31144,-.04153){3}{\line(1,0){.31144}}
\multiput(153.098,9.703)(.31144,-.04153){3}{\line(1,0){.31144}}
\multiput(154.966,9.453)(.31144,-.04153){3}{\line(1,0){.31144}}
\multiput(156.835,9.204)(.31144,-.04153){3}{\line(1,0){.31144}}
\multiput(158.704,8.955)(.31144,-.04153){3}{\line(1,0){.31144}}
\multiput(160.572,8.706)(.31144,-.04153){3}{\line(1,0){.31144}}
\multiput(162.441,8.457)(.31144,-.04153){3}{\line(1,0){.31144}}
\multiput(164.31,8.207)(.31144,-.04153){3}{\line(1,0){.31144}}
\multiput(166.178,7.958)(.31144,-.04153){3}{\line(1,0){.31144}}
\multiput(168.047,7.709)(.31144,-.04153){3}{\line(1,0){.31144}}
\multiput(168.981,7.585)(.063264,.045743){12}{\line(1,0){.063264}}
\multiput(170.5,8.682)(.063264,.045743){12}{\line(1,0){.063264}}
\multiput(172.018,9.78)(.063264,.045743){12}{\line(1,0){.063264}}
\multiput(173.536,10.878)(.063264,.045743){12}{\line(1,0){.063264}}
\multiput(175.055,11.976)(.063264,.045743){12}{\line(1,0){.063264}}
\multiput(176.573,13.074)(.063264,.045743){12}{\line(1,0){.063264}}
\multiput(178.091,14.172)(.063264,.045743){12}{\line(1,0){.063264}}
\multiput(179.61,15.269)(.063264,.045743){12}{\line(1,0){.063264}}
\multiput(181.128,16.367)(.063264,.045743){12}{\line(1,0){.063264}}
\multiput(182.646,17.465)(.063264,.045743){12}{\line(1,0){.063264}}
\multiput(184.165,18.563)(.063264,.045743){12}{\line(1,0){.063264}}
\multiput(185.683,19.661)(.063264,.045743){12}{\line(1,0){.063264}}
\multiput(187.201,20.759)(-.046719,.149991){6}{\line(0,1){.149991}}
\multiput(186.641,22.558)(-.046719,.149991){6}{\line(0,1){.149991}}
\multiput(186.08,24.358)(-.046719,.149991){6}{\line(0,1){.149991}}
\multiput(185.519,26.158)(-.046719,.149991){6}{\line(0,1){.149991}}
\multiput(184.959,27.958)(-.046719,.149991){6}{\line(0,1){.149991}}
\multiput(184.398,29.758)(-.046719,.149991){6}{\line(0,1){.149991}}
\multiput(183.837,31.558)(-.046719,.149991){6}{\line(0,1){.149991}}
\multiput(183.277,33.358)(-.046719,.149991){6}{\line(0,1){.149991}}
\multiput(182.716,35.158)(-.046719,.149991){6}{\line(0,1){.149991}}
\multiput(182.156,36.958)(-.046719,.149991){6}{\line(0,1){.149991}}
\multiput(143.474,60.282)(.0577081,.0456584){13}{\line(1,0){.0577081}}
\multiput(144.975,61.469)(.0577081,.0456584){13}{\line(1,0){.0577081}}
\multiput(146.475,62.656)(.0577081,.0456584){13}{\line(1,0){.0577081}}
\multiput(147.975,63.843)(.0577081,.0456584){13}{\line(1,0){.0577081}}
\multiput(149.476,65.03)(.0577081,.0456584){13}{\line(1,0){.0577081}}
\multiput(150.976,66.217)(.0577081,.0456584){13}{\line(1,0){.0577081}}
\multiput(152.477,67.404)(.0577081,.0456584){13}{\line(1,0){.0577081}}
\multiput(153.977,68.591)(.0577081,.0456584){13}{\line(1,0){.0577081}}
\multiput(155.478,69.778)(.0577081,.0456584){13}{\line(1,0){.0577081}}
\multiput(156.978,70.966)(.0577081,.0456584){13}{\line(1,0){.0577081}}
\multiput(158.478,72.153)(.0577081,.0456584){13}{\line(1,0){.0577081}}
\multiput(159.979,73.34)(.0577081,.0456584){13}{\line(1,0){.0577081}}
\multiput(161.479,74.527)(.0577081,.0456584){13}{\line(1,0){.0577081}}
\multiput(162.98,75.714)(.0577081,.0456584){13}{\line(1,0){.0577081}}
\multiput(164.48,76.901)(.0577081,.0456584){13}{\line(1,0){.0577081}}
\multiput(165.98,78.088)(.0577081,.0456584){13}{\line(1,0){.0577081}}
\multiput(167.481,79.275)(.0577081,.0456584){13}{\line(1,0){.0577081}}
\multiput(143.474,59.16)(.23265,-.03924){4}{\line(1,0){.23265}}
\multiput(145.335,58.846)(.23265,-.03924){4}{\line(1,0){.23265}}
\multiput(147.197,58.532)(.23265,-.03924){4}{\line(1,0){.23265}}
\multiput(149.058,58.218)(.23265,-.03924){4}{\line(1,0){.23265}}
\multiput(150.919,57.904)(.23265,-.03924){4}{\line(1,0){.23265}}
\multiput(152.78,57.59)(.23265,-.03924){4}{\line(1,0){.23265}}
\multiput(154.641,57.276)(.23265,-.03924){4}{\line(1,0){.23265}}
\multiput(156.503,56.962)(.23265,-.03924){4}{\line(1,0){.23265}}
\multiput(158.364,56.648)(.23265,-.03924){4}{\line(1,0){.23265}}
\multiput(160.225,56.334)(.23265,-.03924){4}{\line(1,0){.23265}}
\multiput(162.086,56.02)(.23265,-.03924){4}{\line(1,0){.23265}}
\multiput(163.947,55.706)(.23265,-.03924){4}{\line(1,0){.23265}}
\multiput(165.809,55.392)(.23265,-.03924){4}{\line(1,0){.23265}}
\multiput(166.739,55.236)(.0483869,.0467202){14}{\line(1,0){.0483869}}
\multiput(168.094,56.544)(.0483869,.0467202){14}{\line(1,0){.0483869}}
\multiput(169.449,57.852)(.0483869,.0467202){14}{\line(1,0){.0483869}}
\multiput(170.804,59.16)(.0483869,.0467202){14}{\line(1,0){.0483869}}
\multiput(172.159,60.468)(.0483869,.0467202){14}{\line(1,0){.0483869}}
\multiput(173.513,61.776)(.0483869,.0467202){14}{\line(1,0){.0483869}}
\multiput(174.868,63.085)(-.04215,.128556){7}{\line(0,1){.128556}}
\multiput(174.278,64.884)(-.04215,.128556){7}{\line(0,1){.128556}}
\multiput(173.688,66.684)(-.04215,.128556){7}{\line(0,1){.128556}}
\multiput(173.098,68.484)(-.04215,.128556){7}{\line(0,1){.128556}}
\multiput(172.508,70.284)(-.04215,.128556){7}{\line(0,1){.128556}}
\multiput(171.918,72.083)(-.04215,.128556){7}{\line(0,1){.128556}}
\multiput(171.328,73.883)(-.04215,.128556){7}{\line(0,1){.128556}}
\multiput(170.737,75.683)(-.04215,.128556){7}{\line(0,1){.128556}}
\multiput(170.147,77.483)(-.04215,.128556){7}{\line(0,1){.128556}}
\multiput(169.557,79.283)(-.04215,.128556){7}{\line(0,1){.128556}}
\put(117.227,115.321){\makebox(0,0)[cc]{$_{1}$}}
\put(94.522,102.707){\makebox(0,0)[cc]{$_{1|2}$}}
\put(117.227,103.548){\makebox(0,0)[cc]{$_{2|1}$}}
\put(105.734,84.367){\makebox(0,0)[cc]{$_{3|12}$}}
\put(88.735,64.586){\makebox(0,0)[cc]{$_{1|23}$}}
\put(115.235,65.147){\makebox(0,0)[cc]{$_{2|13}$}}
\put(88.735,75.367){\makebox(0,0)[cc]{$_{13|2}$}}
\put(106.856,55.867){\makebox(0,0)[cc]{$_{12|3}$}}
\put(115.235,75.798){\makebox(0,0)[cc]{$_{23|1}$}}
\put(158.485,26.541){\makebox(0,0)[cc]{$_{\Upsilon(a)}$}}
\put(153.105,73.367){\makebox(0,0)[cc]{$_{\Upsilon(a)}$}}
\put(152.535,101.367){\makebox(0,0)[cc]{$_{\Upsilon(a)}$}}
\put(104.735,69.867){\makebox(0,0)[cc]{$_{a}$}}
\put(106.295,101.306){\makebox(0,0)[cc]{$_{a}$}}
\put(106.295,99.063){\circle*{1}}
\put(153.946,98.503){\circle*{1}}
\put(90.038,56.177){\makebox(0,0)[cc]{$_{1|2|3}$}}
\put(122.833,83.086){\makebox(0,0)[cc]{$_{3|2|1}$}}
\put(74.06,8.246){\makebox(0,0)[cc]{$_{1|2|3|4}$}}
\put(117.507,41.564){\makebox(0,0)[cc]{$_{4|3|2|1}$}}
\put(94.522,80.563){\circle*{2}}
\put(119.469,80.563){\circle*{2}}
\put(94.242,59.821){\circle*{2}}
\put(119.469,59.821){\circle*{2}}
\put(94.242,69.632){\circle*{2}}
\put(119.749,70.192){\circle*{2}}
\put(143.575,60.101){\circle*{2}}
\put(169.082,80.003){\circle*{2}}
\put(117.227,38.238){\circle*{2}}
\put(78.265,11.89){\circle*{2}}
\put(143.295,11.329){\circle*{2}}
\put(181.696,38.238){\circle*{2}}
\put(143.855,115.321){\makebox(0,0)[cc]{$_{0}$}}
\put(95.644,98.783){\circle*{2}}
\put(116.946,98.503){\circle*{2}}
\put(144.135,98.222){\circle*{2}}
\put(164.597,98.503){\circle*{2}}
\put(116.946,118.964){\circle*{2}}
\put(144.416,118.964){\circle*{2}}
\put(156.749,70.753){\circle*{1.121}}
\put(162.635,25.064){\circle*{1.121}}
\multiput(145.436,118.584)(.22497,.04426){4}{\line(1,0){.22497}}
\multiput(147.236,118.938)(.22497,.04426){4}{\line(1,0){.22497}}
\multiput(149.036,119.292)(.22497,.04426){4}{\line(1,0){.22497}}
\multiput(150.836,119.646)(.22497,.04426){4}{\line(1,0){.22497}}
\multiput(152.635,120)(.22497,.04426){4}{\line(1,0){.22497}}
\multiput(154.435,120.354)(.22497,.04426){4}{\line(1,0){.22497}}
\multiput(156.235,120.708)(.22497,.04426){4}{\line(1,0){.22497}}
\multiput(158.035,121.062)(.22497,.04426){4}{\line(1,0){.22497}}
\multiput(159.835,121.416)(.22497,.04426){4}{\line(1,0){.22497}}
\multiput(161.634,121.77)(.22497,.04426){4}{\line(1,0){.22497}}
\multiput(162.534,121.948)(.04485,-.07288){10}{\line(0,-1){.07288}}
\multiput(163.431,120.49)(.04485,-.07288){10}{\line(0,-1){.07288}}
\multiput(164.328,119.032)(.04485,-.07288){10}{\line(0,-1){.07288}}
\multiput(165.225,117.575)(.04485,-.07288){10}{\line(0,-1){.07288}}
\multiput(166.122,116.117)(.04485,-.07288){10}{\line(0,-1){.07288}}
\multiput(167.019,114.66)(-.23025,-.04505){4}{\line(-1,0){.23025}}
\multiput(165.177,114.299)(-.23025,-.04505){4}{\line(-1,0){.23025}}
\multiput(163.335,113.939)(-.23025,-.04505){4}{\line(-1,0){.23025}}
\multiput(161.493,113.578)(-.23025,-.04505){4}{\line(-1,0){.23025}}
\multiput(159.651,113.218)(-.23025,-.04505){4}{\line(-1,0){.23025}}
\multiput(157.809,112.857)(-.23025,-.04505){4}{\line(-1,0){.23025}}
\multiput(155.967,112.497)(-.23025,-.04505){4}{\line(-1,0){.23025}}
\multiput(154.125,112.137)(-.065826,.04672){11}{\line(-1,0){.065826}}
\multiput(152.677,113.164)(-.065826,.04672){11}{\line(-1,0){.065826}}
\multiput(151.229,114.192)(-.065826,.04672){11}{\line(-1,0){.065826}}
\multiput(149.781,115.22)(-.065826,.04672){11}{\line(-1,0){.065826}}
\multiput(148.333,116.248)(-.065826,.04672){11}{\line(-1,0){.065826}}
\multiput(146.884,117.276)(-.065826,.04672){11}{\line(-1,0){.065826}}
\put(94.522,30.95){\line(1,0){10.091}}
\put(111.06,30.95){\line(1,0){6.447}}
\multiput(94.301,62.933)(.114583,.047619){8}{\line(1,0){.114583}}
\multiput(96.134,63.695)(.114583,.047619){8}{\line(1,0){.114583}}
\multiput(97.968,64.457)(.114583,.047619){8}{\line(1,0){.114583}}
\multiput(99.801,65.219)(.114583,.047619){8}{\line(1,0){.114583}}
\multiput(101.634,65.981)(.114583,.047619){8}{\line(1,0){.114583}}
\multiput(103.468,66.743)(.114583,.047619){8}{\line(1,0){.114583}}
\multiput(105.301,67.505)(.114583,.047619){8}{\line(1,0){.114583}}
\multiput(107.134,68.267)(.114583,.047619){8}{\line(1,0){.114583}}
\multiput(108.968,69.028)(.114583,.047619){8}{\line(1,0){.114583}}
\multiput(110.801,69.79)(.114583,.047619){8}{\line(1,0){.114583}}
\multiput(112.634,70.552)(.114583,.047619){8}{\line(1,0){.114583}}
\multiput(114.468,71.314)(.114583,.047619){8}{\line(1,0){.114583}}
\multiput(116.301,72.076)(.114583,.047619){8}{\line(1,0){.114583}}
\multiput(118.134,72.838)(.114583,.047619){8}{\line(1,0){.114583}}

\put(107.467,68.281){\circle*{1.009}}
\put(78.54,23.041){\circle*{1.064}}
\put(88.462,27.918){\circle*{1.009}}
\put(88.126,17.322){\circle*{1.009}}
\put(117.221,25.731){\circle*{.752}}
\put(117.389,25.731){\circle*{1.009}}
\put(117.221,30.777){\circle*{1.009}}
\put(78.54,16.818){\circle*{1.009}}
\put(88.294,21.527){\circle*{1.009}}
\put(102.758,16.65){\circle*{.673}}
\put(88.126,34.645){\circle*{.673}}
\put(94.349,31.113){\circle*{.673}}
\put(107.971,32.29){\line(0,-1){17.322}}
\multiput(102.926,23.545)(.08709284,.04805122){56}{\line(1,0){.08709284}}
\multiput(117.221,25.563)(-.084089642,-.048051224){112}{\line(-1,0){.084089642}}
\multiput(117.221,20.854)(-1.84997,-.04204){4}{\line(-1,0){1.84997}}
\multiput(103.598,20.518)(.92499,-.04204){4}{\line(1,0){.92499}}
\put(106.121,30.777){\line(1,0){1.177}}
\multiput(108.476,30.777)(.67272,.04204){4}{\line(1,0){.67272}}
\multiput(82.307,19.408)(.129428,.044176){7}{\line(1,0){.129428}}
\multiput(84.119,20.027)(.129428,.044176){7}{\line(1,0){.129428}}
\multiput(85.931,20.645)(.129428,.044176){7}{\line(1,0){.129428}}
\multiput(87.743,21.264)(.129428,.044176){7}{\line(1,0){.129428}}
\multiput(89.555,21.882)(.129428,.044176){7}{\line(1,0){.129428}}
\multiput(91.367,22.501)(.129428,.044176){7}{\line(1,0){.129428}}
\multiput(93.179,23.119)(.129428,.044176){7}{\line(1,0){.129428}}
\multiput(94.991,23.738)(.129428,.044176){7}{\line(1,0){.129428}}
\multiput(96.803,24.356)(.129428,.044176){7}{\line(1,0){.129428}}
\multiput(98.615,24.975)(.129428,.044176){7}{\line(1,0){.129428}}
\multiput(100.427,25.593)(.129428,.044176){7}{\line(1,0){.129428}}
\multiput(102.239,26.212)(.129428,.044176){7}{\line(1,0){.129428}}
\multiput(104.051,26.83)(.129428,.044176){7}{\line(1,0){.129428}}
\multiput(105.863,27.448)(.129428,.044176){7}{\line(1,0){.129428}}
\multiput(107.675,28.067)(.129428,.044176){7}{\line(1,0){.129428}}
\multiput(109.487,28.685)(.129428,.044176){7}{\line(1,0){.129428}}
\put(97.544,24.554){\circle*{1.009}}
\put(97.544,26.909){\makebox(0,0)[cc]{$_{a}$}}

\put(83.401,22.034){\makebox(0,0)[cc]{$_{_{1|234}}$}}
\put(112.8,26.5){\makebox(0,0)[cc]{$_{_{234|1}}$}}
\multiput(88.294,21.695)(.08408964,.04805122){70}{\line(1,0){.08408964}}
\put(94.349,24.722){\circle*{.673}}
\put(107.803,20.182){\circle*{1.009}}
\put(107.971,14.968){\circle*{.673}}
\put(107.803,25.395){\circle*{1.009}}
\put(107.971,32.29){\circle*{1.009}}
\put(102.926,23.377){\circle*{.673}}
\end{picture}

\vspace{0.1in}

\begin{center}

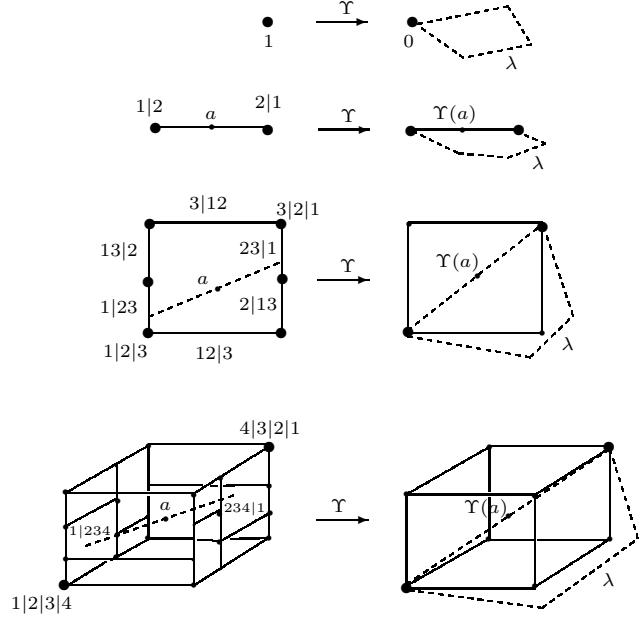
Figure 7. The modelling map $\Upsilon$ for $n_0=0,1,2,3.$
\end{center}

\vspace{0.2in}

\subsection{
The chain complexes of  $ \widehat{\mathbf{\Omega}}X$   and $ \widehat{\mathbf{\Lambda}}X$ }

The chain complex $(C_\ast(\widehat{\mathbf{\Omega}}X),d)$
of  $\widehat{\mathbf{\Omega}} X$ is
\[C_\ast(\widehat{\mathbf{\Omega}}
X)=C'_\ast(\widehat{\mathbf{\Omega}} X)/C'_\ast(D(1)),\]
where $ C'_\ast(\widehat{\mathbf{\Omega}} X)$ is
the free $\Bbbk$-module  generated by
the set
$\widehat{\mathbf{\Omega}} X$
and
$D(1)\subset \widehat{\mathbf{\Omega}}X$ denotes the set  of degeneracies arising
from the unit $1\in \widehat{\mathbf{\Omega}} X;$ the differential $d$
for a generator $\bar \sigma  \in  \widehat{\mathbf{\Omega}}_{m-1} X,\, m>1, $  is defined by
\[
  d(\bar \sigma)=
\sum_{A|B\in \overline{P}(m) }(-1)^{\# A} sgn(A,B)\,d_{A|B}(\bar \sigma ),
\]
and extended as  a derivation.
Analogously,
the chain complex $(C_\ast(\widehat{\mathbf{\Lambda}}X),d)$
of  $\widehat{\mathbf{\Lambda}} X$ is
\[C_\ast(\widehat{\mathbf{\Lambda}}
X)=C'_\ast(\widehat{\mathbf{\Lambda}} X)/C'_\ast(D(1)),\ \  D(1)\subset \widehat{\mathbf{\Omega}}X   \subset \widehat{\mathbf{\Lambda}} X, \]
while  the differential $d=\{\partial_n\}_{n\geq 1}$ with
$\partial_n:   C_n(\widehat{\mathbf{\Lambda}}X)  \rightarrow
C_{n-1}(\widehat{\mathbf{\Lambda}}X)    $
is given
by
 \begin{equation}\label{sign}
   \partial_n=\bigoplus_{\substack{{n=n_0+r}\\{n_0,r\geq 0\,;\,k\geq 1}}}\,
 \partial_{n_0,r,k}
   \end{equation}
   in which $\partial_{n_0,r,k}$ acts on $   C_\ast(\widehat{\mathbf{\Lambda}}_{n_0,r,k}X) $
and is defined by
\begin{multline}\label{Pset}
  \partial_{n_0,r,k}=
\sum_{A|B\in \overline{P}'(n_0) }(-1)^{\# A}
sgn(A,B)\,d_{A|B}+\\
\hspace{0.5in} \sum_{C|D\in \overline{P}''(n_0)}
(-1)^{
(\#C+1)( \# D+r)} sgn(C,D)\, d^{op}_{C|D}+
\\
  \sum_{\substack{A_i|B_i\in P(n_i)\\ 1\leq i\leq k}}
  (-1)^{\# A_i+n_0+r_{i-1}} sgn(A_i,B_i)\,d_{A_i|B_i}.
\end{multline}
Thus, the three kinds of summand  in above formula represents $d$  as the sum \[d=d_1+d_2+d_3.\]
Furthermore,   (\ref{DeltaP}) induces the coproduct
 \begin{equation}\label{Lambdadiag}
 \Delta_{\mathbf{\Lambda}}:
 C^{\diamond}_\ast(\widehat{\mathbf{\Lambda}}X)
 \rightarrow
 C^{\diamond}_\ast(\widehat{\mathbf{\Lambda}}X)\otimes
 C^{\diamond}_\ast(\widehat{\mathbf{\Lambda}}X)
  \end{equation}
  making  $C^{\diamond}_\ast(\widehat{\mathbf{\Lambda}}X)$ as a dg
  (non-coassociative) coalgebra.

\begin{remark}\label{permufaces}
1.  Note that both $\widehat{\mathbf{\Omega}}X$  and
    $\widehat{\mathbf{\Lambda}}X$ are permutahedral sets, while the sign of the second summand in (\ref{Pset}) is not the standard permutahedral sign; instead it agrees with the one of differential of the coHochschild complex   of $C_\ast(X)$ (cf. Theorem \ref{hat-coHoch}  below);

2. The relations among permutahedral set  face operators obtained via morphisms of (closed) cubical
necklical sets rely on
the coassociativity of the cubical diagonal, and the exposition is
more transparent rather than  the one  in \cite{KStwisted}.
\end{remark}

\subsection{The $\Cap$ -- product  on $C^\diamond(\widehat{\mathbf{\Lambda}}X) $} \label{circledcirc}

Given a cube   $u\in X_{m},$ a  \emph{cubical edge-path}  $\lambda_u$
from $\min u $ to $\max u$
 is defined  for $m=1$ as  $\lambda_u=u$  and for $m>1$
as
the composition $\lambda_u=u_1\ast\cdots \ast u_m$ of edges
where
  $u_1=d^0_2\circ \cdots \circ d^0_m(u),$
    $u_i= d^1_{A_i}d^0_{B_i}(u)$
 for $(A_i,B_i)=(\{1,...,i-1\},\{i+1,...,m\})$  with $1<i<m,$
    and    $u_m=d^1_1\circ \cdots  \circ  d^1_{m-1}(u).$
Given two vertices $a,b\in X_0,$ fix an
 edge-path $\lambda_{a,b}$   from  $a$ to $b$ as a composition of
 cubical edge-paths $\lambda_u.$

For two elements  $a\in\widehat{\mathbf{\Omega}}(X;x_1,x_2)$
and $b\in\widehat{\mathbf{\Omega}}(X; y_1,y_2)$
 define the product $a b\in \widehat{\mathbf{\Omega}}(X;x_1,y_2)$ as   the concatenation
 $a \, \bar \lambda_{x_2,y_1}\,b,$ 
 but  usually we omit the edge-path
$ \lambda_{x_2,y_1}.$  An element of $\widehat{\mathbf{\Lambda}}X $
is usually denoted by $u]a$  for $u\in X$  and $a\in \widehat{\mathbf{\Omega}}X, $ while denote $u]:= u]1$
for the unit $a=1\in \widehat{\mathbf{\Omega}}X.$  
Define the $\Cap$ -- product
\[\Cap :
C^{\diamond}_{p,s}(\widehat{\mathbf{\Lambda}}X)\otimes
C^{\diamond}_{q,t}(\widehat{\mathbf{\Lambda}}X) \rightarrow
C^{\diamond}_{p+q-n, s+t}(\widehat{\mathbf{\Lambda}}X)\] for
elementary chain pair
 $       \alpha\otimes \beta       \in C^{\diamond}(\widehat{\mathbf{\Lambda}}X)\otimes
C^{\diamond}(\widehat{\mathbf{\Lambda}}X)$
   by
   \begin{equation}\label{capformula}
 \alpha\,\Cap  \beta=\left\{
  \begin{array}{llll}
  (-1)^{|a||v|}\, u\sqcap \, v\,]\,ab,          &   \alpha= u\,]\,a,  \,\, \beta=e_w\,]\,b  ,
  \vspace{1mm}\\
  -d_2(\alpha\Cap e_w\,]\,b\,)+  \alpha\Cap d_2(e_w\,]\,b\,), &
   \beta=d_{ \underline {q}\smallsetminus \underline{1}\,\mid\, \underline 1} (e_w\,]\,b\,),\, q=|e_w|.
     \end{array}
   \right.
 \end{equation}

\begin{proposition}\label{algebra2}
 The  product \[\Cap :
C^{\diamond}_{p,s}(\widehat{\mathbf{\Lambda}}X)\otimes
C^{\diamond}_{q,t}(\widehat{\mathbf{\Lambda}}X) \rightarrow
C^{\diamond}_{p+q-n, s+t}(\widehat{\mathbf{\Lambda}}X)\] for
 $       \alpha\otimes \beta       \in C^{\diamond}_{p,s}(\widehat{\mathbf{\Lambda}}X)\otimes
C^{\diamond}_{q,t}(\widehat{\mathbf{\Lambda}}X)$
satisfies the equality
\begin{equation}\label{leibnitz}
 d( \alpha\Cap \beta )= (-1)^{n+q}\,  d\alpha\Cap \beta +    \alpha\Cap d\beta.
\end{equation}
  \end{proposition}

\begin{proof}
The proof is similar to that of Proposition \ref{algebra}. Indeed,
Consider $d_\epsilon ( \alpha\Cap \beta )$  for   $ \epsilon=1,2,3.$

(i) Let $(\alpha,\beta)=(u]a,e_w]b).$

(i1) $\epsilon=1.$ For a component $d_{A|B}(u\sqcap  e_w]ab)$ of $d_1( u]a\Cap e_w]b )$ we have
\begin{multline*}
d_{A|B}(u\sqcap  e_w]ab) =u_{B}]\bar u_{A} ab\ \   \text{with}\\
u=w\times (u\sqcap  e_w)=w\times u_B\times u_A= u_D\times u_A\ \ \text{and}\ \   u_D\sqcap e_w= u_{B}.
\end{multline*}
 Consequently,
\begin{multline*}
  d_{A|B}(u\sqcap e_w]\,ab) =u_{B}]\bar u_{A} ab =u_D\sqcap e_w]\bar u_A\,ab=\\
  u_D]\bar u_A \,a\,\Cap e_w]b =d_{A|D}(u]a)\Cap e_w]b,
 \end{multline*}
 and then
 \[  d_1(\alpha\Cap \beta)=(-1)^{n+q} d_1(\alpha)\Cap\beta.\]
 In other words, the above equality follows from (\ref{left}).
The definition of the $\Cap$ -- product implies that $d^{op}_{C|D}(u]a)\Cap' e_w]b=0$ for all $C|D,$ hence,
$
d_2(\alpha) \Cap e_w]v =0,
$
 and, consequently,
\[  d_1(\alpha\Cap \beta)=(-1)^{n+q}(d_1+d_2)(\alpha)\Cap\beta.\]

 (i2)  $\epsilon=2.$ We have
  \[d_2(\alpha\Cap\beta)= \alpha\Cap (d_1+ d_2)(\beta),\]
 because of  the second item of the definition of the $\Cap$ -- product
 and
 the equality
   $\alpha\Cap' d_{A|B}(e_w]\,b\,) = 0$ for any component  $ d_{A|B}(e_w]\,b\,)$  of $d_1(e_w]\,b\,) $  unless
$A|B=  \underline {q}\smallsetminus \underline{1}\,|\,\underline{1}. $
Furthermore,  the definition of the $\Cap$ -- product implies that $d^{op}_{C|D}(u]a)\Cap' e_w]b=0$ for all $C|D,$ hence,
$
d_2(\alpha) \Cap e_w]v =0,
$
 and then
\[  (d_1+d_2)(\alpha\Cap \beta)=(-1)^{n+q}(d_1+d_2)(\alpha)\Cap\beta+\alpha \Cap (d_1+d_2)(\beta).\]

(ii) Let
$(\alpha,\beta)=(u]a\,,  d_{ \underline {q}\smallsetminus \underline{1}\,\mid \, \underline{1}}(e_w]\,b));$ in particular, $|\beta|=q-1.$

(ii1).  $\epsilon=1.$ Then
\begin{multline*}
d_1(\alpha\Cap  \beta)= d_1(-d_2(\alpha\Cap e_w]b )+ \alpha\Cap d_2(e_w]b))=
d_2d_1(\alpha\Cap e_w]b )+ d_1(\alpha\Cap d_2(e_w]b))=
\\
(-1)^{n+q}\,d_2(d_1(\alpha)\Cap e_w]b) -(-1)^{n+q}d_1(\alpha)\Cap d_2(e_w]b)=\\
(-1)^{n+q}\,d_1(\alpha)\Cap d_2(e_w]b)- (-1)^{n+q}\,d_1(\alpha)\Cap \beta
 -
 (-1)^{n+q}\,d_1(\alpha)\Cap d_2(e_w]b)=\\-(-1)^{n+q}\, d_1(\alpha)\Cap \beta.
\end{multline*}

 (ii2) $\epsilon=2.$ Then
\begin{multline*}
d_2(\alpha\Cap \beta)=d_2(-d_2(\alpha\Cap e_w]b )+ \alpha\Cap d_2(e_w]b))=d_2(\alpha\Cap d_2(e_w]b))=
\\
\alpha\Cap(d_1+d_2)(d_2(e_w]b))=\alpha\Cap d_1d_2(e_w]b)=-
\alpha\Cap d_2d_1(e_w]b)=-
\alpha\Cap d_2(\beta).
\end{multline*}

Since $     d_2(\alpha)\Cap \beta=0 $  (because $w\nsubseteq d^1_{A}(u),$ neither $A$) and   $ \alpha\Cap d_1(\beta)=0=
 - u\Cap d_1d_1(e_w]b),    $ obtain
\[ (d_1 +d_2)(\alpha \Cap \beta) = (-1)^{n+q+1} (d_1+d_2)(\alpha)\Cap \beta +\alpha\Cap (d_1+d_2)(\beta).  \]

Finally, the verification of $d_3$ is a $\Cap$ -- derivation   is obvious.
\end{proof}
Note   that by the cellular map (\ref{varsigma}) we have $\varsigma_{q}\left( d_{ \underline {q}\smallsetminus \underline{1}\,\mid \, \underline 1}(v]\,\,)\right)=
d^0_1(v)$ for a cube $v\in K^{\Box}_q.$

\begin{proposition}\label{Capcommut}
    The  product
   $\Cap : H_p(\widehat{\mathbf{\Lambda}}X)\otimes
   H_{q}(\widehat{\mathbf{\Lambda}}X) \rightarrow
   H_{p+q-n}(\widehat{\mathbf{\Lambda}}X)$
    is commutative and associative.
  \end{proposition}
 \begin{proof}  Let $(X^{op},\widetilde{d}^\epsilon_i) $ be the cubical set as
in the proof of Proposition  \ref{commut}.
Denote $\iota :  \widehat{\mathbf \Lambda}X\rightarrow
\widehat{\mathbf \Lambda}X^{op},\,\,$
 $ u\,]\,\bar a_1\cdots \bar a_k\rightarrow
u^{op}\,]\,\overline {a_k^{op}}\cdots \overline {a_1^{op}}, $ and
define \[\Cap^{op}: C_\ast(\widehat{\mathbf \Lambda}X^{op})\otimes C_\ast(\widehat{\mathbf \Lambda}X^{op})\rightarrow
C_\ast(\widehat{\mathbf \Lambda}X^{op})  \]
 by $u]a \Cap^{op} v]b= u\sqcap ^{op} v\,]\,ab.$ Then
 for $u]a \otimes v]b \in C_p(\widehat{\mathbf \Lambda}X) \otimes C_q(\widehat{\mathbf \Lambda}X)$
 \[  \iota(u]a \Cap v]b)=(-1)^{pq}  \iota(v]b)\Cap^{op} \iota(u]a).  \]
Since $\iota$  induces an isomorphism in homology, the commutativity follows. To check the associativity is straightforward.
 \end{proof}

\subsection{The twisted tensor product  $C^{\diamond}_\ast(\widehat{\mathbf{\Lambda}} X)\otimes _\tau
 C^{\diamond}_\ast(\widehat{\mathbf{\Omega}} X) $ }

 Consider the tensor product of chain complexes $C^{\diamond}_\ast(\widehat{\mathbf{\Lambda}} X)\otimes
 C^{\diamond}_\ast(\widehat{\mathbf{\Omega}} X) .$
 Using Sweedler's notations $\Delta(u)=u_{(1)}\otimes u_{(2)}$ and $\Delta(u)=u'\otimes u''$ as well
 define a map $\Theta=\Theta_1+\Theta_2,$
 \[ \Theta_i:  C^{\diamond}_\ast(\widehat{\mathbf{\Lambda}} X) \otimes
 C^{\diamond}_\ast(\widehat{\mathbf{\Omega}} X)  \rightarrow
 C^{\diamond}_\ast(\widehat{\mathbf{\Lambda}} X) \otimes
 C^{\diamond}_\ast(\widehat{\mathbf{\Omega}} X),\ \ i=1,2   \]
 for $  u]a\otimes b \in   C^{\diamond}_\ast(\widehat{\mathbf{\Lambda}} X)\otimes
 C^{\diamond}_\ast(\widehat{\mathbf{\Omega}} X) $  by
 \[\Theta_1(u]a\otimes b)= u_{(1)}\,] (\bar u_{(2)})^{'}\!a \otimes  (\bar u_{(2)})^{''} b\ \
  \text{and}\ \   \Theta_2(u]a\otimes b)=  u_{(2)}\,]\, a (\bar u_{(1)})^{'}\!\otimes\,  b \,(\bar u_{(1)})^{''}\!\!.  \]
   Then the tensor product of modules  $ C^{\diamond}_\ast(\widehat{\mathbf{\Lambda}} X)\otimes
 C^{\diamond}_\ast(\widehat{\mathbf{\Omega}} X)$   with the differential
  $d_\Theta:= d_{\mathbf{\Lambda}} \otimes 1 +1\otimes d_{\mathbf{\Omega}}+\Theta   $
  is denoted by
  $C^{\diamond}_\ast(\widehat{\mathbf{\Lambda}} X)\otimes _\tau
 C^{\diamond}_\ast(\widehat{\mathbf{\Omega}} X). $ The equality  $d^2_\Theta=0$
 uses the fact that the coproduct $\Delta_{\mathbf{\Omega}}:
 C^{\diamond}_\ast(\widehat{\mathbf{\Omega}}X)
 \rightarrow 
 C^{\diamond}_\ast(\widehat{\mathbf{\Omega}}X)
 \otimes
 C^{\diamond}_\ast(\widehat{\mathbf{\Omega}} X)
 $ is chain; in particular, $\Theta\circ \Theta=0,$ and $d_i\otimes 1$ is a summand component of $\Theta_i$
 for $i=1,2.$

 Define the chain maps
 \[\nu_r : C^{\diamond}_\ast(\widehat{\mathbf{\Lambda}} X)   \rightarrow   C^{\diamond}_\ast(\widehat{\mathbf{\Lambda}} X)\otimes _\tau
 C^{\diamond}_\ast(\widehat{\mathbf{\Omega}} X)\ \ \text{and}\ \
 \nu_l : C^{\diamond}_\ast(\widehat{\mathbf{\Lambda}} X)   \rightarrow   C^{\diamond}_\ast(\widehat{\mathbf{\Omega}} X)\otimes _\tau
 C^{\diamond}_\ast(\widehat{\mathbf{\Lambda}} X)    \]
and
\[ \mu_r: C^{\diamond}_\ast(\widehat{\mathbf{\Lambda}} X)\otimes _\tau
 C^{\diamond}_\ast(\widehat{\mathbf{\Omega}} X)\rightarrow   C^{\diamond}_\ast(\widehat{\mathbf{\Lambda}} X)\ \  \text{and}\ \
  \mu_l: C^{\diamond}_\ast(\widehat{\mathbf{\Omega}} X)\otimes _\tau
 C^{\diamond}_\ast(\widehat{\mathbf{\Lambda}} X)\rightarrow   C^{\diamond}_\ast(\widehat{\mathbf{\Lambda}} X)  \]
 as follows. For  $u]a\in  C^{\diamond}_\ast(\widehat{\mathbf{\Lambda}} X)  ,$ let
 \begin{equation}\label{nu}
 \nu_r( u]a)= u] a'\otimes a''   \ \     \text{and} \ \
 \nu_l( u]a)= a'\otimes u] a''  ,
 \end{equation}
 and for
  $  \alpha  \in   C^{\diamond}_\ast(\widehat{\mathbf{\Lambda}} X)\otimes _\tau
 C^{\diamond}_\ast(\widehat{\mathbf{\Omega}} X)$  and
 $  \beta \in   C^{\diamond}_\ast(\widehat{\mathbf{\Omega}} X)\otimes _\tau
 C^{\diamond}_\ast(\widehat{\mathbf{\Lambda}} X),$ let
 \begin{equation}\label{mur}
    \mu_r( \alpha  )= \left\{ \!\! \begin{array}{llll}
                                u]ab , & \alpha= u]a\otimes b \notin\operatorname{Im}\Theta ,\vspace{1mm} \\
                             u_{(1)}\,] \bar u_{(2)}a b
                                    +  u_{(2)}\,]\, a  b\,\bar
                                    u_{(1)} , &  \alpha=
                                 u_{(1)}\,] (\bar u_{(2)})'\!
                                 a \otimes  (\bar u_{(2)})'' b\,
    + \vspace{1mm} \\
    & \hspace{0.27in} u_{(2)}\,]\, a (\bar u_{(1)})'\!\otimes\,  b \,(\bar u_{(1)})''\in \operatorname{Im}\Theta

                               \end{array} \right.               \end{equation}
 and
 \begin{equation}\label{mul}
   \mu_l( \beta  )= \left\{\!\! \begin{array}{llll}
                                u]ab , & \beta=a\otimes  u]b \notin\operatorname{Im}\Theta ,\vspace{1mm} \\
                            u_{(1)}\,] \bar u_{(2)}  ab
                                    +   u_{(2)}\,]\, ab\,\bar u_{(1)}, &  \beta=
                            (\bar u_{(2)})' a\otimes u_{(1)}\,] (\bar u_{(2)})'' \,b\,

    + \vspace{1mm} \\
    & \hspace{0.27in}   a \,(\bar u_{(1)})'\otimes u_{(2)}\,]\, b\,(\bar u_{(1)})''\in \operatorname{Im}\Theta .

                               \end{array} \right.
    \end{equation}

Also there are "an extended switch chain maps"
\[ \mathcal {T}_r:   C^{\diamond}_\ast(\widehat{\mathbf{\Lambda}} X)   \otimes C^{\diamond}_\ast(\widehat{\mathbf{\Lambda}} X)\otimes _\tau
 C^{\diamond}_\ast(\widehat{\mathbf{\Omega}} X)\rightarrow
  C^{\diamond}_\ast(\widehat{\mathbf{\Lambda}} X)   \otimes C^{\diamond}_\ast(\widehat{\mathbf{\Lambda}} X)\otimes _\tau
 C^{\diamond}_\ast(\widehat{\mathbf{\Omega}} X)  \]
 and
 \[ \mathcal {T}_l:
 C^{\diamond}_\ast(\widehat{\mathbf{\Omega}} X) \otimes_\tau C^{\diamond}_\ast(\widehat{\mathbf{\Lambda}} X)   \otimes C^{\diamond}_\ast(\widehat{\mathbf{\Lambda}} X)\rightarrow
 C^{\diamond}_\ast(\widehat{\mathbf{\Omega}} X) \otimes _\tau C^{\diamond}_\ast(\widehat{\mathbf{\Lambda}} X)   \otimes C^{\diamond}_\ast(\widehat{\mathbf{\Lambda}} X) \]
 defined for
    \[\alpha  \in  C^{\diamond}_\ast(\widehat{\mathbf{\Lambda}} X)\otimes  C^{\diamond}_\ast(\widehat{\mathbf{\Lambda}} X)
 \otimes _\tau
 C^{\diamond}_\ast(\widehat{\mathbf{\Omega}} X)\ \  \text{and}\ \
   \beta \in   C^{\diamond}_\ast(\widehat{\mathbf{\Omega}} X)\otimes _\tau
 C^{\diamond}_\ast(\widehat{\mathbf{\Lambda}} X)\otimes C^{\diamond}_\ast(\widehat{\mathbf{\Lambda}} X)\]
  by

 \begin{equation}\label{tr}
   \mathcal{T}_r( \alpha  )= \left\{ \!\! \begin{array}{llll}
                                v]b\otimes u]a\otimes c ,
                                &\ \alpha= u]a\otimes v]b \otimes c \notin\operatorname{Im}(1\otimes \Theta), \vspace{2mm} \\

                             (d_1+d_2)(v]b)    \otimes  u]a\otimes c\,+\vspace{1mm}\\
                              v]b    \otimes  \Theta(u]a\otimes c) ,

                                    &\  \alpha= u]a\otimes v]b \otimes c\in
                                    \operatorname{Im}(1\otimes \Theta)
                                    \end{array} \right.
 \end{equation}
 and
 \begin{equation}\label{tl} \mathcal{T}_l( \beta  )=\left\{ \!\! \begin{array}{llll}
                                a\otimes v]c\otimes u]b ,
                                &\ \ \beta= a\otimes u]b\otimes v] c \notin\operatorname{Im}(\Theta\otimes 1), \vspace{2mm} \\

                              \Theta(a\otimes v]c)\otimes  u]b\,+
                              \vspace{1mm}\\
                              a\otimes v]c\otimes  (d_1+d_2)(u]b),

                                    &\ \  \beta=a\otimes u]b\otimes v]c\in
                                    \operatorname{Im}( \Theta\otimes 1).
                                    \end{array} \right.
 \end{equation}

\subsection{Proof of Theorem \ref{string}}

Denote $f:=  \Delta_{\mathbf\Lambda}\circ \Cap$  and $g:=  (\Cap\otimes \mu_r)\circ (1\otimes \mathcal{T}_r)  \circ (\Delta_{\mathbf\Lambda}\otimes \nu_r),$ so that
\[ f,g :   C^{\diamond}_\ast(\widehat{\mathbf{\Lambda}} X)   \otimes C^{\diamond}_\ast(\widehat{\mathbf{\Lambda}} X)\rightarrow  C^{\diamond}_\ast(\widehat{\mathbf{\Lambda}} X)   \otimes C^{\diamond}_\ast(\widehat{\mathbf{\Lambda}} X).\]
We have for $u]a\otimes v]b\in C^{\diamond}_\ast(\widehat{\mathbf{\Lambda}} X)   \otimes C^{\diamond}_\ast(\widehat{\mathbf{\Lambda}} X)$ that
$f(u]a\otimes v]b)=g(u]a\otimes v]b)$ for $|u|+|v|=n,$ while
for $|u|+|v|>n$ both   $f(u]a\otimes v]b)$   and 
$g(u]a\otimes v]b)$
lie in the same (acyclic) subcomplex
$u]a'b' \otimes u]a''b''\subset    C^{\diamond}_\ast(\widehat{\mathbf{\Lambda}} X)   \otimes C^{\diamond}_\ast(\widehat{\mathbf{\Lambda}} X), $ so that
by the standard acyclic argument we get a chain homotopy
\begin{equation}\label{cbistring1}
  \Delta_{\mathbf\Lambda}\circ \Cap\simeq
       (\Cap\otimes \mu_r)\circ (1\otimes \mathcal{T}_r)  \circ (\Delta_{\mathbf\Lambda}\otimes \nu_r).
  \end{equation}
  Entirely dually, we establish  a chain homotopy
\begin{equation}\label{cbistring2}
  \Delta_{\mathbf\Lambda}\circ \Cap^{op}\simeq
       (\mu_l\otimes\Cap^{op} )\circ (\mathcal{T}_l\otimes 1)  \circ (\nu_l\otimes \Delta_{\mathbf\Lambda}) .
  \end{equation}
Consequently,
(\ref{bistring1}) and (\ref{bistring2})  hold respectively.
\hspace{1.8in} $\Box$

\vspace{0.2in}

\section{Algebraic models for the free loop space and the
hat-coHochschild  construction.}

\subsection{Algebraic preliminaries}

We fix a ground commutative ring $\Bbbk$  with unit $1_{\Bbbk}$. All
modules are assumed to be  over $\Bbbk.$
We recall some algebraic constructions associated to differential
graded (dg) coassociative  coaugmented coalgebras. Recall that a dg
coalgebra $(C,d_C, \Delta)$ is \text{coaugmented} if it is equipped
with a map of dg coalgebras $\epsilon: \Bbbk \to C$. Denote
$\overline{C}= \text{coker} (\epsilon)$.
 Given a coaugmented dg coalgebra $(C, d_C, \Delta, \epsilon)$ which is
 free as a $\Bbbk$-module on each degree, the \textit{cobar
 construction} of $C$ is the differential graded (dg) associative
 algebra $(\Omega C,d_{\Omega C})$ defined as follows.
For any $ c \in \overline{C}$ write $\Delta(c)= \sum  c'
\otimes  c''$ for the induced coproduct on $\overline{C}$. The
underlying algebra of the cobar construction is the tensor algebra
\[\Omega C= Ts^{-1}\overline{C}= \Bbbk \oplus s^{-1}\overline{C} \oplus
\left(s^{-1} \overline{C}\, \right)^{\otimes 2} \oplus \left(s^{-1}
\overline{C}\, \right)^{\otimes 3} \cdots ;
 \] 
  Denoting  $[\bar c_1|...|\bar c_n]:= s^{-1}
 c_1 \otimes ... \otimes s^{-1} c_n \in \Omega C,$ 
  the differential
$d_{\Omega C}$ is defined by extending
\begin{equation}\label{cobarsign}
  d_{\Omega C}([\bar c]) =-\left[\, \overline{d_{C}  (c)} \,\right] +
  \sum (-1)^{|c'|} \left[\, \bar{ c'} \mid \bar {c''} \, \right]
\end{equation}
 as a derivation to all of $\Omega C.$ Thus, the cobar construction defines a functor
 from the category of coaugmented dg coalgebras to the category of
 augmented dg  algebras.

The \textit{coHochschild complex} of $C$ is the dg $\Bbbk$-module
$\Lambda C= (C \otimes \Omega C, d_{\Lambda C})$ with differential
$d_{\Lambda C}= d_C \otimes 1 + 1 \otimes d_{\Omega C} + \theta_1 +
\theta_2$ where

\begin{equation}\label{thetas}
\begin{array}{llll}
\theta _1 (v\otimes [ \bar c_1|\dotsb | \bar c_n]) = \sum \,
(-1)^{|v'|+1}\,v'\otimes
 [\,\bar{ v''}\,|\, \bar c_1|\!\dotsb\! |\bar c_n],
 \newline $\vspace{1mm}$
 \\
\theta _2 (v\otimes [ \bar c_1|\dotsb | \bar c_n]) = \sum\,
 (-1)^{(| v'|+1) (|{v''}|+\epsilon^c_n)}\, v''\otimes [ \bar
 c_1|\!\dotsb \!| \bar c_{n}|\, \bar{v'}\,],
\newline $\vspace{1mm}$\\
\hspace{2.8in}
\epsilon^x_n=|x_1|+\cdots +|x_n|+n.
\end{array}
\end{equation}

\subsection{The hat-coHochschild construction of a cubical chain
complex} Let $(X,x_0)$ be a pointed cubical set, and let
$(C_\ast(X),d_{C},\Delta_{_\Box})$ be  the cubical chain complex of
$X.$
In fact, the definition of the \textit{hat-cobar construction}  of the
cubical chain complex mimics the one of   the simplicial chain complex.
Consider the coaugmented  dg coalgebra $(C_\ast (Z(X)),
d_C,\Delta_{_\Box}, \epsilon)$, where $\epsilon$ is determined by the
choice of fixed point $x_0$.
Obtain a new coaugmented dg coalgebra
\[ A:=(  C_\ast (Z(X))  , d_A=0,\Delta'_{_\Box}, \epsilon)
 \ \ \text{with}\
\Delta'_{_\Box}\  \text{to be}\  \Delta_{_\Box}
 \text{
 without the primitive  terms}.\]

Let $(\Omega A,d_{\Omega A})$ be the cobar construction of $A,  $ and
define the submodule $\Omega'_n A\subset \Omega _nA$ for $n\geq 0$
 to be
 generated by monomials $[\bar a_1|\cdots| \bar a_k]\in \Omega_n
 A,\,k\geq 1,  $   where  $a_i\in Z(X)$ with
 $\min a_1=\max a_k= x_0 $ and
 $\max a_i=\min a_{i+1}$
for all $i.$
Then $\Omega'A$ inherits the structure of a dg algebra. In particular,
$\Omega' A= \Omega A$ when $X_0=\{x_0\}$.
 Define the \emph{hat-cobar construction}  $\widehat{\Omega}C_\ast(X)$
 of the dg coalgebra $C_\ast(X)$ as
\[  \widehat{\Omega}C_\ast(X)=\Omega' A /\sim ,  \]
where $\sim$ is generated by
\[
[\bar a_1|...|\bar a_{i-1}|\bar a_i|\bar a_{i+1}|\bar a_{i+2}|...|\bar
a_k ]\sim [\bar a_1|...|\bar a_{i-1}|\bar a_{i+2}|...|\bar a_k]
\ \  \text{whenever}\ \  a_{i+1}=a_i^{op};
\]
in particular, $[\,\bar a_i|\bar a_{i+1}]\sim 1_{\Bbbk}.$

The \textit{hat-coHochschild complex}  $(\widehat{\Lambda} C_\ast(X),d_{\widehat{\Lambda}C})$   of  $C_\ast(X)$ is defined as
\[\widehat{\Lambda} C_\ast(X)=C_\ast(X)\otimes \widehat{\Omega}
C_\ast(X)\]
 with  differential
        $d_{\widehat{\Lambda}C}= d_C\otimes 1 + 1\otimes d_{
        \widehat{\Omega} C}+ \theta _1+\theta _2,$
 where $\theta_1$ and $\theta_2$ are defined as in (\ref{thetas}).
  The homology of  $\widehat{\Lambda} C_\ast(X) $ is called the
  \textit{hat-coHochschild homology}
 of
$ C_\ast(X)$ and is denoted by $\widehat{HH}_*( C_\ast(X)).$

 We have a straightforward
\begin{theorem}\label{hat-coHoch} For a cubical set $(X,x_0)$
the permutahedral chain complex  $C^\diamond_\ast(\widehat{\mathbf{\Omega}} X)$
coincides with the hat-cobar construction $\widehat{\Omega}C_\ast(X),$
and
the permutahedral chain complex  $C^{\diamond}_\ast(\widehat{\mathbf{\Lambda}} X)$
coincides with the hat-coHochschild complex
$\widehat{\Lambda}C_\ast(X)$
of the cubical chain complex $C_\ast(X).$
\end{theorem}
In particular, the component $d_i$ of the differential $d$ in $C^{\diamond}_\ast(\widehat{\mathbf{\Lambda}} X)$ is identified with the component $\theta_i$ of $d_{\widehat{\Lambda} C}$
in $\widehat{\Lambda}C_\ast(X)$ for $i=1,2.$
Furthermore,
for a $1$-reduced $X$ (e.g., $X=\operatorname{Sing}_I^1(Y,y)$ the
cubical singular set consisting of all singular cubes in a topological space $Y$
which collapse edges to a fixed point $y \in Y$), the hat-cobar
construction $\widehat{\Omega} C_\ast(X)$ coincides with the Adams'
cobar construction  $\Omega C_\ast(X)$ of  the dg coalgebra $C_\ast(X),$
and, consequently, the hat-coHochschild construction $\widetilde{\Lambda}
C_\ast(X)$ coincides with the standard coHochschild construction
$\Lambda C_\ast(X).$
Thus  we obtain
\begin{theorem}\label{coHoch}
 For a 1-reduced cubical set $X$
the permutahedral chain complex  $C^\diamond_\ast({\mathbf{\Omega}} X)$    coincides
with the cobar construction ${\Omega}C_\ast(X),$
and
the permutahedral chain complex  $C^{\diamond}_\ast({\mathbf{\Lambda}} X)$    coincides
with the coHochschild complex
${\Lambda}C_\ast(X)$
of the cubical chain complex $C_\ast(X).$
\end{theorem}

It follows directly from Theorems \ref{freeloopmodel} and
\ref{hat-coHoch}  that for a path connected cubical set $X$ we have an
isomorphism  $\widehat{HH}_*( C_\ast(X)) \cong H_*(\Lambda Y)$ for
$Y=|X|$. Moreover,  from the homotopy invariance of the free loop space
we have the following direct
\begin{corollary}
If $f_\ast:C_\ast(X)\rightarrow C_\ast(X')$ is induced by a weak
equivalence $f: X\rightarrow X',$
then $\widehat{\Lambda}f_\ast:\widehat{\Lambda}C_\ast(X) \rightarrow
\widehat{\Lambda}C_\ast(X') $ is a quasi-isomorphism.
\end{corollary}

\medskip

\section{Loop bialgebra}

\subsection{Hat-cobar construction of a dg coalgebra}
Recall the definition of the hat-cobar construction of a dg coalgebra
from \cite{RS1}.
Let  $(C,d_C,\Delta)$ be a dg coalgebra such that the module of cycles
$Z_1(C)\subset C_1$ is free  with basis $\mathcal{Z}_1.$
Let $G_1$ be the free group generated by $\mathcal{Z}_1,$ and let
$\Bbbk[G_1]$ be the group ring.
Define a graded module $C[1]$ as
  $C[1]_0=C_0$, $C[1]_1=\Bbbk[G_1]$ and $C[1]_i=C_i$ for $i\geq 2.$
  Then $C\subset C[1]$ extends  to the dg coalgebra
 $(C[1], d, \Delta)$ (with $d(C[1]_1)=0$).

Define the hat-cobar construction  $(\widehat{\Omega} C,d_{\widehat\Omega})$ of $C$  as
the standard cobar construction
${\Omega} C[1]$ of $C[1]$ modulo the relations $[\,\bar 1_{G_1}]=
1_{\Bbbk}$ and
\begin{multline*}
[\bar a_1|...|\bar a_{i-1}|\bar a_i|\bar a_{i+1}|\bar a_{i+2}|...|\bar
a_k ]= [\bar a_1|...|\bar a_{i-1}|\overline {a_i a_{i+1}}\,|\bar
a_{i+2}|...|\bar a_k]
\ \  \text{whenever}\\  a_{i},a_{i+1}\in G_1.
\end{multline*}

\begin{remark}
  Regarding simplicial and cubical chain complexes $C_\ast(X)$  as dg
  coalgebras $C,$ their hat-cobar constructions are different unless
  $C_1(X)=0.$
\end{remark}

The hat-coHochschild complex $\widehat{\Lambda}C$ of a dg coalgebra $C$
is the tensor product
$C\otimes  \widehat{\Omega}C$ with the differential defined with the
same formula as in the coHochschild complex. An element $w\otimes
\varpi\in \widehat{\Lambda}C $ is denoted by $w]\varpi.$

\subsection{Hat-Hirsch  coalgebra}

A dg coalgebra $(C,d,\Delta)$ is \textbf{hat-Hirsch coalgebra}
  if there are  cooperations
$E^{p,q}: C\rightarrow C^{\otimes p}\otimes C^{\otimes q},\,p,q\geq
0,\,p+q\geq 1,$
of degree $p+q-1$  such that
\begin{itemize}
  \item  $E^{1,0}=E^{0,1}=Id: C\rightarrow C;$
  \item   $E^{p,q}(C_0)=0$ for all $p,q;$

  \item  $E^{p,q}$ extends to a linear map $E^{p,q}: C[1]\rightarrow
      C[1]^{\otimes p}\otimes C[1]^{\otimes q};$
     \item $E^{p,q}$ extends multiplicatively  to   the chain map
         $\Delta_E:  \widehat{\Omega}C\rightarrow
         \widehat{\Omega}C\otimes  \widehat{\Omega} C .  $
       \end{itemize}
 In particular, $(\widehat{\Omega} C,d_{\widehat{\Omega}C}\,, \cdot\,,
 \Delta_E )$ is a dg bialgebra. A hat-Hircsh coalgebra   $C$ is \textbf{trivial}
 if $E^{p,q}=0$ unless $(p,q) =(0,1)$  and $(p,q)=(1,0).$

 A motivated example is
 $C_\ast=C_\ast(X)$ as in Theorem \ref{hat-coHoch} (cf.
 \cite{KStwisted}) where the cooperatons  $E^{p,q}(u)$  for  $u\in X_n $ are defined by the diagonal components of $\Delta_P(\bar u)$
 in
 $\widehat{\mathbf{\Omega}}_{*,p}X \times \widehat{\mathbf{\Omega}}_{*,q}X.$
 When $E^{p,q}=0$ either for $p>1$ or $q>1,$ one obtains a (co)Gerstenhaber structure on $C$ a main example of which is $C=C_\ast(Y),$ the simplicial chain complex  of a simplicial set $Y,$ for which  the "geometric" diagonal on
 (co)Hochschild complex of $C_\ast(Y;\Bbbk)$  with any coefficients $\Bbbk$
 (i.e., inducing the standard coproduct on the free loop  homology $H_\ast(\Lambda|Y|;\Bbbk)$) is constructed  using an explicit diagonal of freehedra  in \cite{saneFREE}.

Similarly here we first construct the coproduct on the (hat)-Hochschild chain complex $\widehat{\Lambda}C_\ast(X)$ this times using an explicit diagonal of permutahedra   given by  (\ref{Lambdadiag}). Then
 the analysis  of the diagonal  on
$\widehat{\Lambda}C_\ast(X)$ (cf. Example \ref{coproduct})  leads to  formula (\ref{loopco}) of the coproduct on the coHochschild complex $\widehat{\Lambda}C$  for any  hat-Hirsch coalgebra $C.$
Namely, given
$u] \varpi \in \widehat{\Lambda}C,$
let
\[
\begin{array}{rll}
\Delta(u)&= &\sum u'\otimes u'',\vspace{1mm}\\

E^{p,q}(u')&=&\sum a_1\otimes\cdots \otimes a_p
\otimes b_1\otimes\cdots \otimes b_q \in C^{\otimes p}\otimes
C^{\otimes q},\vspace{1mm}\\
E^{s,t}(u'')&= &\sum c_1\otimes\cdots \otimes c_s\otimes
d_1\otimes\cdots \otimes d_t\in C^{\otimes s}\otimes C^{\otimes t},\vspace{1mm} \\
\Delta_E(\varpi )&=&\sum \varpi'\otimes \varpi'',\\
\end{array}
\]
and then define the coproduct
\[  \Delta_\Lambda:
\widehat{\Lambda}C\rightarrow  \widehat{\Lambda}C\otimes
\widehat{\Lambda} C \]
 by
\begin{multline}\label{loopco}
\Delta_{\Lambda}(u]\varpi)  \!=\!\!
\sum a_p]\,\bar c_1\cdots \bar c_s\cdot {\varpi'}\cdot
 \bar a_1\cdots \bar a_{p-1}\, \otimes\,
d_1]\,\bar d_2\cdots \bar d_t\cdot {\varpi''}\cdot
 \bar b_1\cdots \bar b_{q}.
 \end{multline}

 In other words, given a hat-Hirsch  coalgebra $(C,d_C,\Delta,
\{E^{p,q}\}),$ the coproducts
$\Delta:C\rightarrow C\otimes C$     and
$\Delta_E: \widehat{\Omega}C\rightarrow  \widehat{\Omega}C\otimes
\widehat{\Omega} C $    canonically determine
the (twisted) coproduct      $\Delta_\Lambda:
\widehat{\Lambda}C\rightarrow  \widehat{\Lambda}C\otimes
\widehat{\Lambda} C $ above.

In the following example we show how the structural cooperations $E^{p,q}$ on $C^\ast(X)$  are incorporated 
in the permutahedral coproduct $\Delta_{\Lambda}$  given by (\ref{loopco}).

\begin{example}\label{coproduct}

Let $C_\ast=C_\ast(X)$ as in Theorem \ref{hat-coHoch}.
  Let  $u\in C_7$  and  $\varpi\in \widehat{\Omega}_rC,\,r\geq 0,$ and 
  $u] \varpi\in
  \widehat{\Lambda}_{7+r}C$ be an elementary chain.

(i) Let $\sigma=3|2|1|5|6|4|8|7$  be a vertex of $P_8,$ and let
\[ A_\sigma\otimes B_\sigma
=123|5|46|78\otimes 3|2|156|48|7\]
 be the corresponding SCP in $\Delta_P(P_8).$  Taking into account  bijection (\ref{fbi})
  remove the integer $1$ and shift down by $1$ the blocks of the partitions to  obtain a pair
 \[12|4|35|67\otimes 2|1|45|37|6\]
 that  is  identified with  a component of
 \[\Delta_{P}(12|34567)=\Delta_P(12)\otimes \Delta_P (34567),\]
 where $12|34567\subset P_7 $ corresponds to the component $u'\otimes u''= d^0_{(34567)}(u)\otimes d^1_{(12)}(u) $ of $\Delta(u).$
Then  $A_\sigma\otimes B_\sigma $  corresponds to   a component of $\Delta_\Lambda$
\[ a_1]\,c_1c_2c_3\,\varpi'  \otimes\, d_1]\,d_2d_3\, \varpi''b_1b_2
      \]
      in which
\begin{itemize}

\item $a_1\otimes ( b_1\otimes b_2)$ is determined by the component
$12\otimes 2\,|1\in \Delta_P(12);$

\item $(c_1\otimes c_2\otimes c_3)\,\otimes\, (d_1\otimes d_2\otimes d_ 3)$ is determined by the component
$ 4|35|67\otimes  45|37|6 \in \Delta_P(34567),$
 and

\item  $\varpi'\otimes \varpi''\in \Delta_P(\varpi).$
\end{itemize}

(ii)  Let $\sigma=4|3|5|1|2|7|6|8$  be a vertex of $P_8,$ and let
\[A_\sigma\otimes B_\sigma
=34|15|2|67|8\otimes 4|35|127|68\]
 be the corresponding SCP in $\Delta_P(P_8).$ As above  taking into account   bijection (\ref{fbi})   remove the integer $1$ and  shift down by $1$ the blocks of the partitions to  obtain a pair
 \[23|4|1|56|7\otimes 3|24|16|57\]
 that
 is  identified with a component of
 \[\Delta_{P}(234|1567)=\Delta_P(234)\otimes \Delta_P (1567),\]
 where $234|1567\subset P_7 $ corresponds to the component $u'\otimes u''= d^0_{(1567)}(u)\otimes d^1_{(234)}(u) $ of $\Delta(u).$
Then  $A_\sigma\otimes B_\sigma $  corresponds to   a component of
 $\Delta_\Lambda$
\[ a_2]\,c_1c_2c_3\,\varpi'a_1  \otimes d_1]\,d_2\, \varpi''b_1b_2
      \]
in which
\begin{itemize}

\item $(a_1\otimes a_2)\otimes (b_1\otimes b_2)$ is determined by the component
$23|4\otimes 3|24\in \Delta_P(234);$

\item $(c_1\otimes c_2\otimes c_3)\,\otimes\, (d_1\otimes d_2)$ is determined by the component
$ 1|56|7\otimes  16| 57 \in \Delta_P(1567),$
and

\item  $\varpi'\otimes \varpi''\in \Delta_P(\varpi).$
\end{itemize}

\end{example}

\subsection{Intersection  bialgebra}

Let $(C,d_C,\Delta)$ be
a dg $n$ -- dimensional coalgebra (i.e.,  $C_i=0$  for $i>n$) endowed 
with the product  $m:C\otimes C\rightarrow C$ of degree $-n,$ too.
Consider a pairing  in the hat-coHochschild complex   $\widehat{\Lambda} C$
of degree $-n$ 
\[ \Cap:   \widehat{\Lambda} C \otimes \widehat{\Lambda} C \rightarrow
\widehat{\Lambda} C \]
  defined for
 $(\alpha,\beta)\in  \widehat{\Lambda} C \otimes \widehat{\Lambda} C$
by  (cf. (\ref{capformula}))
\begin{equation}\label{Capformula}
  \alpha\Cap \beta=\left\{
  \begin{array}{llll}
 (-1)^{|a||v|}\, m(u, v)\,]\, ab, &     \alpha=u]a\notin   \operatorname{Im}(\theta_2),\\
   &                      \beta=v]b \notin   \operatorname{Im}(\theta_1),
     \vspace{1mm}\\
    \theta_1(u]a \Cap \beta)-  \theta_1(u]a) \Cap \beta,& \alpha=\theta_2(u]a) ,\vspace{1mm}\\
  \theta_2(\alpha\Cap v]b)-  \alpha \Cap \theta_2(v]b),&
   \beta= \theta_1(v]b),\vspace{1mm}\\
   0, & \text{otherwise}.
\end{array}
 \right.
\end{equation}
Then $\Cap$ satisfies the Leibnitz rule.
In the case  $C=C_\ast(X)$ we have that
\[\theta_1(u]a \Cap \beta)-  \theta_1(u]a) \Cap \beta=0\ \ \text{and}\ \
  \theta_2(u]\,a)\Cap \beta=0 \ \  \text{for any}\ \  u]a \in  \widehat{\Lambda} C .\]
Note that if for $C$ there are equalities (compare Proposition \ref{serrediagonal})
\begin{equation}\label{comodule}
  \Delta\circ m= (m\otimes 1)(1\otimes T)(\Delta\otimes 1)
= (1\otimes m)(T\otimes 1)(1\otimes \Delta)
\end{equation}
and \begin{equation}\label{zero}
  \theta_2(\alpha)\Cap \beta +  \alpha\Cap \theta_1(\beta)=0,
\end{equation}
then the $\Cap$ -- product   defined by the formula 
\begin{equation}\label{strictCap}
 u]a\Cap v]b  =   (-1)^{|a||v|}\, m(u,v)\,]\,ab
 \end{equation}
   is chain.
Let  $(C,d_C,\Delta,m)$ be a dg $n$ -- dimensional module with  the coproduct $\Delta$  and  the product $m$ of degree $-n.$

(i) $C$ is \textbf{an  intersection bialgebra} if $\widehat{\Lambda} C$
admits the $\Cap$ -- product defined by (\ref{Capformula});

(ii)  $C$ is \textbf{a strict  intersection bialgebra} if $\widehat{\Lambda} C$
admits the $\Cap$ -- product defined by the formula given by (\ref{strictCap}) and satisfying 
 (\ref{comodule})--(\ref{zero}) (see Examples \ref{strictexampl1} and \ref{strictexampl2}  below).

\subsection{The twisted tensor product  $\Lambda C \otimes _\tau
 \Omega C $ }\label{twistedsec}

 Let $(C,d_C , \Delta)$ be a dg coalgebra, and
 consider the tensor product of chain complexes
  $\Lambda C\otimes \Omega C .$
 Using Sweedler's notations $\Delta(u)=u_{(1)}\otimes u_{(2)}$ and $\Delta(u)=u'\otimes u''$ as well
 define a map $\Theta=\Theta_1+\Theta_2,$
 \[ \Theta_i: \Lambda C\otimes \Omega C\rightarrow \Lambda C\otimes \Omega C ,\ \ i=1,2\ \   \text{for} \ \   u]a\otimes b \in  \Lambda C\otimes \Omega C  \ \ \text{by}  \]
 \[\Theta_1(u]a\otimes b)= u_{(1)}\,] (\bar u_{(2)})^{'}\!a \otimes  (\bar u_{(2)})^{''} b\ \
  \text{and}\ \   \Theta_2(u]a\otimes b)=  u_{(2)}\,]\, a (\bar u_{(1)})^{'}\!\otimes\,  b \,(\bar u_{(1)})^{''}\!\!.  \]
   Then the tensor product of modules  $ \Lambda C\otimes
 \Omega C$   with the differential
  $d_\Theta:= d_{\Lambda} \otimes 1 +1\otimes d_{\Omega}+\Theta   $
  is denoted by
  $\Lambda C \otimes _\tau \Omega C. $ The equality  $d^2_\Theta=0$
 uses the fact that  $\Delta: C\rightarrow  C\otimes C$ is chain; in particular, $\Theta\circ \Theta=0.$ Also note that $\theta_i\otimes 1$ is a summand component of $\Theta_i$ for $i=1,2;$ since for $\alpha
 =u]a\otimes b\in \Lambda C \otimes_\tau \Omega C,$  $\Theta(\alpha)$ contains $1]\bar ua\otimes b$ and
 $1]a\bar u\otimes b$  as summand components, $\alpha$
 is uniquely resolved from the equality $\Theta_i(\alpha)= \beta .$
Define the following chain maps
\[  \nu_r: \Lambda C\rightarrow \Lambda C\otimes_\tau \Omega C \ \  \text{and}
\ \
\nu_l: \Lambda C\rightarrow \Omega C\otimes_\tau \Lambda C,
\]
\[  \mu_r: \Lambda C \otimes \Omega C \rightarrow \Lambda C  \ \  \text{and}
\ \
\mu_l:    \Omega C\otimes_\tau  \Lambda C   \rightarrow  \Lambda C
\]
and
\[ \mathcal{T}_r: \Lambda C \otimes \Lambda C \otimes_\tau \Omega C \rightarrow \Lambda C \otimes \Lambda C \otimes_\tau\Omega C
\ \  \text{and}\ \
 \mathcal{T}_l: \Omega C \otimes_\tau \Lambda C \otimes \Lambda C \rightarrow \Omega C \otimes _\tau \Lambda C \otimes\Lambda C
 \]
by the formulas given by (\ref{nu}),  (\ref{mur}) -- (\ref{mul})
and (\ref{tr}) -- (\ref{tl}), respectively.

\subsection{Loop bialgebra}\label{loopbialgebra}

Let $C:=(C,d_C,\Delta, m)$ be
a (strict) intersection bialgebra such that it  is  a hat-Hirsch coalgebra
$(C,d_C, \{E^{p,q}\}),$ too.

(i) An intersection bialgebra $C$ is \textbf {loop bialgebra} if
 the following chain   homotopies  hold for     $\widehat{\Lambda} C:$

  \begin{equation}\label{lb1}
  \Delta_\Lambda\circ \Cap\simeq
       (\Cap\otimes \mu_r)\circ (1\otimes \mathcal{T}_r)  \circ (\Delta_\Lambda\otimes \nu_r)
  \end{equation}
and
\begin{equation}\label{lb2}
   \Delta_\Lambda\circ \Cap\simeq
       (\mu_l\otimes\Cap )\circ (\mathcal{T}_l\otimes 1)  \circ (\nu_l\otimes \Delta_{\Lambda}) .
  \end{equation}

 (ii)  A strict intersection bialgebra $C$ is \textbf {strict loop bialgebra} if
 the following   equalities   hold for     $\widehat{\Lambda} C:$
        \begin{equation}\label{slb1}
 \Delta_\Lambda\circ \Cap=
       (\Cap\otimes \mu_r)\circ (1\otimes \mathcal{T}_r)  \circ (\Delta_\Lambda\otimes \nu_r)
  \end{equation}
and
\begin{equation}\label{slb2}
 \Delta_\Lambda\circ \Cap=
       (\mu_l\otimes\Cap )\circ (\mathcal{T}_l\otimes 1)  \circ (\nu_l\otimes \Delta_{\Lambda}) .
  \end{equation}

Note that $C=C_\ast(X)$ is a loop bialgebra   as it satisfies (\ref{lb1}) -- (\ref{lb2}) (cf. (\ref{cbistring1}) -- (\ref{cbistring2})). For a strict loop bialgebras see Examples \ref{strictexampl1} and \ref{strictexampl2} below.

\subsection{$\Bbbk$ \!-\! Formal (co)algebras (spaces)}

A dg $\Bbbk$ -- coalgebra $(C,d_C)$  is $\Bbbk$--\textbf{formal} (or shortly, formal)  if there is a dg coalgebra $(B,d_B)$  with  zig-zag dg coalgebra maps
\[    (C,d_C)\rightarrow (B, d_B) \leftarrow  (H(C,d_C),0)  \]
inducing isomorphisms in homology. A space $X$ is $\Bbbk$--\textbf{formal} if the chain
coalgebra  $C_\ast(X;\Bbbk)$ is so. For a formal dg coalgebra $C$ we have the isomorphisms
\[  H_\ast(\Omega C)\approx H_\ast(\Omega H )\ \  \text{and}\ \  H_\ast(\Lambda C)\approx H_\ast(\Lambda H ) \ \  \text{for}\ \ H:=H_\ast(C).     \]
The  $\Cap$ -- product on $\Lambda C$ given by formula (\ref{Capformula})
for   an intersection  bialgebra  $C$
lifts to $\Lambda H$  when $C$ is formal.

Let calculate the string topology product for some  simply connected formal spaces (compare \cite{CJY}).
\begin{example}\label{strictexampl1}
  Let $n$ --  dimensional space   $X=\Sigma Y$ be a suspension on a polyhedron $Y.$
  Then $X$ is formal.  Assume that
  for any pair $x_i\otimes x_j\in H_i(X)\otimes H_j(X)$  
\[x_i\sqcap x_j=0 \ \  \text{unless}\ \ (i,j)\in  \mathcal{I}_n :=\{ (0,n),(n,0),(n,n)\}        \]
and
\[ x_0\sqcap x_n=x_n\sqcap x_0=x_0,\,  x_n\sqcap x_n=x_n \ \text{for a unique}\ x_n\in H_n(X)  \]
(e.g. $X$ is an $n$ -- sphere $S^n$).Then the coproduct
$\Delta : H_\ast\rightarrow H_\ast \otimes H_\ast$  on
$H_\ast:=H_\ast(X)$
consists only of the primitive part,  $\Delta(u)=u\otimes 1 + 1 \otimes u.$
Consequently,  the differential on $\Omega H_\ast$ is zero, and  we recover
  the Bott-Samelson isomorphism of algebras
\[  T^a( H_{>0}(Y)) \approx   H_\ast(\Omega H_\ast,0)\approx H_\ast(\Omega X)
 . \]
In fact $T^a( H_{>0}(Y))$  has the coproduct obtained  by multiplicative extension of the one on $H_\ast(Y),$ and the above isomorphism is the one of Hopf algebras,  since it is induced
by the inclusion $Y\hookrightarrow \Omega X.$ We assume that in turn $Y$ is a suspension in which case the coproduct on $H_\ast(Y)$ is trivial, too.

Furthermore,  the coHochschild differential $d_{\Lambda H}$ in $\Lambda H_\ast$ is equal to $\theta:=\theta_1+\theta_2$
with
\begin{equation*}
\begin{array}{llll}
\theta  (u]a) = - x_0]\bar ua +
 (-1)^{|\bar u ||a|} x_0]a\bar u .
\end{array}
\end{equation*}
Then $(H_\ast\,,\,\sqcap)$ with $\Cap: \Lambda H\otimes \Lambda H\rightarrow \Lambda H$ given by (\ref{strictCap})
is a strict intersection bialgebra as it satisfies (\ref{comodule})--(\ref{zero}).

It is immediate to detect generating $\theta$ -- cycles  in the coHochschild complex $\Lambda H_\ast=(H_\ast \otimes T^a (H_{>0}(Y)),\theta).$ Indeed,
let $x_i\in H_i$ be a basis element and $\bar x_i\in \Omega_{i-1} H_\ast.$
We have the following $\theta$ -- cycles:
\[  x_0],\, x_0]\bar x_i,\, x_i]\bar x_i \,(i\ \text{is odd}),\,  x_i]\bar x_i\bar x_i \,(i\ \text{is even}),  \]
and relations
$  x_i]a \, \Cap\, x_j]b =0 $ unless  $ (i,j)\in \mathcal{I}_n.$
In particular,
$x_0]\,\Cap\, x_n]\bar x_n=x_0]\bar x_n$  and
$x_0]\,\Cap\, x_n]\bar x_n \bar x_n =x_0]\bar x_n\bar x_n$
with   the relation  $\theta(x_n]\bar x_n)=- 2 x_0]\bar x_n\bar x_n. $
Consequently, denoting
$ a_0:=cls(x_0]),\,   a_i:=cls(x_0]\bar x_i)\, (i>0),\,
b_i := cls( x_i]\bar x_i), c_i:=cls(x_i]\bar x_i\bar x_i) , $
obtain
\[
 H_\ast(\Lambda X;\Bbbk)=
 \!\left\{
\begin{array}{lll}
\!\!\Bbbk[a_i,b_i,c_i]/
\left(a_ia_j,\,a_ib_l,\,a_ic_l,l<n,\, b_sb_t,\,b_sc_t,c_sc_t, (s,t)\notin \mathcal{I}_n \right),\\
\hspace{3.3in} n\  \text{is odd},
\vspace{1mm}\\
\!\!\Bbbk[a_i,b_i, c_i]/
\left(a_ia_j,\,a_ib_l,\, a_ic_l,l<n,\, b_sb_t,\,  b_sc_t,\,  c_sc_t,
 (s,t)\notin \mathcal{I}_n,\right.\\
 \left.
 \hspace{2.7in} 2a_0c_n \right),\
   n \  \text{is even}.
\end{array}
\right.
\]
 Since  $\Delta_{\Omega H}$ is obtained by the multiplicative extension
 of the primitive $\Delta_{H_\ast(Y)},$ the coalgebra $ H_\ast$ can be viewed as a trivial Hirsch coalgebra, hence
   $\Delta_{\Lambda H}=\Delta_H\otimes \Delta_{\Omega H}.$
   Consequently,  for the intersection bialgebra $H_\ast$ the equalities given by
  (\ref{slb1}) -- (\ref{slb2})
 hold,  and  $H_\ast$ is a  strict  loop bialgebra.

\end{example}
\begin{example}\label{strictexampl2}
  Let $X$  be  an $n$ --  dimensional space with $n$ even such that the cohomology $H^\ast(X)$ is a tensor product
   $\underset{1\leq i\leq r}{\bigotimes}\Bbbk[y_i]/y_i^{n_i+1}$
   of truncated polynomial algebras  with even dimensional generators $y_i\in H^{ev}(X).$ Thus, $n=|y_1|n_1+\cdots + |y_r|n_r.$ Assume that $X$ is formal (e.g.,  $r=1$).
   Assume also that $X$ is a Poincar\'e
duality  space with the isomorphism
\[   \phi: H_i(X)\xrightarrow{\approx} H^{n-i}(X)\ \ \text{and}\ \ \phi(u\sqcap v)=\phi(u)\phi(v)              \]
  (e.g. $X$  is the Cartesian product of  projective spaces).
  Then $( H_\ast(X)\,,\,\sqcap)$ with $\Cap: \Lambda H\otimes \Lambda H\rightarrow \Lambda H$ given by (\ref{strictCap})
is a strict intersection bialgebra as it satisfies (\ref{comodule})--(\ref{zero}).
   Denote the dual elements    of
    $y_1^{s_1}\otimes \cdots \otimes y_r^{s_r} \in H^{\ast}(X)  $ by $x_{s_1...s_r}\in H_{\ast}(X)$ for $0\leq \eta_\leq n_i.$

  For each $1\leq i\leq r,$ there is a $d_\Omega$-cycle in $\Omega H_\ast$ for $H_\ast:=H_\ast(X):$
  \[   \varpi_i:= \sum_{\substack{s_i+t_i=n_i+1\\1\leq s_i,t_i\leq n_i}} \bar x_{_0...\,s_i\,..._0}\bar x_{_0...\,t_i\,..._0}  ,
  \]
  and $\theta$ -- cycles in $\Lambda H_\ast:$
\[ \hspace{-2.8in}
  a'_i=
  x_{ n_1...n_i-1...n_r } ]\ \ \ \text{and} \]
    \[ c'_i= \sum_{\substack{  s_j+k_j+\ell_j=\epsilon_j n_j+1\\
    \epsilon_j-1\, \leq \,s_j,k_j,\ell_j\, \leq \, \epsilon_j n_j+1}}
  x_{s_1...s_r}]\bar x_{k_1...k_r} \bar x_{\ell_1...\ell_r},\ \ \ \epsilon_j=
  \begin{cases}
    2, &  j=i \\
    1, & \mbox{otherwise}.
  \end{cases}
  \]
   In particular, $c'_i$ contains a summand of the form
     $x_{n_1...n_r}]\,\varpi_i;$
    there is also a $\theta$ -- cycle
  \[
   b'= \sum_{1\leq k_i\leq n_i}
 k\, x_{n_1-k_1\,...\,n_r-k_r}]\,\bar x_{k_1...\, k_r}\ \  \text{for} \ \  k=k_1+\cdots+
k_r.    \]
The  $\sqcap$ -- product acts  in $H_\ast$ as
\[   x_{k_1...k_r} \sqcap x_{\ell_1...\ell_r}= x_{s_1...s_r}\ \ \text{with}\ \  s_i=k_i+\ell_i-n_i,\,1\leq i\leq r, \]
that in particular implies the   relations  in $\Lambda H_\ast:$
  \[\hspace{-1.6in}
  {a'_1}^{\Cap n_1} \Cap'\cdots \Cap' {a'_r}^{\Cap n_r}  =x_0]\ \   \text{and}  \]
    \[    x_0] \Cap' a'_i=0, \,\,\, \, x_0] \Cap'  b'=0, \,\,\, \,
  x_0]  \Cap'  c'_i=  x_0] \varpi_i\ \ \text{for all}\ \ i.
  \]
   Denote
   \[ \hspace{-0.3in} \tilde{b}: = \sum_{2\leq k_i\leq n_i}
     x_{n_1-k_1\,...\,n_r-k_r}]\bar x_{k_1...\, k_r}\ \   \text{and}\]
      \[\tilde c_i:= \underset {0\leq k_i<n_i} {\sum}(k_i+1)\,
       x_{_0\,...\,n_i-k_i\,..._0}]\bar x_{_0...\, k_i+1\,..._0},  \]
  so that
  \[     \theta(\tilde b)=b'\Cap b'   \        \ \ \text{and}\ \  \ \theta(\tilde c_i)= (n_i+1)\, x_0]\varpi_i=(n_i+1)\, x_0]  \Cap  c'_i. \]
  Denoting $a=cls(x_0]),\, a_i=cls(a'_i),\,\,  b=cls(b')$  and $c_i=cls(c'_i),$ obtain
  \[
 H_\ast(\Lambda X;\Bbbk)=\Bbbk[b]/(b^2)\otimes
 \Bbbk[a_i,b,c_i]/\left(aa_i,\,  ab,\,   (n_i+1)ac_i,  1\leq i\leq r\right).
\]

\end{example}
\noindent Assume that $X$ is a product of complex projective spaces $\mathbb{C}P^m.$ Since
the homotopy equivalence $\Omega \mathbb{C}P^m\simeq \Omega S^{2n+1}\times S^1,$ we have the coalgebra isomorphism $H_\ast(\Omega \mathbb{C}P^m)\approx H_\ast(\Omega S^{2n+1}\times S^1).$ Consequently, like the previous example
we can regard $H_\ast$ as a trivial Hirsch coalgebra, so that
   $\Delta_{\Lambda H}=\Delta_H\otimes \Delta_{\Omega H}.$
   Consequently,  for the intersection bialgebra $H_\ast$ the equalities given by   (\ref{slb1}) -- (\ref{slb2})  hold,  and  $H_\ast$ is a  strict  loop bialgebra.


\vspace{0.15in}


\begin{thebibliography}{99}



\bibitem{CS}
M. Chas and D. Sullivan, String topology, preprint, math.GT/9911159.


\bibitem{CJY} R. Cohen, J. D. S. Jones, and J. Yan, The loop homology algebra of spheres
and projective spaces, In Categorical decomposition techniques in algebraic topology (Isle of Skye, 2001),
  Progr. Math., Birkhäuser, Basel, 215 (2004), 77--92.

\bibitem{Friedman} G. Friedman, Singular intersection homology, Texas
    Christian University (2019).

 \bibitem {KStwisted} T. Kadeishvili and  S. Saneblidze,  The twisted
     Cartesian model for the double path fibration, Georgian Math. J.,
     22 (4) (2015), 489--508.

\bibitem{RS1} M. Rivera and S. Saneblidze,
A combinatorial model for the path fibration, J. Homology, Homotopy and
Appl., 14 (2019), 393--410.


\bibitem{RS2} M. Rivera and S. Saneblidze,
A combinatorial model for the free loop fibration, \textit{Bull.
L.M.S.},
 \textbf{50} (6) (2018), 1085--1101.


\bibitem{R-T} M. Rivera and A.Takeda, String topology via the coHochschild complex and local intersections, math.AT/ 2508.15684v2.

\bibitem{saneFREE} S. Saneblidze,  The bitwited Cartesian model for the free
loop fibration, Topology and Its Applications, 156 (2009), 897--910.


\bibitem{SU} S. Saneblidze and R. Umble, Comparing diagonals of associahedra, J. Homology, Homotopy and Appl., 26   (2024), 141--149.


\bibitem{Whitehead} G. Whitehead, Elements of Homotopy Theory,
    Springer-Verlag GTM 62 (1978).

\end{thebibliography}
\end{document}